%% file: quasitree.tex
\documentclass[12pt]{article}

\usepackage{xcolor}
\usepackage{comment}
\usepackage{xspace}
\usepackage{tikz}
\usepackage{graphicx}
\usepackage{multirow}

\usepackage{amsmath,amsthm, amssymb, mathtools}
\usepackage{fullpage}
\usepackage{float}

\newcommand{\bG}{\mathbb{G}}
\newcommand{\bH}{\mathbb{H}}

\newcommand{\bQ}{Q}

\newcommand{\V}{\mathcal{V}}
\newcommand{\B}{\mathcal{B}}
\newcommand{\pG}{\underline{\mathbb{G}}}
\newcommand{\pH}{\underline{\mathbb{H}}}

\newcommand{\al}{\alpha}
\newcommand{\be}{\beta}
\newcommand{\ga}{\gamma}
\newcommand{\Q}{\mathcal{Q}}
\newcommand{\T}{\boldsymbol{T}}
\newcommand{\Tt}{\mathcal{T}}

\newcommand{\wV}{\omega_{\mathcal{V}}}
\newcommand{\wB}{\omega_{\mathcal{B}}}
\newcommand{\wVdual}{\omega_{\mathcal{V}^*}}
\newcommand{\wBdual}{\omega_{\mathcal{B}^*}}
\DeclareMathOperator{\ba}{\backslash}
\DeclareMathOperator{\con}{/}

\newtheorem{theorem}{Theorem}[section]

\newtheorem{lemma}[theorem]{Lemma}
\newtheorem{proposition}[theorem]{Proposition}
\theoremstyle{definition}
\newtheorem{definition}[theorem]{Definition}
\newtheorem{example}[theorem]{Example}
\newtheorem{claim}[theorem]{Claim}
\theoremstyle{remark}
\newtheorem*{remark}{Remark}

\title{A quasi-tree expansion for the surface Tutte polynomial}
\author{Maya Thompson \footnote{mayathompson.math@gmail.com.}}
\date{}

\begin{document}

\maketitle

\abstract{The surface Tutte polynomial has recently been generalised to pseudo-surfaces equipping it with recursive deletion-contraction relations~\cite{maya1}. We use these relations to show that this generalisation naturally possesses a quasi-tree expansion. This extends quasi-tree expansions of the Bollob\'as--Riordan, Las~Vergnas and Krushkal polynomials~\cite{Butler, Champan, V-T}, which we recover from our main result.}

\section{Introduction}
The Tutte polynomial is one of the most influential graph invariants due to its far-reaching applications across many areas of mathematics and the depth of its underlying theory. It specialises to numerous well-known invariants, such as the Jones polynomial of an alternating knot, the partition function of the Potts model in statistical mechanics, the weight enumerator of linear codes, and the chromatic polynomial. A detailed and extensive background on the Tutte polynomial is given in~\cite{Handbook}.

The Tutte polynomial can famously be defined in three equivalent ways, as a state sum over all subgraphs, recursively through deletion and contraction, and as a sum over all spanning trees.  
Tutte first introduced his polynomial in~\cite{TuttePoly} via the spanning tree expansion, using the notion of activity, which we now describe.
Let $G=(V,E)$ be a graph and fix a total ordering $\prec$ on the edges. Given a maximal spanning forest $F$ of $G$, an edge $e\in E(G)-E(F)$ is \emph{externally active} if it is the smallest edge in the unique cycle of $F\cup e$. An edge $e\in E(F)$ is \emph{internally active} if it is the smallest edge in the \emph{cut defined by $e$}. That is, in the set $U_F(e)=\{f\in E(G) : (F\ba e)\cup f  \> \text{ is a maximal spanning forest}\}$. Using activity, the Tutte polynomial can be defined by its \emph{spanning tree expansion} as
\begin{equation}\label{AETutte}T(G;x,y)=\sum_{i,j}t_{ij}x^{i}y^{j},\end{equation}
where $t_{ij}$ is the number of maximal spanning forests of $G$ with internal activity $i$ and external activity $j$.

We are interested in analogues of the Tutte polynomial for graphs cellularly embedded in surfaces, or equivalently ribbon graphs. There has been significant interest in such \emph{topological Tutte polynomials} fueled by Bollob\'as and Riordan's papers~\cite{BR1, BR2}, with several variations arising~\cite{BR1, BR2, maps1, maps2,  HM,  KMT, Krushkal, LV, maya1} over the years, many of which admit expansions akin to the spanning tree expansion~\eqref{AETutte} of the Tutte polynomial. Bollob\'as and Riordan used Tutte's notion of activities to provide a spanning tree expansion for their generalisation of the Tutte polynomial~\cite{BR1}. However, for ribbon graphs, the more natural analogue of a spanning tree is a quasi-tree, i.e., a spanning ribbon subgraph with one boundary component. This perspective is demonstrated by Champanerkar, Kofman, and Stolzfus who extend the concept of activity to quasi-trees and give a quasi-tree expansion for the Bollob\'as--Riordan polynomial for orientable ribbon graphs~\cite{Champan}. This expansion was generalised to non-orientable ribbon graphs by Vignes--Tourneret~\cite{V-T}. Butler~\cite{Butler} then provided a more general quasi-tree expansion for the Krushkal polynomial, from which a quasi-tree expansion can be derived for the Las~Vergnas polynomial and rederived for the Bollob\'as--Riordan polynomial. 

Notably, a key feature of these quasi-tree expansions is that they implicitly rely on structural properties of \emph{graphs embedded in pseudo-surfaces}, a setting formalised more recently. (Informally, a pseudo-surface is obtained from a surface by taking the topological quotient of finitely many paths in the surface.) This setting proves to be the correct one for both deletion-contraction and activity-based expansions.

One topological Tutte polynomial of particular interest is the \emph{surface Tutte polynomial}, introduced by Goodall et al~\cite{maps1,maps2}, which emulated Tutte's construction of the Tutte polynomial as the dichromate.  
As a result of their approach, the surface Tutte polynomial has been shown to have an array of applications, including enumerating the number of colourings, flows and acyclic orientations of ribbon graphs. However, unlike the classical Tutte polynomial, it can only be defined as a state sum and does not admit equivalent definitions via deletion-contraction or activity. These limitations arise as the polynomial is confined to graphs cellularly embedded in surfaces. 

To address the lack of deletion-contraction relations, the \emph{packaged surface Tutte polynomial} was introduced~\cite{maya1}. This polynomial generalises the surface Tutte polynomial to graphs embedded in pseudo-surfaces, which we realise as \emph{packaged ribbon graphs} (defined in Section~\ref{ss.pack}). It unifies various approaches from the literature, inheriting the key properties of its predecessors and specialising to each of them. Notably, it can be defined both as a state sum and recursively through deletion-contraction relations that hold for all edge types.

In this paper, we provide a quasi-tree expansion for the packaged surface Tutte polynomial, thereby completing the trio of formulations that mirror those of the Tutte polynomial. Working in the pseudo-surface setting is essential, as it underpins both the deletion-contraction relations and constructions needed for a quasi-tree expansion. Our approach uses the deletion-contraction relations of the packaged surface Tutte polynomial~\cite{maya1} with previously established notions of activity for quasi-trees~\cite{Champan, V-T}. Our result generalises the quasi-tree expansions found in the literature, and we recover them accordingly.

This paper is structured as follows. In Section~2, we provide the necessary background on ribbon graphs, quasi-trees and a notion of activity for them, and packaged ribbon graphs, which describe graphs embedded in pseudo-surfaces. Section~3 features our main result, Theorem~\ref{mainthm}, a quasi-tree expansion for the packaged surface Tutte polynomial, from which we recover the expansion for the Krushkal polynomial~\cite{Butler}, and hence those for the Bollob\'as--Riordan and Las~Vergnas polynomials.

\section{Preliminaries}

\subsection{Ribbon graphs}
We begin with a brief overview of standard ribbon graph terminology and notation.  Further details on ribbon graphs can be found in \cite{graphsonsurfaces}. A reader familiar with them may skip this subsection.

\begin{definition}
    A \emph{ribbon graph} $\bG=(V,E)$ is a surface with boundary represented by the union of a set of discs $V$, called the vertices, and another set of discs $E$, called the edges, satisfying the following conditions:
   \begin{enumerate}
        \item A vertex and edge are either disjoint or meet in an arc. (An edge is \emph{adjacent} to a vertex if they meet in an arc.)
        \item Each such arc lies on the boundary of just one vertex and one edge.
        \item Every edge contains two such arcs.
    \end{enumerate}
\end{definition}

Given a ribbon graph $\bG$, we use $v(\bG)$, $e(\bG)$, $b(\bG)$ and $k(\bG)$ to denote the number of vertices, edges, boundary components and connected components respectively. When needed we denote its vertex set by $V(\bG)$, edge set by $E(\bG)$ and set of boundary components by $B(\bG)$. Where it makes sense, we  use this notation for graphs too. The \emph{Euler genus} of a ribbon graph is
\[\ga(\bG)=2k(\bG)-v(\bG)+e(\bG)-b(\bG).\]

Let $\bG=(V,E)$ be a ribbon graph and $e\in E$ be an edge. The ribbon graph obtained by \emph{deleting} $e$ is denoted by $\bG\ba e:=(V,E- e)$ (we use the convention of denoting singleton sets by the element they contain). When deleting edges one after another the order in which we do so does not matter, and likewise when deleting vertices. This allows us to extend the operation of deletion to subsets of edges and subsets of vertices.
We say that a ribbon graph $\bH$ is a \emph{ribbon subgraph} of a ribbon graph $\bG$ if it can be obtained from $\bG$ by deleting a subset of edges (possibly empty) and then deleting a subset of isolated vertices (possibly empty). If $\bH$ has the same vertex set as $\bG$, then we say $\bH$ is a \emph{spanning ribbon subgraph} of $\bG$ and we may denote it by $\bG|A$ where $A$ is the edge set of $\bH$.

\begin{definition} 
 Let $\bG=(V,E)$ be a ribbon graph. Let $e\in E$ and  $u$ and $v$ be its adjacent vertices, which are not necessarily distinct. Then  $\bG/e$ denotes the ribbon graph obtained as follows. Consider the boundary component(s) of $e\cup u \cup v$ as curves on $\bG$. For each resulting curve, attach a disc, which will form a vertex of $\bG/e$, by identifying its boundary component with the curve. Delete $e$, $u$ and $v$ from the resulting complex.
We say that $\bG/e$ is obtained from $\bG$ by \emph{contracting} $e$.
\end{definition}

When contracting edges one after another the order in which we do so does not matter, so we may also extend the operation of contraction to subsets of edges. 
For a depiction of ribbon graph contraction see Table~\ref{Con} and disregard the colouring and weighting, and similarly for ribbon graph deletion and Table~\ref{Del}.

An edge is a \emph{bridge} if deleting it increases the number of connected components in the ribbon graph.  An edge is a \emph{loop} if it is adjacent to precisely one vertex. A loop is \emph{non-orientable} if the union of the loop with its adjacent vertex is homeomorphic to the M\"obius band, otherwise it is \emph{orientable}.  A loop is \emph{plane} if contracting it increases the number of connected components in the ribbon graph. (Note that plane loops are always orientable.)
  Two loops $e$ and $f$ are \emph{interlaced} if they are adjacent to the same vertex and their ends are met in the cyclic order $e$ $f$ $e$ $f$ when traveling around the boundary of that vertex.

\medskip

The \emph{(geometric) dual} $\bG^*$ of a ribbon graph $\bG=(V,E)$ can be constructed in the following way. 
Recall that, topologically, a ribbon graph is a surface with boundary. We cap off the holes using a set of discs, denoted by $V'$, to obtain a surface without boundary. The geometric dual of $\bG$ is the ribbon graph $\bG^* = (V',E)$. Observe that the edges of $\bG$ and $\bG^*$ are identical. The only change is which arcs on their boundaries meet vertices.

First introduced by Chmutov in \cite{ChmutovPD}, \emph{partial duality} is a local involution on the edges of a ribbon graph that when applied to the entire edge set results in the geometric dual. 
\begin{definition}
Let $\bG=(V,E)$ be a ribbon graph and $A\subseteq E$, then the \emph{partial dual}  $\bG^{A}$ is the ribbon graph constructed as follows. The spanning ribbon subgraph $\bG|A$ is a surface with boundary. 
 Cap off the holes of $\bG|A$ using a set of discs, denoted by $V'$, to obtain a surface without boundary. Then delete the interior of the discs in $V$. The resulting surface with boundary is the ribbon graph $\bG^{A}$ with vertex set $V'$ and edge set $E$. Observe that the edges of $\bG$ and $\bG^{A}$ are identical. The only potential change is which arcs on their boundaries meet vertices.  
\end{definition}

We list some known properties of partial duality that we will utilise later.

\begin{proposition}[\cite{ChmutovPD}]\label{p.pd}
    Let $\bG$ be a ribbon graph and $X,Y\subseteq E(\bG)$. Then
    \begin{enumerate}
        \item $(\bG^X)^Y=(\bG^Y)^X$,
        \item $\bG \con X = \bG^{X}\ba X$, and
        \item partial duality commutes with deletion and contraction.
    \end{enumerate}
\end{proposition}

\begin{table}[!t]
    \centering
    \begin{tabular}{| c | c |}
    \hline
       $\bG$  & $\bG^e$  \\
       \hline
       \input{Diagrams/Pd/nonloop} & \input{Diagrams/Pd/loop} \\ 
       \hline
       \input{Diagrams/Pd/loop} & \input{Diagrams/Pd/nonloop} \\
       \hline
       \input{Diagrams/Pd/nonori1}  & \input{Diagrams/Pd/nonori2}\\\hline
    \end{tabular}
    \caption{The partial dual of an edge of a ribbon graph.}
    \label{t.pd}
\end{table}

Note that, as partial duality is commutative, it can be formed one edge at a time and Table~\ref{t.pd} depicts how to take the partial dual with respect to the different edge types. An example of a ribbon graph $\bG$ with edge $f\in E(\bG)$ and the partial dual $\bG^f$ is given in Figure~\ref{f.ex}.

\subsection{Quasi-trees}
 
\begin{definition}
    A \emph{quasi-tree} is a ribbon graph with precisely one boundary component. For a connected ribbon graph $\bG$, let $\Q_{\bG}$ denote the set of spanning ribbon subgraphs of $\bG$ that are quasi-trees.
\end{definition}

\begin{remark}
    We only consider connected ribbon graphs in this paper, however our work can be generalised to disconnected ribbon graphs using quasi-forests.
    A \emph{quasi-forest} is a ribbon graph where the number of connected components is equal to the number of boundary component. For a ribbon graph $\bG$, a spanning ribbon subgraph is a quasi-forest of $\bG$ if it has $k(\bG)$-many boundary components. When $\bG$ is connected, its quasi-forests are quasi-trees. 
\end{remark}

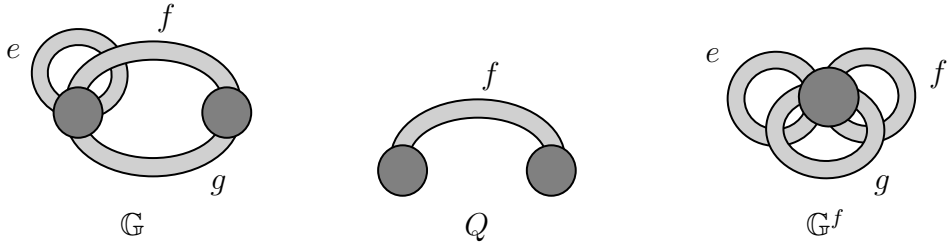
\begin{figure}
    \centering
    \begin{tabular}{ccc}
        \input{Diagrams/QTex1} \qquad &\qquad \input{Diagrams/QTex2} \qquad &\qquad \input{Diagrams/QTex3} \\
    $\bG$\qquad &\qquad $Q$ \qquad &\qquad $\bG^{f}$
    \end{tabular}
    \caption{A ribbon graph $\bG$, quasi-tree $Q\in\Q_\bG$ and partial dual $\bG^{E(Q)}$.}
    \label{f.ex}
\end{figure}

If $Q\in \Q_{\bG}$, then $\bG^{E(Q)}$ has exactly one vertex arising from the single boundary component of $Q$. For two edges $e$ and $e'$ of $\bG$, we say $e$ \emph{links} $e'$ with respect to $Q$ if they are interlaced in  $\bG^{E(Q)}$. For example, in Figure~\ref{f.ex}  the edge $g$ links both $e$ and $f$, but $e$ and $f$ are not linked with respect to the quasi-tree $Q$.

\begin{definition}
    Let $\bG$ be a connected ribbon graph with quasi-tree $Q\in \Q_{\bG}$ and let $\prec$ be a total order on the edges $E(\bG)$. With respect to $Q$, an edge $e\in E(\bG)$ is \emph{live} if it does not link a lower ordered edge, otherwise it is \emph{dead}. We say $e$ is \emph{non-orientable} with respect to $Q$ if it is a non-orientable loop in $\bG^{E(Q)}$, otherwise it is \emph{orientable}. The edges in $E(Q)$ are \emph{internal} and those in $E(\bG)-E(Q)$ are \emph{external}.
\end{definition}

Continuing our example, consider the ribbon graph $\bG$ and quasi-tree $Q\in\Q_\bG$ shown in Figure~\ref{f.ex}. If the edges are ordered $e\prec f \prec g$, then $e$ is externally live, $f$ is internally live, and $g$ is externally dead with respect to $Q$. All three edges are orientable.

This characterisation of activity for quasi-trees is due to Vignes-Tourneret~\cite{V-T} and is equivalent to the definition given in~\cite{Champan}.  For planar ribbon graphs, quasi-trees are always trees and edges are live (resp. dead) when they are active (resp. inactive) in the spanning tree. The choice of ordering $\prec$ can affect which edges are live or dead.

Observe that $b(\bG\ba A^c)=b(\bG^*\ba A)$ for any subset $A\subseteq E(\bG)$. 
So $Q\in\Q_\bG$ if and only if the spanning ribbon subgraph $\bG^*\ba E(Q)$ is a quasi-tree. This gives a natural one-to-one correspondence between $\Q_\bG$ and $\Q_{\bG^*}$. In~\cite[Lemma~4.1]{Butler}, Butler notes the following regarding duality and activity.

\begin{lemma}[\cite{Butler}]\label{l.but}
    Let $\bG$ be a connected ribbon graph with quasi-tree $Q\in\Q_\bG$ and let $Q'$ denote the quasi-tree $\bG^*\ba E(Q)$. If $e$ is live (resp. dead) with respect to $Q$, then $e$ is live (resp. dead) with respect to $Q'$. If $e$ is internal (resp. external) with respect to $Q$, then $e$ is external (resp. internal) with respect to $Q'$.
\end{lemma}

Let $\bG$ be a connected ribbon graph with quasi-tree $\bQ\in \Q_{\bG}$ and let $\prec$ be a total order on the edges $E(\bG)$. We define the following notation. 

\begin{itemize}
    \item  $D_Q$ is the set of internally dead edges;
    \item  $D_Q^*$ is the set of externally dead edges;
    \item  $O_Q$ is the set of internally live orientable edges;
    \item  $O_Q^*$ is the set of externally live orientable edges;
    \item  $N_Q$ is the set of internally live non-orientable edges; and
    \item  $N_Q^*$ is the set of externally live non-orientable edges.
\end{itemize}
Our choice of notation is motivated by the fact that, 
using Lemma~\ref{l.but}, 
$D_Q^*=D_{Q'}$, 
$O_Q^*=O_{Q'}$,
$N_Q^*=N_{Q'}$,
where $Q'=\bG^*\ba E(Q)$.

\subsection{Packaged ribbon graphs}\label{ss.pack}
Packaged ribbon graphs were introduced in~\cite{maya1} building upon the coloured ribbon graphs of Huggett and Moffatt~\cite{HM}.
They consist of a ribbon graph equipped with a partition on the set of vertices and the set of boundary components, where each block in a partition is weighted. Ribbon graphs are known for their ability to describe graphs cellularly embedded in surfaces (see, e.g., \cite{graphsonsurfaces,zbMATH01624430}). Analogously, packaged ribbon graphs describe graphs embedded in pseudo-surfaces equipped with vertex and face weights (details given in~\cite[Section~2.3]{maya1}).
\begin{definition}
   A  \emph{packaged ribbon graph} is a tuple $\pG=(\bG,\V,\wV,\B,\wB)$ where: 
    \begin{enumerate}
        \item $\bG=(V,E)$ is a ribbon graph and $B$ is the set of boundary components in $\bG$;
        \item $\V$ is a partition of $V$, and $\B$ is a partition of $B$;
        \item $\wV:\V\to \mathbb{N}_0$ is a \emph{weighting} on the blocks of $\V$, and $\wB:\B\to\mathbb{N}_0$ is a \emph{weighting} on the blocks of $\B$.
    \end{enumerate}
\end{definition}

For each vertex $v\in V$, we denote its block in the partition $\V$ by $[v]_\V$. Similarly, for each boundary component $b\in B$, we denote its block in the partition $\B$ by $[b]_\B$. At times, when there is no potential for confusion, we omit  subscripts and write $[v]$ for $[v]_{\V}$, and $[b]$ for $[b]_{\B}$.  (Note that in this paper the term ``block" will always refer to the block of a partition and never a block of a graph.)

A packaged ribbon graph consists of a ``coloured ribbon graph'' as defined by Huggett and Moffatt~\cite[Definition~11]{HM} in which the vertex colours and boundary component colours are weighted. In practice, it can be easier to visualise the partitions $\V$ and $\B$ as colourings on the vertices and boundary components respectively. In which case, the block weightings become colour weightings.
For this reason, the figures in this paper will use ribbon graphs with weighted vertex and boundary component colourings to depict packaged ribbon graphs. The weightings $\wV$ and $\wB$ will be represented by a coloured tuple where each integer corresponds to the weight function on that colour. An example can be found in Figure~\ref{Packaging}.

\medskip
Deletion and contraction are defined for packaged ribbon graphs following their definitions for coloured ribbon graphs from~\cite[Definition~13]{HM} and specifying how the weights change, as below.
 The reader may find it helpful to consult Tables~\ref{Del} and~\ref{Con} while reading the definitions.

\begin{definition}
    Let $\pG=(\bG,\V,\wV,\B,\wB)$ be a packaged ribbon graph, $B$ be its set of boundary components, and $e$ be an edge in $\bG$. We use $\pG\ba e$ to denote the packaged ribbon graph $(\bG\ba e,\V',\wV',\B',\wB')$ obtained by \emph{deleting} $e$, and define it as follows. Recalling $\bG$ and $\bG\ba e$ have the same vertices,  we set $\V':=\V$ and $\wV':=\wV$. For the boundary components, $\B'$ and $\wB'$ are defined as follows.
    \begin{enumerate}
        \item If, in $\bG$,  the edge $e$ intersects two boundary components $a,b\in B$ and $[a]\neq [b]$, then delete $e$ from $\bG$ and let $a'$ denote the boundary component formed by the deletion. Obtain the partition $\B'$ from $\B$ by removing both $[a]$ and $[b]$, and inserting $[a']:=[a]\cup[b]\cup \{a'\}-\{a,b\}$. The other blocks are unchanged. Define 
        \[\wB'([x]):=
        \begin{cases}
        \wB([a])+\wB([b]) & \text{ if } [x]=[a'],\\
        \wB([x]) & \text{ if } [x]\neq [a'].
        \end{cases}\]
        
        \item If, in $\bG$,  the edge $e$ intersects two boundary components $a,b\in B$ and $[a]=[b]$, then delete $e$ from $\bG$ and let $a'$ denote the boundary component formed by the deletion. Obtain the partition $\B'$ from $\B$ by replacing $[a]$ with $[a']:=[a]\cup\{a'\}- \{a, b\}$. The other blocks are unchanged. Define
        \[\wB'([x]):=
        \begin{cases}
        \wB([x])+1 & \text{ if } [x]=[a'],\\
        \wB([x]) & \text{ if } [x]\neq [a'].
        \end{cases}\]
       
        \item If, in $\bG$, the edge $e$ intersects one boundary component $a\in B$ twice,  and $b(\bG\ba e)=b(\bG)+1$, then delete $e$ from $\bG$ and let $a'$, $b'$ denote the two boundary components formed by the deletion. Obtain the partition $\B'$ from $\B$ by replacing $[a]$ with $[a']:=[a]\cup\{a',b'\}- \{a\}$. The other blocks are unchanged. Define 
        \[\wB'([x]):=
        \begin{cases}
        \wB([x])+1 & \text{ if } [x]=[a'],\\
        \wB([x]) & \text{ if } [x]\neq [a'].
        \end{cases}\]
        
         \item If, in $\bG$, the edge $e$ intersects one boundary component $a\in B$ twice, and $b(\bG\ba e)=b(\bG)$, then delete $e$ from $\bG$ and let $a'$ denote the boundary component formed by the deletion. Obtain the partition $\B'$ from $\B$ by replacing $[a]$ with $[a']:=[a]\cup\{a'\}- \{a\}$. The other blocks are unchanged. Define 
        \[\wB'([x]):=
        \begin{cases}
        \wB([x])+1 & \text{ if } [x]=[a'],\\
        \wB([x]) & \text{ if } [x]\neq [a'].
        \end{cases}\]
    \end{enumerate}
    
\end{definition}

\begin{table}[!t]
    \centering
    \begin{tabular}{| c | c |}
        \hline
        $\pG$ & $\pG\ba e$ \\
        \hline
          \input{Diagrams/Del/11} & \input{Diagrams/Del/12} \\
        \hline
        \input{Diagrams/Del/21} & \input{Diagrams/Del/22.tex} \\
        \hline
        \input{Diagrams/Del/31.tex} & \input{Diagrams/Del/32.tex} \\
        \hline
        \input{Diagrams/Del/41.tex} & \input{Diagrams/Del/42.tex} \\
        \hline
    \end{tabular}
    \caption{Deleting an edge $e$ in a packaged ribbon graph $\pG$.}
    \label{Del}
\end{table}

\begin{definition}
    Let $\pG=(\bG,\V,\wV,\B,\wB)$ be a packaged ribbon graph, $V$ be its vertex set, and $e$ be an edge in $\bG$. We use $\pG/e$ to denote the packaged ribbon graph $(\bG/e,\V',\wV',\B',\wB')$ obtained by \emph{contracting} $e$ and define it as follows. There is a natural correspondence between the boundary components of $\bG$ and those of $\bG/e$. This correspondence induces a partition $\B'$ on the boundary components of $\bG/e$ and a weighting $\wB'$. For the vertices, $\V'$ and $\wV'$ are defined as follows.
    \begin{enumerate} 
        
        \item If $e$ is a non-loop edge incident to $u,v\in V$ and $[u]\neq [v]$, then contract $e$ in $\bG$ and let $u'$ denote the vertex formed by the contraction. Obtain the partition $\V'$ from $\V$ by removing $[u]$ and $[v]$, and inserting $[u']:=[u]\cup [v]\cup\{u'\}-\{u,v\}$. The other blocks are unchanged. Define
        \[\wV'([x]):=
        \begin{cases}
        \wV([u])+\wV([v]) & \text{ if } [x]=[u'],\\
        \wV([x]) & \text{ if } [x]\neq [u'].
        \end{cases}\]

        \item If $e$ is a non-loop edge incident to $u,v\in V$ and $[u]=[v]$, then contract $e$ in $\bG$ and let $u'$ denote the vertex formed by the contraction. Obtain the partition $\V'$ from $\V$ by replacing $[u]$ with $[u']:=[u]\cup \{u'\}- \{u,v\}$. The other blocks are unchanged. Define
        \[\wV'([x]):=
        \begin{cases}
        \wV([x])+1 & \text{ if } [x]=[u'],\\
        \wV([x]) & \text{ if } [x]\neq [u'].
        \end{cases}\]

        \item If $e$ is an orientable loop incident to $u\in V$,         
        then contract $e$ in $\bG$ and let $u'$, $v'$ denote the two vertices formed by the contraction. Obtain the partition $\V'$ from $\V$ by replacing $[u]$ with $[u']:=[u]\cup \{u',v'\}- \{u\}$. The other blocks are unchanged. Define
        \[\wV'([x]):=
        \begin{cases}
        \wV([x])+1 & \text{ if } [x]=[u'],\\
        \wV([x]) & \text{ if } [x]\neq [u'].
        \end{cases}\]

         \item If $e$ is a non-orientable loop incident to $u\in V$, then contract $e$ in $\bG$ and let $u'$ denote the vertex formed by the contraction.  Obtain the partition $\V'$ from $\V$ by replacing $[u]$ with $[u']:=[u]\cup \{u'\}- \{u\}$. The other blocks are unchanged. Define
        \[\wV'([x]):=
        \begin{cases}
        \wV([x])+1 & \text{ if } [x]=[u'],\\
        \wV([x]) & \text{ if } [x]\neq [u'].
        \end{cases}\]
    \end{enumerate}
    
\end{definition}

\begin{table}[!t]
    \centering
    \begin{tabular}{| c | c |}
        \hline
        $\pG$ & $\pG/e$ \\
        \hline
          \input{Diagrams/Con/11} & \input{Diagrams/Con/12} \\
        \hline
        \input{Diagrams/Con/21} & \input{Diagrams/Con/22.tex} \\
        \hline
        \input{Diagrams/Con/31.tex} & \input{Diagrams/Con/32.tex} \\
        \hline
        \input{Diagrams/Con/41.tex} & \input{Diagrams/Con/42.tex} \\
        \hline
    \end{tabular}
    \caption{Contracting an edge $e$ in a packaged ribbon graph $\pG$.}
    \label{Con}
\end{table}

We can also extend the notion of duality to packaged ribbon graphs. The \emph{dual} $\pG^*=(\bG^*,\V^*,\wVdual,\B^*,\wBdual)$ of a packaged ribbon graph $\pG=(\bG,\V,\wV,\B,\wB)$ is obtained as follows. The ribbon graph $\bG^*$ is the dual of the ribbon graph $\bG$. As the vertices of $\bG$ are in one-to-one correspondence with the boundary components of $\bG^*$, the vertex partition $\V$ on $V(\bG)$ induces a boundary component partition $\B^*$ on the set of all boundary components in $\bG^*$. In a similar way, we can obtain a vertex partition $\V^*$ on $V(\bG^*)$ from the boundary component partition $\B$ on the set of all boundary components in $\bG$. The partition $\V^*$ (resp. $\B^*$) naturally inherits the weighting $\wVdual$ (resp. $\wBdual$) from the weighting $\wB$ (resp. $\wV$) on $\B$ (resp. $\V$). 
 At times we will abuse notation and simply denote the dual of $\pG=(\bG,\V,\wV,\B,\wB)$ by $\pG^*=(\bG^*,\B,\wB,\V,\wV)$. An example of a packaged ribbon graph and its dual is given in Figure~\ref{Packaging}.

 \medskip

We will make use of the following graph that is naturally associated with a packaged ribbon graph. It arises from its underlying graph by identifying all vertices that are in the  same block in the partition. 
\begin{definition} 
    Let $\pG=(\bG,\V,\wV,\B,\wB)$ be a packaged ribbon graph. Its \emph{packaging} $G(\bG;\V)$ is the vertex-weighted graph obtained in the following way. Create a vertex for each block in the partition $\V$. 
    The vertex set of $G(\bG;\V)$ is  the partition $\V$.
    There is an edge $([u], [v])$ in $G(\bG;\V)$ for each edge $(u,v)$ in $\bG$, and each edge has this form. 
   The vertices of $G(\bG;\V)$ are weighted by $\wV$.
\end{definition}
Some examples are shown in Figure~\ref{Packaging}. 
We note that there is an asymmetry in our definitions as we do not define the packaging for the boundary component partition $\B$. We could do this, but for the sake of notational simplicity we work with  $G(\bG^*;\B)$ instead. 

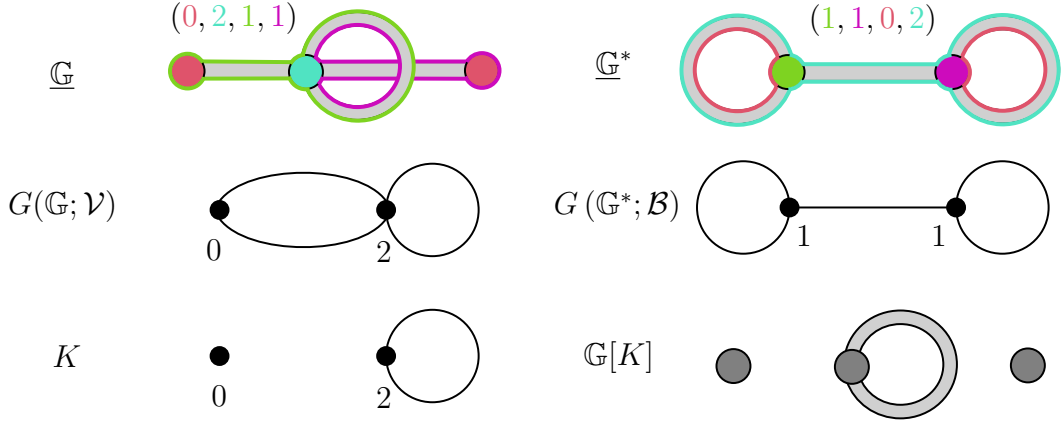
\begin{figure}
    \centering
    \input{Diagrams/Packaging}
    \caption{Packaged ribbon graph $\pG$ and its dual $\pG^*$, with packagings $G(\bG;\V)$ and $G(\bG^*;\B)$, respectively, and a subgraph $K$ of $G(\bG;\V)$ with corresponding ribbon subgraph $\bG[K]$.}
    \label{Packaging}
\end{figure}

Every edge of $G(\bG;\V)$ corresponds to a unique edge of $\bG$, and every vertex of $G(\bG;V)$ corresponds to a distinct block in $\V$. Thus, each subgraph $K$ of $G(\bG;\V)$ gives rise to a ribbon subgraph $\mathbb{K}=(U,A)$ where
$U:=\{u\in [w]\in \V : w\in V(K)\}$ 
(thus $U$ consists of all vertices in $\bG$ that are sent to vertices in $K$ when forming the packaging)
and $A$ is the set of edges in $\bG$ corresponding to the edges in $K$. We denote the ribbon subgraph $\mathbb{K}$ obtained in this way by $\bG[K]$.
Moreover, in this case we set 
\[b(\bG[K]):=b(\mathbb{K}).\]
Figure~\ref{Packaging} shows an example of $\bG[K]$ where $b(\bG[K])=4$.

\section{Quasi-tree expansions}

\subsection{Packaged surface Tutte polynomial}
The packaged surface Tutte polynomial was introduced in~\cite{maya1}, though the definition used here is from~\cite{thesis}, as it uses Euler genus instead of genus to better accommodate non-orientability. (Note that~\cite{maya1,thesis} both contain a typo in the definition of $\ga(\bG^*,H)$ that we correct here.) The packaged surface Tutte polynomial extends the surface Tutte polynomial of~\cite{maps1,maps2} to the pseudo-surface setting via packaged ribbon graphs. This extension not only satisfies a deletion-contraction relation for all edges, which is pivotal for the proof of Theorem~\ref{mainthm}, but also facilitates a cleaner quasi-tree expansion. Further discussion of this choice is given in the final Remark of Section~\ref{s.mainthm}.

\begin{definition}\label{PackPoly}
    Let $\pG=(\bG,\V,\wV,\B,\wB)$ be a packaged ribbon graph. Define the \emph{packaged surface Tutte polynomial} $\T(\pG;\boldsymbol{x},\boldsymbol{y})$, with variables $\boldsymbol{x}=(x,x_0,x_1,x_2, \ldots)$, $\boldsymbol{y}=(y,y_0,y_1,y_2,\ldots)$, as follows
    \[\T(\pG;\boldsymbol{x},\boldsymbol{y}):= \sum_{A\subseteq E} x^{n(G(\bG^*|A^c;\B))}y^{n(G(\bG|A;\V))}\prod_{\substack{H \text{ cpt. of } \\  G(\bG^*|A^c;\B)}}x_{\gamma(\bG^*,H)}\prod_{\substack{K \text{ cpt. of } \\  G(\bG|A;\V)}}y_{\gamma(\bG,K)},\]
    where $n(G):=e(G)-v(G)+k(G)$ is the \emph{nullity} of a (vertex-weighted) graph $G$, and
    \[\gamma(\bG,K):=2k(K)+e(K)-v(K)+\wV(K)-b(\bG[K]),\]
    \[\gamma(\bG^*,H):=2k(H)+e(H)-v(H)+\wB(H)-b(\bG^*[H]),\]
    here $\wV(K):=\sum_{v\in V(K)} \wV([v])$, and $\wB(H):=\sum_{v\in V(H)} \wB([v])$.
\end{definition}

The \emph{surface Tutte polynomial} $\Tt(\bG;\boldsymbol{x},\boldsymbol{y})$, where $\boldsymbol{x}$ and $\boldsymbol{y}$ are as above, of an orientable ribbon graph $\bG$ can be recovered from the packaged surface Tutte polynomial as follows. Let $\pG $ be the  packaged ribbon graph $(\bG, \{  \{v\}:v\in V \} , 0_V, \{ \{b\}:b\in B \} , 0_B )$, where  $0_{\cdot}$ denotes the zero-map. Then,
\begin{equation}\label{e.SurfTut}
\Tt(\bG;\boldsymbol{x},\boldsymbol{y})=\T(\pG;\boldsymbol{x'},\boldsymbol{y'}) ,
\end{equation}
where $x'=x$, $y'=y$, $x'_\ga = x_{\ga/2}$, $y'_\ga=y_{\ga/2}$, for $\ga \in \mathbb{N}$.

As alluded to earlier, the packaged surface Tutte polynomial can be equivalently defined  through a recursive deletion-contraction relation that terminates when no edges remain. This was originally shown for orientable packaged ribbon graphs (i.e., those that do not contain a M\"obius band) in~\cite{maya1}. As we allow non-orientability, we instead make use of the analogous result from~\cite{thesis}, noting that an edge is a loop in the packaging if and only if the set of its adjacent vertices is the subset of a block in the partition.

\begin{theorem}[\cite{maya1, thesis}]\label{dcthm}
    Let $\pG=(\bG,\V,\wV,\B,\wB)$ be a packaged ribbon graph and $e\in E$. Then
	\begin{equation}\label{dceq}
	    \T(\pG;\boldsymbol{x},\boldsymbol{y}) = x^{\al(e)}\T(\pG\backslash e;\boldsymbol{x},\boldsymbol{y})+y^{\be(e)}\T(\pG/e;\boldsymbol{x},\boldsymbol{y}),
	\end{equation}
    where
    \[\al(e)=\begin{cases}
        1 & \text{if $e$ is a loop in $G(\bG^*;\B)$,}\\
        0 & \text{otherwise.}
    \end{cases} \enspace \text{ and } \enspace \be(e)=\begin{cases}
        1 & \text{if $e$ is a loop in $G(\bG;\V)$,}\\
        0 & \text{otherwise.}\end{cases}.\]
    If $\bG$ has no edges, then 
       \[\T(\pG;\boldsymbol{x},\boldsymbol{y})=\prod_{[b]\in \B} x_{1- \vert [b]\vert +\wB([b])} \prod_{[v]\in\V} y_{1-\vert [v]\vert +\wV([v])}\]
    where $\vert [x] \vert$ denotes how many elements are in the block $[x]$. 
\end{theorem}

\subsection{A quasi-tree expansion for the packaged surface Tutte polynomial}\label{s.mainthm}
Our main result is that the packaged surface Tutte polynomial has an equivalent formulation as a sum over its quasi-trees, or a quasi-tree expansion. This completes the expected three formulations akin to the classic Tutte polynomial.

\begin{theorem}\label{mainthm}
Let $\pG=(\bG,\V,\wV,\B,\wB)$ be a connected packaged ribbon graph and $\prec$ be a total ordering of its edges. Then
\[\T(\pG;\boldsymbol{x},\boldsymbol{y})=\sum_{Q\in\Q_{\bG}}x^{n(G(\bG^*|D_Q^*\cup  N_Q^*;\B)}y^{n(G(\bG|D_Q\cup N_Q;\V)}\T(\pG\ba D_Q^*\cup N_Q^*\con D_Q\cup N_Q;\boldsymbol{x},\boldsymbol{y}).\]
\end{theorem}

Note that the quasi-tree expansion of $\T(\pG;\boldsymbol{x},\boldsymbol{y})$ does depend on the ordering $\prec$, since the activities of each edge with respect to a quasi-tree depend on the edge orders.

\begin{proof}
Using Equation~\eqref{dceq} we delete and contract edges of $\pG$ in reverse order beginning with the highest order edge. If an edge is a bridge or a plane loop, then we skip it and proceed to the next edge. This process terminates when either no edges remain or all remaining edges are bridges and plane loops. (Figure~\ref{f.prf} shows an example illustrating this process.)  
This results in an expression of the form
\[\T(\pG;\boldsymbol{x},\boldsymbol{y})=\sum_{\pH\in\mathcal{T}(\pG)}x^iy^j \T(\pH;\boldsymbol{x},\boldsymbol{y}),\]
where $\mathcal{T}(\pG)$ is the set of terminal packaged ribbon graphs obtained from the above process. For $\pH\in\mathcal{T}(\pG)$, let $A$ and $B$ denote the set of contracted edges and the set of deleted edges respectively (i.e. $\pH=\pG\ba B \con A$). Then $i$ is the minimum number of edges that need to be deleted to break all cycles in  $G(\pG^*|B;\B)$, and $j$ is the minimum number of edges that need to be deleted to break all cycles  in $G(\pG|A;\V)$. These are equal to the nullity of the respective packagings. 

 \begin{figure}
    \centering
    \input{Diagrams/QTtree}
    \caption{Applying the deletion-contraction process from the proof of Theorem~\ref{mainthm} to a packaged ribbon graph with ordering $e\prec f\prec g$. The labels on the arrows are the coefficients for the packaged surface Tutte polynomial applied to the packaged ribbon graph at the head of the arrow. }
    \label{f.prf}
\end{figure}
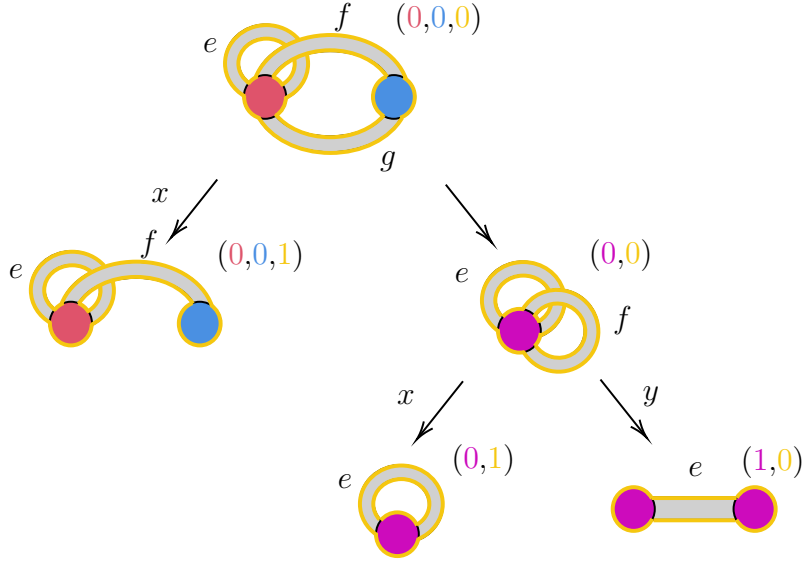

We show that the terminal packaged ribbon graphs are in one-to-one correspondence with the quasi-trees of $\pG$. 
As quasi-trees only have one boundary component, a subset $A\subseteq E(\pG)$ is the edge set of a quasi-tree of $\pG$ if and only if $\pG\ba A^c\con A$ is an edgeless single vertex.
Deleting a bridge or contracting a plane loop disconnects $\pG$ and is the only way to do so through deleting or contracting a single edge. Each terminal packaged ribbon graph $\pH=\pG\ba B\con A$ is connected and either an edgeless single vertex or contains only bridges and plane loops. So there is a unique way to delete or contract any remaining edges (i.e. contract bridges and delete plane  loops) to obtain a single vertex with no edges. This gives rise to a unique quasi-tree with edge set equal to the union of $A$ and the set of bridges in $\pH$. Its external edges are the union of $B$ and the set of plane loops in $\pH$. All quasi-trees can be constructed in this way. 

We require the following claim that we prove later. 
\begin{claim}\label{claim}
    Within the terminal packaged ribbon graphs any remaining edges are live and orientable with respect to the corresponding quasi-tree.  
\end{claim}

So for each terminal packaged ribbon graph, edges that are not in it are either dead or live and non-orientable with respect to the corresponding quasi-tree. Thus, the set of edges that were contracted is the set of internally dead and internally live non-orientable edges. Similarly, the set of edges that were deleted is the set of externally dead and externally live non-orientable edges. So each terminal packaged ribbon graph is $\pG\ba D^*_Q\cup N_Q^*\con D_Q\cup  N_Q$ where $Q$ is its corresponding quasi-tree.
Hence, the result follows. 
\end{proof}

\begin{example}
    Let $\pG=(\bG,\V,\wV,\B,\wB)$ be the packaged ribbon graph consisting of a loop interlaced with a 2-cycle where every block has size one and weight 0. A depiction of $\pG$ is shown at the top of Figure~\ref{f.prf}. Suppose the edges of $\bG$ are ordered $e\prec f \prec g$. 
    The ribbon graph $\bG$ has three quasi-trees given by the edge sets $\{f\}$, $\{g\}$ and $\{e,f,g\}$. The quasi-tree $\bG|f$ has $e$ externally live, $f$ internally live and $g$ externally dead. The quasi-tree $\bG|g$ has $e$ externally live, $f$ externally dead and $g$ internally dead. Finally, the quasi-tree $\bG$ has $e$ internally live and $f$ and $g$ internally dead. All edges are orientable with respect to each quasi-tree. So 
    \begin{multline*}
        \T(\pG;\boldsymbol{x},\boldsymbol{y})=x\T(\pG\ba g;\boldsymbol{x},\boldsymbol{y}) +x\T(\pG\ba f\con g ;\boldsymbol{x},\boldsymbol{y})+y\T(\pG\con \{f,g\};\boldsymbol{x},\boldsymbol{y})\\
        =x(xx_2y_0(xy_0+1)+yx_0y_0(xy_0+1)) + x(xx_2y_0+yx_0y_0)+y(xx_0y_0+yx_0y_2).
    \end{multline*}
\end{example}

\begin{proof}[Proof of Claim~\ref{claim}]
Let $\pH=\pG\ba B\con A$ be a terminal packaged ribbon graph. Denote its ribbon graph by $\bH$ and corresponding quasi-tree by $Q$. If $\bH$ is edgeless, then we are done, so let $e$ be an edge in $\bH$.
First, we suppose $e$ is non-orientable with respect to $Q$. 
So $e$ is a non-orientable loop in $\bG^{E(Q)}$, and thus  in  $\bG^{E(Q)}\ba A\cup B$. By Proposition~\ref{p.pd}, \begin{equation}\label{eq.claim}
    \bG^{E(Q)}\ba A\cup B= (\bG^{A}\ba A\cup B)^{E(Q)-A}=\bH^{E(Q)-A}.
\end{equation} So there esists a partial dual of $\bH$ in which $e$ is a non-orientable loop. However, as $e$ is either a bridge or a plane loop, no such partial dual exists and $e$ must be orientable.

Now suppose $e$ is dead and therefore links a lower ordered edge $f$ with respect to $Q$. Let $A'\subseteq A$ be the set of edges in $A$ whose order is greater than $e$, and $B'\subseteq B$ be the set of edges in $B$ whose order is greater than $e$. Then $e,f\notin A'\cup B'$ and $e$ is either a bridge or a plane loop in $\bH'=\bG\ba B'\con A'$. As $e$ links $f$, then $e$ is interlaced with $f$ in $\bG^{E(Q)}$, and thus in $\bG^{E(Q)}\ba A'\cup B'$. By Equation~\eqref{eq.claim}, there exists a partial dual of $\bH'$ in which $e$ and $f$ are loops that interlace one another. In other words, there exists a spanning ribbon subgraph of $\bH'$ with a boundary component that, when traveling around it, meets arcs on the edges $e$ and $f$ twice in the order $e$ $f$ $e$ $f$. However, as $e$ is either a bridge or a plane loop in $\bH'$, this cannot happen. So $e$ must be live.  
\end{proof}

\begin{remark}
    Theorem~\ref{mainthm} can be extended to disconnected packaged ribbon graphs, though we omit the details here. Quasi-forests naturally inherit the definitions of live, dead, non-orientable, internal and external edges. So the proof of Theorem~\ref{mainthm} can be modified for quasi-forests, noting that if a spanning subgraph $\bH$ is a quasi-forest then $\pG\ba E(\pH)^c \con E(\pH)$ is a collection of $k(\pG)$-many isolated vertices. It is then straightforward to obtain an expression for $\T(\pG,\boldsymbol{x},\boldsymbol{y})$ as a sum over the quasi-forests of $\pG$.
\end{remark}

\begin{remark}
When constructing quasi-tree expansions for ribbon graph polynomials, the appropriate setting is graphs embedded in pseudo-surfaces or packaged ribbon graphs. This is because these expansions use deletion and contraction, and deleting and contracting edges results in a loss of information about the topology of the embedded graph when working with surfaces, instead of pseudo-surfaces. (Further discussion on this can be found in~\cite[Section~2]{HM}.) The quasi-tree expansions for the Bollob\'as--Riordan~\cite{Champan, V-T} and Kruskal polynomials~\cite{Butler} inadvertently use the pseudo-surface setting. This can be seen with the construction $G_Q$ which is the packaging (without any vertex weights) $G(\bG\con D_Q\cup N_Q\ba D_Q^*\cup N_Q^*\cup O^*_Q;\V)$.

So whilst it may be possible to directly obtain a quasi-tree expansion for the surface Tutte polynomial, the packaged ribbon graph setting is unavoidable. Further, as the deletion-contraction relation of Theorem~\ref{dcthm} is pivotal to the proof of Theorem~\ref{mainthm} and Theorem~\ref{dcthm} does not hold for the surface Tutte polynomial, finding a quasi-tree expansion for the surface Tutte polynomial directly would be needlessly complicated. Finally, the surface Tutte polynomial~\cite{maps2} for non-orientable ribbon graphs uses a signed genus that distinguishes orientablility, which is not compatible with recursive relations. So any quasi-tree expansion for it would be restricted to orientable ribbon graphs. For these reasons we choose to work with the packaged surface Tutte polynomial for this paper. 

However, using Equation~\eqref{e.SurfTut} and Theorem~\ref{mainthm}, one can obtain a quasi-tree expansion for the surface Tutte polynomial that holds for orientable ribbon graphs.
\end{remark}

\subsection{Derivation of expansions for other polynomials}
Here we rederive the quasi-tree expansions for the Krushkal polynomial, and thus, by extension, for the Bollob\'as--Riordan and Las~Vergnas polynomials, from Theorem~\ref{mainthm}. 

The Krushkal polynomial was introduced~\cite{Krushkal} for graphs embedded in orientable surfaces to unify the earlier topological Tutte polynomials of Bollob\'as--Riordan and Las~Vergnas. It was later extended to non-orientable surfaces by Butler~\cite{Butler}, and we use that variant here.
In our notation, the \emph{generalised Krushkal polynomial} is defined by
\begin{equation*}
    P(\bG;\al,\be,a,b):=\sum_{A\subseteq E}\al^{k(\bG|A)-k(\bG)} \be^{k(\bG^*|A^c)-k(\bG^*)} a^{\tfrac{1}{2}\ga(\bG|A)} b^{\tfrac{1}{2}\ga(\bG^*|A^c)}.
\end{equation*}
It can be obtained from the packaged surface Tutte polynomial as follows. Let $\bG$ be a ribbon graph and $\pG$ be the  packaged ribbon $(\bG, \{  \{v\}:v\in V \} , 0_V, \{ \{b\}:b\in B \} , 0_B )$ (recall that  $0_{\cdot}$ denotes the zero-map). Then,
\begin{equation}\label{kspec}
    P(\bG;\al,\be,a,b) = (\al\be)^{-k(\bG)}\T(\pG;\boldsymbol{x},\boldsymbol{y}),
\end{equation}
with $x=1$, $x_\ga = \be b^{\tfrac{1}{2}\ga}$, $y=1$, $y_\ga=\al a^{\tfrac{1}{2}\ga}$, for $\ga \in \mathbb{N}$.
Since the generalised Krushkal polynomial $P(\bG;\al,\be,a,b)$ is a specialisation of the packaged surface Tutte polynomial, we can recover its quasi-tree expansion~\cite[Theorem~4.4]{Butler} from Theorem~\ref{mainthm}.

\begin{theorem}[\cite{Butler}]\label{t.krush}
    Let $\bG$ be a connected ribbon graph, then 
    \[P(\bG;\al,\be,a,b)=\sum_{Q\in \Q_\bG}T(G_Q;\al+1,a+1)T(G_{Q^*}^*;\be+1,b+1)a^{\tfrac{1}{2}\ga(\bG|D_Q\cup N_Q)}b^{\tfrac{1}{2}\ga(\bG^*|D_Q^*\cup N_Q^*)},\]
    where $G_Q$ is the graph whose vertices are the connected components of $\bG|D_Q\cup N_Q$ and whose edges are those in $O_Q$; $G^*_{Q^*}$ is the graph whose vertices are the connected components of $\bG^*|D_Q^*\cup N_Q^*$ and whose edges are those in $O^*_Q$; and for a graph $G$ 
    \[T(G;x,y):=\sum_{A\subseteq E(G)} (x-1)^{k(G|A)-k(G)}(y-1)^{n(G)}.\]
\end{theorem}

\begin{proof}
    Let $\pG=(\bG, \{  \{v\}:v\in V \} , 0_V, \{ \{b\}:b\in B \} , 0_B )$.
    Using Theorem~\ref{mainthm} and Equation~\eqref{kspec}, 
    \begin{equation*}
        P(\bG;\al,\be,a,b)=(\al\be)^{-k(\bG)}\sum_{Q\in \Q_\bG}\sum_{A\subseteq E(\hat{\bG})} \prod_{\substack{H \text{ cpt. of } \\  G(\hat{\bG}^*|A^c;\B)}} \be b^{\tfrac{1}{2}\gamma(\hat{\bG}^*,H)}
        \prod_{\substack{K \text{ cpt. of } \\  G(\hat{\bG}|A;\V)}} \al a^{\tfrac{1}{2}\gamma(\hat{\bG},K)},
    \end{equation*}
          where $\hat{\pG}=(\hat{\bG},\V, \wV,\B, \wB)$ is the packaged ribbon graph $\pG\ba D^*_Q\cup N^*_Q\con D_Q\cup N_Q$. It is straightforward to verify that $\ga(\hat{\bG},K)$ and $\ga(\hat{\bG}^*,H)$ are additive over connected components. Denote the sum of $\ga(\hat{\bG},K)$ and $\ga(\hat{\bG}^*,H)$ over the connected components  as $\ga_{\V}(\hat{\bG}|A)$ and $\ga_{\B}(\hat{\bG}^*|A^c)$ respectively. 
            So we can rewrite the above as 
    \begin{equation*}
         P(\bG;\al,\be,a,b)=\sum_{Q\in \Q_\bG}\sum_{A\subseteq E(\hat{\bG})}  \al^{k(G(\hat{\bG}|A;\V))-k(\bG)}\be^{k(G(\hat{\bG}^*|A^c;\B))-k(\bG)}a^{\tfrac{1}{2}\gamma_{\V}(\hat{\bG}|A)}b^{\tfrac{1}{2}\gamma_{\B}(\hat{\bG}^*|A^c)}.
    \end{equation*}
    Recall from the proof of Theorem~\ref{mainthm} that $\hat{\bG}$ has edge set $E(\hat{\bG})=O_Q\cup O^*_Q$, where  the edges in $O_Q$ are bridges and those in $O^*_Q$ are plane loops. If an edge is a bridge, then it is a plane loop in the dual $\hat{\bG}^*$, and vice versa. If $e$ is a plane loop, then deleting $e$ 
    does not change $\ga_\V$, nor the number of connected components of the ribbon graph. So 
    \begin{multline*}
         P(\bG;\al,\be,a,b)=\\\sum_{Q\in \Q_\bG}\sum_{A\subseteq O_Q}
         \sum_{B\subseteq O_Q^*}
         \al^{k(G(\hat{\bG}|A;\V))-k(\bG)}\be^{k(G(\hat{\bG}^*|B;\B))-k(\bG)}a^{\tfrac{1}{2}\gamma_\V(\hat{\bG}|A)}b^{\tfrac{1}{2}\gamma_\B(\hat{\bG}^*|B)}.
    \end{multline*}
    We make the following observations. As deletion does not change the vertex weights and all blocks started with weight 0, then $\sum_{v\in V(\hat{\bG}|A)} \wV([v])=n(G(\bG|D_Q\cup N_Q;\V))$. Additionally, as contraction does not change the number of boundary components and every edge in $A$ is a bridge, then $b(\hat{\bG}|A)=b(\bG|D_Q\cup N_Q\cup A)=b(\bG|D_Q\cup N_Q)-|A|$. Finally, note that contracting an edge in a packaged ribbon graph does not disconnect the packaging, so $k(G(\bG|D_Q\cup N_Q;\V))=v(G(\bG/D_Q \cup N_Q;\V))$. Similar observations can be made for the dual $\pG^*$. So we obtain
    \[\tfrac{1}{2}\gamma_\V(\hat{\bG}|A)=n(G(\hat{\bG}|A;\V))+\tfrac{1}{2}\ga(\bG|D_Q\cup N_Q),\]
    and
    \[ \tfrac{1}{2}\gamma_\B(\hat{\bG}^*|B)=n(G(\hat{\bG}^*|B;\V))+\tfrac{1}{2}\ga(\bG^*|D_Q^*\cup N_Q^*).\]
    Note that the packaging $G(\hat{\bG}|A;\V)$ (without vertex weights) is equivalent to the graph $G_Q|A$, and the packaging $G(\hat{\bG}^*|B;\B)$ (without vertex weights) is equivalent to the graph $G_Q^*|B$. 
    Finally, we can deduce from the  proof of Theorem~\ref{mainthm} that contracting $D_Q\cup N_Q$ and deleting $D_Q^*\cup N^*_Q$ does not change the number of connected components in the ribbon graph, so $k(\bG\ba D^*_Q\cup N^*_Q\con D_Q\cup N_Q)=k(\bG)$.
    Hence, the result follows.
\end{proof}

\begin{remark}
    Butler derives a quasi-tree expansions for the Las~Vergnas polynomial~\cite[Theorems~5.6]{Butler} and rederives the quasi-tree expansion for the Bollob\'as--Riordan~\cite[Theorems~5.2]{Butler} polynomial from  the quasi-tree expansion of the Krushkal polynomial, Theorem~\ref{t.krush}. So it follows that such expansions can also be recovered from Theorem~\ref{mainthm}, the quasi-tree expansion of the packaged surface Tutte polynomial.
\end{remark}

\section*{Declarations}

This research did not receive any specific grant from funding agencies in the public, commercial, or not-for-profit sectors.
For the purpose of open access, the author has applied a Creative Commons Attribution (CC
BY) licence to any Author Accepted Manuscript version arising. No underlying data is associated
with this article. There are no conflicts of interest.

 \bibliographystyle{abbrv} 
\bibliography{quasitree.bib}
\end{document}

%% file: Diagrams/Pd/nonloop.tex
\tikzset{every picture/.style={line width=0.75pt}} 

\begin{tikzpicture}[x=0.45pt,y=0.45pt,yscale=-1,xscale=1]

\draw  [draw opacity=0][fill={rgb, 255:red, 208; green, 208; blue, 208 }  ,fill opacity=1 ] (76.75,45.59) -- (23.64,47.04) -- (24.07,62.85) -- (77.18,61.4) -- cycle ;
\draw  [draw opacity=0][fill={rgb, 255:red, 208; green, 208; blue, 208 }  ,fill opacity=1 ] (87.82,66.54) -- (50.87,104.71) -- (39.5,93.71) -- (76.45,55.54) -- cycle ;
\draw  [draw opacity=0][fill={rgb, 255:red, 208; green, 208; blue, 208 }  ,fill opacity=1 ] (89.12,38.83) -- (48.34,4.77) -- (38.2,16.91) -- (78.98,50.97) -- cycle ;
\draw    (81.39,73.37) -- (50.87,104.71) ;
\draw    (70.27,62.27) -- (39.5,93.71) ;
\draw    (48.34,4.77) -- (80.41,32.07) ;
\draw    (38.2,16.91) -- (71.75,44.59) ;
\draw  [draw opacity=0][fill={rgb, 255:red, 208; green, 208; blue, 208 }  ,fill opacity=1 ] (233.72,68.62) -- (277.39,98.87) -- (286.4,85.86) -- (242.73,55.61) -- cycle ;
\draw  [draw opacity=0][fill={rgb, 255:red, 208; green, 208; blue, 208 }  ,fill opacity=1 ] (233.72,42.61) -- (277.39,12.36) -- (286.4,25.36) -- (242.73,55.61) -- cycle ;
\draw  [fill={rgb, 255:red, 208; green, 208; blue, 208 }  ,fill opacity=1 ] (81.25,44.91) -- (226.5,44.91) -- (226.5,65.91) -- (81.25,65.91) -- cycle ;
\draw  [fill={rgb, 255:red, 128; green, 128; blue, 128 }  ,fill opacity=1 ] (201,55) .. controls (201,41.19) and (212.19,30) .. (226,30) .. controls (239.81,30) and (251,41.19) .. (251,55) .. controls (251,68.81) and (239.81,80) .. (226,80) .. controls (212.19,80) and (201,68.81) .. (201,55) -- cycle ;
\draw    (107.9,44.91) -- (202.9,44.91) ;
\draw    (108.7,65.71) -- (203.7,65.71) ;
\draw    (241.35,74.06) -- (277.39,98.87) ;
\draw    (250.1,61.01) -- (286.4,85.86) ;
\draw    (277.39,12.36) -- (243,36.66) ;
\draw    (286.4,25.36) -- (250.5,49.91) ;
\draw    (23.64,47.04) -- (76.75,45.59) ;
\draw    (24.07,62.85) -- (77.18,61.4) ;
\draw  [fill={rgb, 255:red, 128; green, 128; blue, 128 }  ,fill opacity=1 ] (61,56) .. controls (61,42.19) and (72.19,31) .. (86,31) .. controls (99.81,31) and (111,42.19) .. (111,56) .. controls (111,69.81) and (99.81,81) .. (86,81) .. controls (72.19,81) and (61,69.81) .. (61,56) -- cycle ;

\end{tikzpicture}

%% file: Diagrams/Pd/loop.tex
\tikzset{every picture/.style={line width=0.75pt}} 

\begin{tikzpicture}[x=0.45pt,y=0.45pt,yscale=-1,xscale=1]

\draw  [draw opacity=0][fill={rgb, 255:red, 208; green, 208; blue, 208 }  ,fill opacity=1 ] (136.75,49.88) -- (83.64,51.33) -- (84.07,67.14) -- (137.18,65.69) -- cycle ;
\draw  [draw opacity=0][fill={rgb, 255:red, 208; green, 208; blue, 208 }  ,fill opacity=1 ] (147.82,70.83) -- (110.87,109) -- (99.5,98) -- (136.45,59.83) -- cycle ;
\draw  [draw opacity=0][fill={rgb, 255:red, 208; green, 208; blue, 208 }  ,fill opacity=1 ] (149.12,43.12) -- (108.34,9.06) -- (98.2,21.2) -- (138.98,55.26) -- cycle ;
\draw    (141.39,77.66) -- (110.87,109) ;
\draw    (130.27,66.56) -- (99.5,98) ;
\draw    (108.34,9.06) -- (140.41,36.36) ;
\draw    (98.2,21.2) -- (131.75,48.88) ;
\draw    (83.64,51.33) -- (136.75,49.88) ;
\draw    (84.07,67.14) -- (137.18,65.69) ;
\draw  [draw opacity=0][fill={rgb, 255:red, 208; green, 208; blue, 208 }  ,fill opacity=1 ] (159.82,69.25) -- (203.49,99.5) -- (212.5,86.5) -- (168.83,56.25) -- cycle ;
\draw  [draw opacity=0][fill={rgb, 255:red, 208; green, 208; blue, 208 }  ,fill opacity=1 ] (159.82,43.24) -- (203.49,12.99) -- (212.5,26) -- (168.83,56.25) -- cycle ;
\draw    (167.45,74.7) -- (203.49,99.5) ;
\draw    (176.2,61.65) -- (212.5,86.5) ;
\draw    (203.49,12.99) -- (169.1,37.3) ;
\draw    (212.5,26) -- (176.6,50.55) ;
\draw  [fill={rgb, 255:red, 208; green, 208; blue, 208 }  ,fill opacity=1 ] (147.07,76.01) .. controls (152.49,92.27) and (164.03,103.51) .. (177.4,103.51) .. controls (195.97,103.51) and (211.03,81.82) .. (211.03,55.06) .. controls (211.03,28.31) and (195.97,6.62) .. (177.4,6.62) .. controls (165.09,6.62) and (154.33,16.14) .. (148.47,30.35) -- (159.08,39.41) .. controls (162.42,27.8) and (169.37,19.81) .. (177.4,19.81) .. controls (188.69,19.81) and (197.84,35.59) .. (197.84,55.06) .. controls (197.84,74.53) and (188.69,90.32) .. (177.4,90.32) .. controls (168.8,90.32) and (161.43,81.15) .. (158.42,68.17) -- cycle ;
\draw  [fill={rgb, 255:red, 128; green, 128; blue, 128 }  ,fill opacity=1 ] (124.7,55.06) .. controls (124.7,40.51) and (136.5,28.71) .. (151.05,28.71) .. controls (165.6,28.71) and (177.4,40.51) .. (177.4,55.06) .. controls (177.4,69.62) and (165.6,81.41) .. (151.05,81.41) .. controls (136.5,81.41) and (124.7,69.62) .. (124.7,55.06) -- cycle ;

\end{tikzpicture}

%% file: Diagrams/Pd/nonori1.tex
\tikzset{every picture/.style={line width=0.75pt}} 

\begin{tikzpicture}[x=0.45pt,y=0.45pt,yscale=-1,xscale=1]

\draw  [draw opacity=0][fill={rgb, 255:red, 208; green, 208; blue, 208 }  ,fill opacity=1 ] (135.75,48.88) -- (82.64,50.33) -- (83.07,66.14) -- (136.18,64.69) -- cycle ;
\draw  [draw opacity=0][fill={rgb, 255:red, 208; green, 208; blue, 208 }  ,fill opacity=1 ] (146.82,69.83) -- (109.87,108) -- (98.5,97) -- (135.45,58.83) -- cycle ;
\draw  [draw opacity=0][fill={rgb, 255:red, 208; green, 208; blue, 208 }  ,fill opacity=1 ] (148.12,42.12) -- (107.34,8.06) -- (97.2,20.2) -- (137.98,54.26) -- cycle ;
\draw    (140.39,76.66) -- (109.87,108) ;
\draw    (129.27,65.56) -- (98.5,97) ;
\draw    (107.34,8.06) -- (139.41,35.36) ;
\draw    (97.2,20.2) -- (130.75,47.88) ;
\draw    (82.64,50.33) -- (135.75,48.88) ;
\draw    (83.07,66.14) -- (136.18,64.69) ;
\draw  [draw opacity=0][fill={rgb, 255:red, 208; green, 208; blue, 208 }  ,fill opacity=1 ] (159.82,69.25) -- (203.49,99.5) -- (212.5,86.5) -- (168.83,56.25) -- cycle ;
\draw  [draw opacity=0][fill={rgb, 255:red, 208; green, 208; blue, 208 }  ,fill opacity=1 ] (159.82,43.24) -- (203.49,12.99) -- (212.5,26) -- (168.83,56.25) -- cycle ;
\draw    (167.45,74.7) -- (203.49,99.5) ;
\draw    (176.2,61.65) -- (212.5,86.5) ;
\draw    (203.49,12.99) -- (169.1,37.3) ;
\draw    (212.5,26) -- (176.6,50.55) ;
\draw  [fill={rgb, 255:red, 208; green, 208; blue, 208 }  ,fill opacity=1 ] (176.15,9.62) .. controls (191.55,9.82) and (213,31.92) .. (204.5,56.92) .. controls (196.5,47.42) and (189,24.17) .. (176.15,22.81) .. controls (164.25,23.42) and (158.67,36.18) .. (157.25,41.92) .. controls (147.42,33.68) and (152.76,38.2) .. (147.5,33.42) .. controls (149.28,28.75) and (160.75,9.42) .. (176.15,9.62) -- cycle ;
\draw  [fill={rgb, 255:red, 208; green, 208; blue, 208 }  ,fill opacity=1 ] (176.15,104.22) .. controls (191.55,104.02) and (213,81.92) .. (204.5,56.92) .. controls (196.5,66.42) and (189,89.67) .. (176.15,91.03) .. controls (164.25,90.42) and (158.67,77.66) .. (157.25,71.92) .. controls (147.42,80.16) and (152.76,75.64) .. (147.5,80.42) .. controls (149.28,85.1) and (160.75,104.42) .. (176.15,104.22) -- cycle ;
\draw  [fill={rgb, 255:red, 128; green, 128; blue, 128 }  ,fill opacity=1 ] (124.7,55.06) .. controls (124.7,40.51) and (136.5,28.71) .. (151.05,28.71) .. controls (165.6,28.71) and (177.4,40.51) .. (177.4,55.06) .. controls (177.4,69.62) and (165.6,81.41) .. (151.05,81.41) .. controls (136.5,81.41) and (124.7,69.62) .. (124.7,55.06) -- cycle ;

\end{tikzpicture}

%% file: Diagrams/Pd/nonori2.tex
\tikzset{every picture/.style={line width=0.75pt}} 

\begin{tikzpicture}[x=0.45pt,y=0.45pt,yscale=-1,xscale=1]

\draw  [draw opacity=0][fill={rgb, 255:red, 208; green, 208; blue, 208 }  ,fill opacity=1 ] (126.75,47.88) -- (73.64,49.33) -- (74.07,65.14) -- (127.18,63.69) -- cycle ;
\draw  [draw opacity=0][fill={rgb, 255:red, 208; green, 208; blue, 208 }  ,fill opacity=1 ] (137.82,68.83) -- (100.87,107) -- (89.5,96) -- (126.45,57.83) -- cycle ;
\draw  [draw opacity=0][fill={rgb, 255:red, 208; green, 208; blue, 208 }  ,fill opacity=1 ] (139.12,41.12) -- (98.34,7.06) -- (88.2,19.2) -- (128.98,53.26) -- cycle ;
\draw    (131.39,75.66) -- (100.87,107) ;
\draw    (120.27,64.56) -- (89.5,96) ;
\draw    (98.34,7.06) -- (130.41,34.36) ;
\draw    (88.2,19.2) -- (121.75,46.88) ;
\draw    (73.64,49.33) -- (126.75,47.88) ;
\draw    (74.07,65.14) -- (127.18,63.69) ;
\draw  [draw opacity=0][fill={rgb, 255:red, 208; green, 208; blue, 208 }  ,fill opacity=1 ] (177.92,68.25) -- (221.59,98.5) -- (230.6,85.5) -- (186.93,55.25) -- cycle ;
\draw  [draw opacity=0][fill={rgb, 255:red, 208; green, 208; blue, 208 }  ,fill opacity=1 ] (177.92,42.24) -- (221.59,11.99) -- (230.6,25) -- (186.93,55.25) -- cycle ;
\draw    (185.55,73.7) -- (221.59,98.5) ;
\draw    (194.3,60.65) -- (230.6,85.5) ;
\draw    (221.59,11.99) -- (187.2,36.3) ;
\draw    (230.6,25) -- (194.7,49.55) ;
\draw  [fill={rgb, 255:red, 208; green, 208; blue, 208 }  ,fill opacity=1 ] (196.82,5.95) .. controls (212.22,6.15) and (233.67,28.25) .. (225.17,53.25) .. controls (217.17,43.75) and (209.67,20.5) .. (196.82,19.14) .. controls (184.92,19.75) and (179.34,32.51) .. (177.92,38.25) .. controls (172.25,45.43) and (172.26,47.71) .. (167,42.93) .. controls (166.75,30.93) and (181.42,5.75) .. (196.82,5.95) -- cycle ;
\draw  [fill={rgb, 255:red, 208; green, 208; blue, 208 }  ,fill opacity=1 ] (196.82,100.55) .. controls (212.22,100.35) and (233.67,78.25) .. (225.17,53.25) .. controls (217.17,62.75) and (209.67,86) .. (196.82,87.36) .. controls (184.92,86.75) and (179.34,73.99) .. (177.92,68.25) .. controls (172.25,61.07) and (172.51,60.4) .. (167.25,65.18) .. controls (167.25,77.54) and (181.42,100.75) .. (196.82,100.55) -- cycle ;
\draw  [fill={rgb, 255:red, 128; green, 128; blue, 128 }  ,fill opacity=1 ] (145.82,71.68) .. controls (137.73,78.91) and (125.96,77.22) .. (119.55,67.9) .. controls (113.13,58.58) and (114.49,45.17) .. (122.58,37.93) .. controls (130.68,30.7) and (142.44,32.39) .. (148.86,41.71) .. controls (156.6,52.96) and (160.48,58.59) .. (160.47,58.59) .. controls (160.48,58.59) and (155.59,62.95) .. (145.82,71.68) -- cycle ;
\draw  [fill={rgb, 255:red, 128; green, 128; blue, 128 }  ,fill opacity=1 ] (167.19,67.93) .. controls (173.54,77.15) and (185.31,78.69) .. (193.49,71.39) .. controls (201.68,64.08) and (203.16,50.68) .. (196.81,41.46) .. controls (190.46,32.24) and (178.69,30.69) .. (170.51,38) .. controls (160.64,46.82) and (155.7,51.23) .. (155.69,51.23) .. controls (155.7,51.23) and (159.53,56.8) .. (167.19,67.93) -- cycle ;

\end{tikzpicture}

%% file: Diagrams/QTex1.tex
\tikzset{every picture/.style={line width=0.75pt}} 

\begin{tikzpicture}[x=0.75pt,y=0.75pt,yscale=-1,xscale=1]

\draw  [color={rgb, 255:red, 0; green, 0; blue, 0 }  ,draw opacity=1 ][fill={rgb, 255:red, 208; green, 208; blue, 208 }  ,fill opacity=1 ][line width=0.75]  (202.56,102) .. controls (190.99,98.42) and (183.64,87.26) .. (185.84,75.86) .. controls (188.21,63.58) and (200.78,55.69) .. (213.91,58.23) .. controls (227.05,60.77) and (235.77,72.78) .. (233.39,85.05) .. controls (231.57,94.52) and (223.68,101.38) .. (214.15,102.86) -- (212.54,94.93) .. controls (218.97,94.08) and (224.28,89.67) .. (225.47,83.52) .. controls (226.99,75.62) and (221.13,67.85) .. (212.38,66.16) .. controls (203.63,64.46) and (195.29,69.5) .. (193.77,77.39) .. controls (192.35,84.74) and (197.32,91.99) .. (205.07,94.32) -- cycle ;
\draw  [color={rgb, 255:red, 0; green, 0; blue, 0 }  ,draw opacity=1 ][fill={rgb, 255:red, 208; green, 208; blue, 208 }  ,fill opacity=1 ][line width=0.75]  (290.53,100.94) .. controls (290.55,101.29) and (290.56,101.64) .. (290.56,102) .. controls (290.56,118.63) and (270.86,132.12) .. (246.56,132.12) .. controls (222.26,132.12) and (202.56,118.63) .. (202.56,102) .. controls (202.56,101.85) and (202.56,101.7) .. (202.56,101.55) -- (211.81,101.64) .. controls (211.81,101.76) and (211.8,101.88) .. (211.8,102) .. controls (211.8,113.53) and (227.36,122.88) .. (246.56,122.88) .. controls (265.75,122.88) and (281.31,113.53) .. (281.31,102) .. controls (281.31,101.72) and (281.31,101.44) .. (281.29,101.17) -- cycle ;
\draw  [color={rgb, 255:red, 0; green, 0; blue, 0 }  ,draw opacity=1 ][fill={rgb, 255:red, 208; green, 208; blue, 208 }  ,fill opacity=1 ][line width=0.75]  (202.78,95.18) .. controls (202.76,94.83) and (202.75,94.48) .. (202.75,94.13) .. controls (202.75,77.49) and (222.45,64) .. (246.75,64) .. controls (271.05,64) and (290.75,77.49) .. (290.75,94.13) .. controls (290.75,94.27) and (290.75,94.42) .. (290.75,94.57) -- (281.5,94.48) .. controls (281.5,94.36) and (281.5,94.24) .. (281.5,94.13) .. controls (281.5,82.59) and (265.94,73.25) .. (246.75,73.25) .. controls (227.56,73.25) and (212,82.59) .. (212,94.13) .. controls (212,94.4) and (212,94.68) .. (212.02,94.96) -- cycle ;
\draw  [color={rgb, 255:red, 0; green, 0; blue, 0 }  ,draw opacity=1 ][fill={rgb, 255:red, 128; green, 128; blue, 128 }  ,fill opacity=1 ][line width=0.75]  (196.45,99.92) .. controls (196.45,92.89) and (201.98,87.19) .. (208.79,87.19) .. controls (215.6,87.19) and (221.12,92.89) .. (221.12,99.92) .. controls (221.12,106.95) and (215.6,112.65) .. (208.79,112.65) .. controls (201.98,112.65) and (196.45,106.95) .. (196.45,99.92) -- cycle ;
\draw  [color={rgb, 255:red, 0; green, 0; blue, 0 }  ,draw opacity=1 ][fill={rgb, 255:red, 128; green, 128; blue, 128 }  ,fill opacity=1 ][line width=0.75]  (271.45,99.92) .. controls (271.45,92.89) and (276.98,87.19) .. (283.79,87.19) .. controls (290.6,87.19) and (296.12,92.89) .. (296.12,99.92) .. controls (296.12,106.95) and (290.6,112.65) .. (283.79,112.65) .. controls (276.98,112.65) and (271.45,106.95) .. (271.45,99.92) -- cycle ;

\draw (171,63.9) node [anchor=north west][inner sep=0.75pt]    {$e$};
\draw (247,43.4) node [anchor=north west][inner sep=0.75pt]    {$f$};
\draw (274,129.9) node [anchor=north west][inner sep=0.75pt]    {$g$};

\end{tikzpicture}

%% file: Diagrams/QTex2.tex
\tikzset{every picture/.style={line width=0.75pt}} 

\begin{tikzpicture}[x=0.75pt,y=0.75pt,yscale=-1,xscale=1]

\draw  [color={rgb, 255:red, 0; green, 0; blue, 0 }  ,draw opacity=1 ][fill={rgb, 255:red, 208; green, 208; blue, 208 }  ,fill opacity=1 ][line width=0.75]  (202.78,95.18) .. controls (202.76,94.83) and (202.75,94.48) .. (202.75,94.13) .. controls (202.75,77.49) and (222.45,64) .. (246.75,64) .. controls (271.05,64) and (290.75,77.49) .. (290.75,94.13) .. controls (290.75,94.27) and (290.75,94.42) .. (290.75,94.57) -- (281.5,94.48) .. controls (281.5,94.36) and (281.5,94.24) .. (281.5,94.13) .. controls (281.5,82.59) and (265.94,73.25) .. (246.75,73.25) .. controls (227.56,73.25) and (212,82.59) .. (212,94.13) .. controls (212,94.4) and (212,94.68) .. (212.02,94.96) -- cycle ;
\draw  [color={rgb, 255:red, 0; green, 0; blue, 0 }  ,draw opacity=1 ][fill={rgb, 255:red, 128; green, 128; blue, 128 }  ,fill opacity=1 ][line width=0.75]  (196.45,99.92) .. controls (196.45,92.89) and (201.98,87.19) .. (208.79,87.19) .. controls (215.6,87.19) and (221.12,92.89) .. (221.12,99.92) .. controls (221.12,106.95) and (215.6,112.65) .. (208.79,112.65) .. controls (201.98,112.65) and (196.45,106.95) .. (196.45,99.92) -- cycle ;
\draw  [color={rgb, 255:red, 0; green, 0; blue, 0 }  ,draw opacity=1 ][fill={rgb, 255:red, 128; green, 128; blue, 128 }  ,fill opacity=1 ][line width=0.75]  (271.45,99.92) .. controls (271.45,92.89) and (276.98,87.19) .. (283.79,87.19) .. controls (290.6,87.19) and (296.12,92.89) .. (296.12,99.92) .. controls (296.12,106.95) and (290.6,112.65) .. (283.79,112.65) .. controls (276.98,112.65) and (271.45,106.95) .. (271.45,99.92) -- cycle ;

\draw (247,43.4) node [anchor=north west][inner sep=0.75pt]    {$f$};

\end{tikzpicture}

%% file: Diagrams/QTex3.tex
\tikzset{every picture/.style={line width=0.75pt}} 

\begin{tikzpicture}[x=0.75pt,y=0.75pt,yscale=-1,xscale=1]

\draw  [color={rgb, 255:red, 0; green, 0; blue, 0 }  ,draw opacity=1 ][fill={rgb, 255:red, 208; green, 208; blue, 208 }  ,fill opacity=1 ][line width=0.75]  (199.55,81.21) .. controls (197.83,92.87) and (187.45,101.73) .. (175.06,101.52) .. controls (161.61,101.3) and (150.89,90.47) .. (151.11,77.34) .. controls (151.33,64.21) and (162.41,53.75) .. (175.85,53.98) .. controls (185.95,54.15) and (194.5,60.29) .. (198.05,68.89) -- (190.16,71.99) .. controls (187.82,66.49) and (182.26,62.57) .. (175.71,62.46) .. controls (166.95,62.31) and (159.73,69.04) .. (159.59,77.48) .. controls (159.44,85.93) and (166.43,92.9) .. (175.2,93.04) .. controls (183.27,93.18) and (190.02,87.49) .. (191.16,80) -- cycle ;
\draw  [color={rgb, 255:red, 0; green, 0; blue, 0 }  ,draw opacity=1 ][fill={rgb, 255:red, 208; green, 208; blue, 208 }  ,fill opacity=1 ][line width=0.75]  (197.61,70.83) .. controls (200.22,59.33) and (211.24,51.28) .. (223.58,52.42) .. controls (236.97,53.67) and (246.84,65.27) .. (245.63,78.35) .. controls (244.42,91.42) and (232.58,101.02) .. (219.19,99.77) .. controls (209.14,98.84) and (201.07,92.07) .. (198.18,83.22) -- (206.29,80.73) .. controls (208.2,86.4) and (213.44,90.73) .. (219.97,91.33) .. controls (228.7,92.14) and (236.41,85.98) .. (237.19,77.57) .. controls (237.97,69.15) and (231.52,61.68) .. (222.79,60.87) .. controls (214.76,60.12) and (207.59,65.28) .. (205.89,72.67) -- cycle ;
\draw  [color={rgb, 255:red, 0; green, 0; blue, 0 }  ,draw opacity=1 ][fill={rgb, 255:red, 208; green, 208; blue, 208 }  ,fill opacity=1 ][line width=0.75]  (198.05,68.89) .. controls (183,69.66) and (170.94,79.7) .. (170.53,92.37) .. controls (170.09,105.89) and (183.07,117.28) .. (199.54,117.82) .. controls (216,118.36) and (229.7,107.83) .. (230.14,94.31) .. controls (230.49,83.54) and (222.31,74.11) .. (210.64,70.46) -- (207.02,78.49) .. controls (215.57,80.74) and (221.64,86.94) .. (221.41,94.03) .. controls (221.13,102.73) and (211.46,109.47) .. (199.82,109.09) .. controls (188.18,108.71) and (178.98,101.35) .. (179.26,92.65) .. controls (179.53,84.45) and (188.13,77.99) .. (198.87,77.6) -- cycle ;
\draw  [color={rgb, 255:red, 0; green, 0; blue, 0 }  ,draw opacity=1 ][fill={rgb, 255:red, 128; green, 128; blue, 128 }  ,fill opacity=1 ][line width=0.75]  (200.4,91.58) .. controls (192.13,90.62) and (186.19,83.24) .. (187.14,75.1) .. controls (188.08,66.95) and (195.55,61.12) .. (203.82,62.08) .. controls (212.09,63.04) and (218.03,70.43) .. (217.08,78.57) .. controls (216.14,86.72) and (208.67,92.54) .. (200.4,91.58) -- cycle ;

\draw (138.5,50.9) node [anchor=north west][inner sep=0.75pt]    {$e$};
\draw (251.25,56.4) node [anchor=north west][inner sep=0.75pt]    {$f$};
\draw (223.5,114.65) node [anchor=north west][inner sep=0.75pt]    {$g$};

\end{tikzpicture}

%% file: Diagrams/Del/11.tex
\tikzset{every picture/.style={line width=0.75pt}} 

\begin{tikzpicture}[x=0.45pt,y=0.45pt,yscale=-1,xscale=1]

\draw  [draw opacity=0][fill={rgb, 255:red, 208; green, 208; blue, 208 }  ,fill opacity=1 ] (227.07,70.5) -- (270.74,100.75) -- (279.75,87.75) -- (236.08,57.5) -- cycle ;
\draw  [draw opacity=0][fill={rgb, 255:red, 208; green, 208; blue, 208 }  ,fill opacity=1 ] (227.07,44.49) -- (270.74,14.24) -- (279.75,27.25) -- (236.08,57.5) -- cycle ;
\draw    (243.45,62.9) -- (279.75,87.75) ;
\draw    (279.75,27.25) -- (243.85,51.8) ;
\draw  [draw opacity=0][fill={rgb, 255:red, 208; green, 208; blue, 208 }  ,fill opacity=1 ] (81.75,77) -- (38.08,107.25) -- (29.07,94.25) -- (72.74,64) -- cycle ;
\draw  [draw opacity=0][fill={rgb, 255:red, 208; green, 208; blue, 208 }  ,fill opacity=1 ] (81.75,50.99) -- (38.08,20.74) -- (29.07,33.75) -- (72.74,64) -- cycle ;
\draw    (65.37,69.4) -- (29.07,94.25) ;
\draw    (29.07,33.75) -- (64.97,58.3) ;
\draw  [fill={rgb, 255:red, 208; green, 208; blue, 208 }  ,fill opacity=1 ] (85.25,48.41) -- (230.5,48.41) -- (230.5,69.41) -- (85.25,69.41) -- cycle ;
\draw  [fill={rgb, 255:red, 128; green, 128; blue, 128 }  ,fill opacity=1 ] (61,60.5) .. controls (61,46.69) and (72.19,35.5) .. (86,35.5) .. controls (99.81,35.5) and (111,46.69) .. (111,60.5) .. controls (111,74.31) and (99.81,85.5) .. (86,85.5) .. controls (72.19,85.5) and (61,74.31) .. (61,60.5) -- cycle ;
\draw  [fill={rgb, 255:red, 128; green, 128; blue, 128 }  ,fill opacity=1 ] (201,56.75) .. controls (201,42.94) and (212.19,31.75) .. (226,31.75) .. controls (239.81,31.75) and (251,42.94) .. (251,56.75) .. controls (251,70.56) and (239.81,81.75) .. (226,81.75) .. controls (212.19,81.75) and (201,70.56) .. (201,56.75) -- cycle ;
\draw [color={rgb, 255:red, 223; green, 83; blue, 107 }  ,draw opacity=1 ][line width=2.25]    (107.9,49.41) -- (202.9,49.41) ;
\draw [color={rgb, 255:red, 40; green, 226; blue, 229 }  ,draw opacity=1 ][line width=2.25]    (108.7,70.21) -- (205.25,70.41) ;
\draw [color={rgb, 255:red, 223; green, 83; blue, 107 }  ,draw opacity=1 ][line width=2.25]    (38.08,20.74) -- (69,43.16) ;
\draw [color={rgb, 255:red, 40; green, 226; blue, 229 }  ,draw opacity=1 ][line width=2.25]    (238.03,77.18) -- (270.74,100.75) ;
\draw [color={rgb, 255:red, 40; green, 226; blue, 229 }  ,draw opacity=1 ][line width=2.25]    (71.62,82.2) -- (38.08,107.25) ;
\draw  [draw opacity=0][line width=2.25]  (109.57,68.85) .. controls (106.13,78.55) and (96.88,85.5) .. (86,85.5) .. controls (80.44,85.5) and (75.3,83.69) .. (71.15,80.62) -- (86,60.5) -- cycle ; \draw  [color={rgb, 255:red, 40; green, 226; blue, 229 }  ,draw opacity=1 ][line width=2.25]  (109.57,68.85) .. controls (106.13,78.55) and (96.88,85.5) .. (86,85.5) .. controls (80.44,85.5) and (75.3,83.69) .. (71.15,80.62) ;  
\draw  [draw opacity=0][line width=2.25]  (67.5,43.68) .. controls (72.07,38.66) and (78.67,35.5) .. (86,35.5) .. controls (96.27,35.5) and (105.09,41.69) .. (108.94,50.55) -- (86,60.5) -- cycle ; \draw  [color={rgb, 255:red, 223; green, 83; blue, 107 }  ,draw opacity=1 ][line width=2.25]  (67.5,43.68) .. controls (72.07,38.66) and (78.67,35.5) .. (86,35.5) .. controls (96.27,35.5) and (105.09,41.69) .. (108.94,50.55) ;  
\draw  [draw opacity=0][line width=2.25]  (201.68,50.95) .. controls (204.29,39.94) and (214.19,31.75) .. (226,31.75) .. controls (231,31.75) and (235.66,33.22) .. (239.57,35.75) -- (226,56.75) -- cycle ; \draw  [color={rgb, 255:red, 223; green, 83; blue, 107 }  ,draw opacity=1 ][line width=2.25]  (201.68,50.95) .. controls (204.29,39.94) and (214.19,31.75) .. (226,31.75) .. controls (231,31.75) and (235.66,33.22) .. (239.57,35.75) ;  
\draw [color={rgb, 255:red, 223; green, 83; blue, 107 }  ,draw opacity=1 ][line width=2.25]  [dash pattern={on 6.75pt off 4.5pt}]  (38.08,20.74) .. controls (74.23,1.49) and (210.25,-1.09) .. (270.74,14.24) ;
\draw  [draw opacity=0][line width=2.25]  (240.24,77.3) .. controls (236.2,80.1) and (231.29,81.75) .. (226,81.75) .. controls (216.45,81.75) and (208.15,76.4) .. (203.94,68.53) -- (226,56.75) -- cycle ; \draw  [color={rgb, 255:red, 40; green, 226; blue, 229 }  ,draw opacity=1 ][line width=2.25]  (240.24,77.3) .. controls (236.2,80.1) and (231.29,81.75) .. (226,81.75) .. controls (216.45,81.75) and (208.15,76.4) .. (203.94,68.53) ;  
\draw [color={rgb, 255:red, 223; green, 83; blue, 107 }  ,draw opacity=1 ][line width=2.25]    (270.74,14.24) -- (242.11,34.94) -- (238.82,37) ;
\draw [color={rgb, 255:red, 40; green, 226; blue, 229 }  ,draw opacity=1 ][line width=2.25]  [dash pattern={on 6.75pt off 4.5pt}]  (38.08,107.25) .. controls (74.23,126.5) and (240.75,124.41) .. (270.74,100.75) ;

\draw (3,15) node [anchor=north west][inner sep=0.75pt]    {$\textcolor[rgb]{0.87,0.33,0.42}{n}$};
\draw (1,95) node [anchor=north west][inner sep=0.75pt]    {$\textcolor[rgb]{0.31,0.89,0.76}{m}$};

\end{tikzpicture}

%% file: Diagrams/Del/12.tex
\tikzset{every picture/.style={line width=0.75pt}} 

\begin{tikzpicture}[x=0.45pt,y=0.45pt,yscale=-1,xscale=1]

\draw  [draw opacity=0][fill={rgb, 255:red, 208; green, 208; blue, 208 }  ,fill opacity=1 ] (84.83,78.75) -- (41.16,109) -- (32.16,96) -- (75.83,65.75) -- cycle ;
\draw    (68.45,71.15) -- (32.16,96) ;
\draw [color={rgb, 255:red, 205; green, 11; blue, 188 }  ,draw opacity=1 ][line width=2.25]    (75.32,84.98) -- (41.16,109) ;
\draw [color={rgb, 255:red, 205; green, 11; blue, 188 }  ,draw opacity=1 ][line width=2.25]  [dash pattern={on 6.75pt off 4.5pt}]  (41.16,109) .. controls (77.31,128.25) and (234.06,127.17) .. (275.16,108.08) ;
\draw  [draw opacity=0][fill={rgb, 255:red, 208; green, 208; blue, 208 }  ,fill opacity=1 ] (84.83,52.74) -- (41.16,22.49) -- (32.16,35.5) -- (75.83,65.75) -- cycle ;
\draw    (32.16,35.5) -- (68.05,60.05) ;
\draw [color={rgb, 255:red, 205; green, 11; blue, 188 }  ,draw opacity=1 ][line width=2.25]    (41.16,22.49) -- (73.67,44.65) ;
\draw [color={rgb, 255:red, 205; green, 11; blue, 188 }  ,draw opacity=1 ][line width=2.25]  [dash pattern={on 6.75pt off 4.5pt}]  (41.16,22.49) .. controls (77.31,3.24) and (234.06,2.49) .. (275.16,21.58) ;
\draw  [draw opacity=0][fill={rgb, 255:red, 208; green, 208; blue, 208 }  ,fill opacity=1 ] (231.49,77.83) -- (275.16,108.08) -- (284.17,95.08) -- (240.5,64.83) -- cycle ;
\draw  [draw opacity=0][fill={rgb, 255:red, 208; green, 208; blue, 208 }  ,fill opacity=1 ] (231.49,51.83) -- (275.16,21.58) -- (284.17,34.58) -- (240.5,64.83) -- cycle ;
\draw    (247.87,70.23) -- (284.17,95.08) ;
\draw    (284.17,34.58) -- (248.27,59.13) ;
\draw  [fill={rgb, 255:red, 128; green, 128; blue, 128 }  ,fill opacity=1 ] (63.92,64.08) .. controls (63.92,50.28) and (75.11,39.08) .. (88.92,39.08) .. controls (102.72,39.08) and (113.92,50.28) .. (113.92,64.08) .. controls (113.92,77.89) and (102.72,89.08) .. (88.92,89.08) .. controls (75.11,89.08) and (63.92,77.89) .. (63.92,64.08) -- cycle ;
\draw  [fill={rgb, 255:red, 128; green, 128; blue, 128 }  ,fill opacity=1 ] (203.92,63.08) .. controls (203.92,49.28) and (215.11,38.08) .. (228.92,38.08) .. controls (242.72,38.08) and (253.92,49.28) .. (253.92,63.08) .. controls (253.92,76.89) and (242.72,88.08) .. (228.92,88.08) .. controls (215.11,88.08) and (203.92,76.89) .. (203.92,63.08) -- cycle ;
\draw [color={rgb, 255:red, 205; green, 11; blue, 188 }  ,draw opacity=1 ][line width=2.25]    (240.67,84.58) -- (275.16,108.08) ;
\draw  [draw opacity=0][line width=2.25]  (72.06,45.62) .. controls (76.51,41.56) and (82.42,39.08) .. (88.92,39.08) .. controls (102.72,39.08) and (113.92,50.28) .. (113.92,64.08) .. controls (113.92,77.89) and (102.72,89.08) .. (88.92,89.08) .. controls (83.16,89.08) and (77.85,87.14) .. (73.63,83.86) -- (88.92,64.08) -- cycle ; \draw  [color={rgb, 255:red, 205; green, 11; blue, 188 }  ,draw opacity=1 ][line width=2.25]  (72.06,45.62) .. controls (76.51,41.56) and (82.42,39.08) .. (88.92,39.08) .. controls (102.72,39.08) and (113.92,50.28) .. (113.92,64.08) .. controls (113.92,77.89) and (102.72,89.08) .. (88.92,89.08) .. controls (83.16,89.08) and (77.85,87.14) .. (73.63,83.86) ;  
\draw  [draw opacity=0][line width=2.25]  (242.23,84.25) .. controls (238.37,86.68) and (233.81,88.08) .. (228.92,88.08) .. controls (215.11,88.08) and (203.92,76.89) .. (203.92,63.08) .. controls (203.92,49.28) and (215.11,38.08) .. (228.92,38.08) .. controls (234.34,38.08) and (239.37,39.81) .. (243.46,42.75) -- (228.92,63.08) -- cycle ; \draw  [color={rgb, 255:red, 205; green, 11; blue, 188 }  ,draw opacity=1 ][line width=2.25]  (242.23,84.25) .. controls (238.37,86.68) and (233.81,88.08) .. (228.92,88.08) .. controls (215.11,88.08) and (203.92,76.89) .. (203.92,63.08) .. controls (203.92,49.28) and (215.11,38.08) .. (228.92,38.08) .. controls (234.34,38.08) and (239.37,39.81) .. (243.46,42.75) ;  
\draw [color={rgb, 255:red, 205; green, 11; blue, 188 }  ,draw opacity=1 ][line width=2.25]    (275.16,21.58) -- (242.42,43.58) ;

\draw (122,55) node [anchor=north west][inner sep=0.75pt]    {$\textcolor[rgb]{0.8,0.04,0.74}{n+m}$};

\end{tikzpicture}

%% file: Diagrams/Del/21.tex
\tikzset{every picture/.style={line width=0.75pt}} 

\begin{tikzpicture}[x=0.45pt,y=0.45pt,yscale=-1,xscale=1]

\draw  [draw opacity=0][fill={rgb, 255:red, 208; green, 208; blue, 208 }  ,fill opacity=1 ] (228.57,71.5) -- (272.24,101.75) -- (281.25,88.75) -- (237.58,58.5) -- cycle ;
\draw  [draw opacity=0][fill={rgb, 255:red, 208; green, 208; blue, 208 }  ,fill opacity=1 ] (228.57,45.49) -- (272.24,15.24) -- (281.25,28.25) -- (237.58,58.5) -- cycle ;
\draw    (244.95,63.9) -- (281.25,88.75) ;
\draw    (281.25,28.25) -- (245.35,52.8) ;
\draw  [draw opacity=0][fill={rgb, 255:red, 208; green, 208; blue, 208 }  ,fill opacity=1 ] (83.25,78) -- (39.58,108.25) -- (30.57,95.25) -- (74.24,65) -- cycle ;
\draw  [draw opacity=0][fill={rgb, 255:red, 208; green, 208; blue, 208 }  ,fill opacity=1 ] (83.25,51.99) -- (39.58,21.74) -- (30.57,34.75) -- (74.24,65) -- cycle ;
\draw    (66.87,70.4) -- (30.57,95.25) ;
\draw    (30.57,34.75) -- (66.47,59.3) ;
\draw  [fill={rgb, 255:red, 208; green, 208; blue, 208 }  ,fill opacity=1 ] (86.75,49.41) -- (232,49.41) -- (232,70.41) -- (86.75,70.41) -- cycle ;
\draw  [fill={rgb, 255:red, 128; green, 128; blue, 128 }  ,fill opacity=1 ] (62.5,61.5) .. controls (62.5,47.69) and (73.69,36.5) .. (87.5,36.5) .. controls (101.31,36.5) and (112.5,47.69) .. (112.5,61.5) .. controls (112.5,75.31) and (101.31,86.5) .. (87.5,86.5) .. controls (73.69,86.5) and (62.5,75.31) .. (62.5,61.5) -- cycle ;
\draw  [fill={rgb, 255:red, 128; green, 128; blue, 128 }  ,fill opacity=1 ] (202.5,57.75) .. controls (202.5,43.94) and (213.69,32.75) .. (227.5,32.75) .. controls (241.31,32.75) and (252.5,43.94) .. (252.5,57.75) .. controls (252.5,71.56) and (241.31,82.75) .. (227.5,82.75) .. controls (213.69,82.75) and (202.5,71.56) .. (202.5,57.75) -- cycle ;
\draw [color={rgb, 255:red, 223; green, 83; blue, 107 }  ,draw opacity=1 ][line width=2.25]    (109.4,50.41) -- (204.4,50.41) ;
\draw [color={rgb, 255:red, 223; green, 83; blue, 107 }  ,draw opacity=1 ][line width=2.25]    (110.2,71.21) -- (206.5,71.16) ;
\draw [color={rgb, 255:red, 223; green, 83; blue, 107 }  ,draw opacity=1 ][line width=2.25]    (39.58,21.74) -- (70.5,44.16) ;
\draw [color={rgb, 255:red, 223; green, 83; blue, 107 }  ,draw opacity=1 ][line width=2.25]    (239.53,78.18) -- (272.24,101.75) ;
\draw [color={rgb, 255:red, 223; green, 83; blue, 107 }  ,draw opacity=1 ][line width=2.25]    (73.12,83.2) -- (39.58,108.25) ;
\draw  [draw opacity=0][line width=2.25]  (111.07,69.85) .. controls (107.63,79.55) and (98.38,86.5) .. (87.5,86.5) .. controls (81.94,86.5) and (76.8,84.69) .. (72.65,81.62) -- (87.5,61.5) -- cycle ; \draw  [color={rgb, 255:red, 223; green, 83; blue, 107 }  ,draw opacity=1 ][line width=2.25]  (111.07,69.85) .. controls (107.63,79.55) and (98.38,86.5) .. (87.5,86.5) .. controls (81.94,86.5) and (76.8,84.69) .. (72.65,81.62) ;  
\draw  [draw opacity=0][line width=2.25]  (69,44.68) .. controls (73.57,39.66) and (80.17,36.5) .. (87.5,36.5) .. controls (97.77,36.5) and (106.59,42.69) .. (110.44,51.55) -- (87.5,61.5) -- cycle ; \draw  [color={rgb, 255:red, 223; green, 83; blue, 107 }  ,draw opacity=1 ][line width=2.25]  (69,44.68) .. controls (73.57,39.66) and (80.17,36.5) .. (87.5,36.5) .. controls (97.77,36.5) and (106.59,42.69) .. (110.44,51.55) ;  
\draw  [draw opacity=0][line width=2.25]  (203.29,51.51) .. controls (206.06,40.72) and (215.85,32.75) .. (227.5,32.75) .. controls (232.5,32.75) and (237.16,34.22) .. (241.07,36.75) -- (227.5,57.75) -- cycle ; \draw  [color={rgb, 255:red, 223; green, 83; blue, 107 }  ,draw opacity=1 ][line width=2.25]  (203.29,51.51) .. controls (206.06,40.72) and (215.85,32.75) .. (227.5,32.75) .. controls (232.5,32.75) and (237.16,34.22) .. (241.07,36.75) ;  
\draw [color={rgb, 255:red, 223; green, 83; blue, 107 }  ,draw opacity=1 ][line width=2.25]  [dash pattern={on 6.75pt off 4.5pt}]  (39.58,21.74) .. controls (75.73,2.49) and (211.75,-0.09) .. (272.24,15.24) ;
\draw  [draw opacity=0][line width=2.25]  (241.74,78.3) .. controls (237.7,81.1) and (232.79,82.75) .. (227.5,82.75) .. controls (218.15,82.75) and (210.01,77.62) .. (205.72,70.02) -- (227.5,57.75) -- cycle ; \draw  [color={rgb, 255:red, 223; green, 83; blue, 107 }  ,draw opacity=1 ][line width=2.25]  (241.74,78.3) .. controls (237.7,81.1) and (232.79,82.75) .. (227.5,82.75) .. controls (218.15,82.75) and (210.01,77.62) .. (205.72,70.02) ;  
\draw [color={rgb, 255:red, 223; green, 83; blue, 107 }  ,draw opacity=1 ][line width=2.25]    (272.24,15.24) -- (243.61,35.94) -- (240.32,38) ;
\draw [color={rgb, 255:red, 223; green, 83; blue, 107 }  ,draw opacity=1 ][line width=2.25]  [dash pattern={on 6.75pt off 4.5pt}]  (39.58,108.25) .. controls (75.73,127.5) and (242.25,125.41) .. (272.24,101.75) ;

\draw (13,55.9) node [anchor=north west][inner sep=0.75pt]    {$\textcolor[rgb]{0.87,0.33,0.42}{n}$};

\end{tikzpicture}

%% file: Diagrams/Del/22.tex
\tikzset{every picture/.style={line width=0.75pt}} 

\begin{tikzpicture}[x=0.45pt,y=0.45pt,yscale=-1,xscale=1]

\draw  [draw opacity=0][fill={rgb, 255:red, 208; green, 208; blue, 208 }  ,fill opacity=1 ] (83.17,75.58) -- (39.5,105.83) -- (30.49,92.83) -- (74.16,62.58) -- cycle ;
\draw    (66.78,67.98) -- (30.49,92.83) ;
\draw [color={rgb, 255:red, 223; green, 83; blue, 107 }  ,draw opacity=1 ][line width=2.25]    (73.65,81.82) -- (39.5,105.83) ;
\draw [color={rgb, 255:red, 223; green, 83; blue, 107 }  ,draw opacity=1 ][line width=2.25]  [dash pattern={on 6.75pt off 4.5pt}]  (39.5,105.83) .. controls (75.65,125.08) and (232.39,124) .. (273.49,104.92) ;
\draw  [draw opacity=0][fill={rgb, 255:red, 208; green, 208; blue, 208 }  ,fill opacity=1 ] (83.17,49.58) -- (39.5,19.33) -- (30.49,32.33) -- (74.16,62.58) -- cycle ;
\draw    (30.49,32.33) -- (66.38,56.88) ;
\draw [color={rgb, 255:red, 223; green, 83; blue, 107 }  ,draw opacity=1 ][line width=2.25]    (39.5,19.33) -- (72,41.48) ;
\draw [color={rgb, 255:red, 223; green, 83; blue, 107 }  ,draw opacity=1 ][line width=2.25]  [dash pattern={on 6.75pt off 4.5pt}]  (39.5,19.33) .. controls (75.65,0.08) and (232.39,-0.68) .. (273.49,18.41) ;
\draw  [draw opacity=0][fill={rgb, 255:red, 208; green, 208; blue, 208 }  ,fill opacity=1 ] (229.82,74.67) -- (273.49,104.92) -- (282.5,91.91) -- (238.83,61.66) -- cycle ;
\draw  [draw opacity=0][fill={rgb, 255:red, 208; green, 208; blue, 208 }  ,fill opacity=1 ] (229.82,48.66) -- (273.49,18.41) -- (282.5,31.41) -- (238.83,61.66) -- cycle ;
\draw    (246.2,67.06) -- (282.5,91.91) ;
\draw    (282.5,31.41) -- (246.6,55.96) ;
\draw  [fill={rgb, 255:red, 128; green, 128; blue, 128 }  ,fill opacity=1 ] (62.25,60.92) .. controls (62.25,47.11) and (73.44,35.92) .. (87.25,35.92) .. controls (101.06,35.92) and (112.25,47.11) .. (112.25,60.92) .. controls (112.25,74.72) and (101.06,85.92) .. (87.25,85.92) .. controls (73.44,85.92) and (62.25,74.72) .. (62.25,60.92) -- cycle ;
\draw  [fill={rgb, 255:red, 128; green, 128; blue, 128 }  ,fill opacity=1 ] (202.25,59.92) .. controls (202.25,46.11) and (213.44,34.92) .. (227.25,34.92) .. controls (241.06,34.92) and (252.25,46.11) .. (252.25,59.92) .. controls (252.25,73.72) and (241.06,84.92) .. (227.25,84.92) .. controls (213.44,84.92) and (202.25,73.72) .. (202.25,59.92) -- cycle ;
\draw [color={rgb, 255:red, 223; green, 83; blue, 107 }  ,draw opacity=1 ][line width=2.25]    (239,81.41) -- (273.49,104.92) ;
\draw  [draw opacity=0][line width=2.25]  (70.4,42.45) .. controls (74.84,38.39) and (80.76,35.92) .. (87.25,35.92) .. controls (101.06,35.92) and (112.25,47.11) .. (112.25,60.92) .. controls (112.25,74.72) and (101.06,85.92) .. (87.25,85.92) .. controls (81.49,85.92) and (76.19,83.97) .. (71.96,80.7) -- (87.25,60.92) -- cycle ; \draw  [color={rgb, 255:red, 223; green, 83; blue, 107 }  ,draw opacity=1 ][line width=2.25]  (70.4,42.45) .. controls (74.84,38.39) and (80.76,35.92) .. (87.25,35.92) .. controls (101.06,35.92) and (112.25,47.11) .. (112.25,60.92) .. controls (112.25,74.72) and (101.06,85.92) .. (87.25,85.92) .. controls (81.49,85.92) and (76.19,83.97) .. (71.96,80.7) ;  
\draw  [draw opacity=0][line width=2.25]  (240.56,81.08) .. controls (236.71,83.51) and (232.14,84.92) .. (227.25,84.92) .. controls (213.44,84.92) and (202.25,73.72) .. (202.25,59.92) .. controls (202.25,46.11) and (213.44,34.92) .. (227.25,34.92) .. controls (232.68,34.92) and (237.7,36.65) .. (241.8,39.58) -- (227.25,59.92) -- cycle ; \draw  [color={rgb, 255:red, 223; green, 83; blue, 107 }  ,draw opacity=1 ][line width=2.25]  (240.56,81.08) .. controls (236.71,83.51) and (232.14,84.92) .. (227.25,84.92) .. controls (213.44,84.92) and (202.25,73.72) .. (202.25,59.92) .. controls (202.25,46.11) and (213.44,34.92) .. (227.25,34.92) .. controls (232.68,34.92) and (237.7,36.65) .. (241.8,39.58) ;  
\draw [color={rgb, 255:red, 223; green, 83; blue, 107 }  ,draw opacity=1 ][line width=2.25]    (273.49,18.41) -- (240.75,40.41) ;

\draw (124,55) node [anchor=north west][inner sep=0.75pt]    {$\textcolor[rgb]{0.87,0.33,0.42}{n+1}$};

\end{tikzpicture}

%% file: Diagrams/Del/31.tex
\tikzset{every picture/.style={line width=0.75pt}} 

\begin{tikzpicture}[x=0.45pt,y=0.45pt,yscale=-1,xscale=1]

\draw [color={rgb, 255:red, 223; green, 83; blue, 107 }  ,draw opacity=1 ][line width=2.25]  [dash pattern={on 6.75pt off 4.5pt}]  (37.58,13.99) .. controls (-2.42,24.44) and (-2.08,88.62) .. (37.58,100.5) ;
\draw  [draw opacity=0][fill={rgb, 255:red, 208; green, 208; blue, 208 }  ,fill opacity=1 ] (226.57,63.75) -- (270.24,94) -- (279.25,81) -- (235.58,50.75) -- cycle ;
\draw  [draw opacity=0][fill={rgb, 255:red, 208; green, 208; blue, 208 }  ,fill opacity=1 ] (226.57,37.74) -- (270.24,7.49) -- (279.25,20.5) -- (235.58,50.75) -- cycle ;
\draw    (242.95,56.15) -- (279.25,81) ;
\draw    (279.25,20.5) -- (243.35,45.05) ;
\draw  [draw opacity=0][fill={rgb, 255:red, 208; green, 208; blue, 208 }  ,fill opacity=1 ] (81.25,70.25) -- (37.58,100.5) -- (28.57,87.5) -- (72.24,57.25) -- cycle ;
\draw  [draw opacity=0][fill={rgb, 255:red, 208; green, 208; blue, 208 }  ,fill opacity=1 ] (81.25,44.24) -- (37.58,13.99) -- (28.57,27) -- (72.24,57.25) -- cycle ;
\draw    (64.87,62.65) -- (28.57,87.5) ;
\draw    (28.57,27) -- (64.47,51.55) ;
\draw  [fill={rgb, 255:red, 208; green, 208; blue, 208 }  ,fill opacity=1 ] (84.75,41.66) -- (230,41.66) -- (230,62.66) -- (84.75,62.66) -- cycle ;
\draw  [fill={rgb, 255:red, 128; green, 128; blue, 128 }  ,fill opacity=1 ] (60.5,53.75) .. controls (60.5,39.94) and (71.69,28.75) .. (85.5,28.75) .. controls (99.31,28.75) and (110.5,39.94) .. (110.5,53.75) .. controls (110.5,67.56) and (99.31,78.75) .. (85.5,78.75) .. controls (71.69,78.75) and (60.5,67.56) .. (60.5,53.75) -- cycle ;
\draw  [fill={rgb, 255:red, 128; green, 128; blue, 128 }  ,fill opacity=1 ] (200.5,50) .. controls (200.5,36.19) and (211.69,25) .. (225.5,25) .. controls (239.31,25) and (250.5,36.19) .. (250.5,50) .. controls (250.5,63.81) and (239.31,75) .. (225.5,75) .. controls (211.69,75) and (200.5,63.81) .. (200.5,50) -- cycle ;
\draw [color={rgb, 255:red, 223; green, 83; blue, 107 }  ,draw opacity=1 ][line width=2.25]    (107.4,42.66) -- (202.4,42.66) ;
\draw [color={rgb, 255:red, 223; green, 83; blue, 107 }  ,draw opacity=1 ][line width=2.25]    (108.2,63.46) -- (204.5,63.41) ;
\draw [color={rgb, 255:red, 223; green, 83; blue, 107 }  ,draw opacity=1 ][line width=2.25]    (37.58,13.99) -- (68.5,36.41) ;
\draw [color={rgb, 255:red, 223; green, 83; blue, 107 }  ,draw opacity=1 ][line width=2.25]    (237.53,70.43) -- (270.24,94) ;
\draw [color={rgb, 255:red, 223; green, 83; blue, 107 }  ,draw opacity=1 ][line width=2.25]    (71.12,75.45) -- (37.58,100.5) ;
\draw  [draw opacity=0][line width=2.25]  (109.07,62.1) .. controls (105.63,71.8) and (96.38,78.75) .. (85.5,78.75) .. controls (79.94,78.75) and (74.8,76.94) .. (70.65,73.87) -- (85.5,53.75) -- cycle ; \draw  [color={rgb, 255:red, 223; green, 83; blue, 107 }  ,draw opacity=1 ][line width=2.25]  (109.07,62.1) .. controls (105.63,71.8) and (96.38,78.75) .. (85.5,78.75) .. controls (79.94,78.75) and (74.8,76.94) .. (70.65,73.87) ;  
\draw  [draw opacity=0][line width=2.25]  (67,36.93) .. controls (71.57,31.91) and (78.17,28.75) .. (85.5,28.75) .. controls (95.77,28.75) and (104.59,34.94) .. (108.44,43.8) -- (85.5,53.75) -- cycle ; \draw  [color={rgb, 255:red, 223; green, 83; blue, 107 }  ,draw opacity=1 ][line width=2.25]  (67,36.93) .. controls (71.57,31.91) and (78.17,28.75) .. (85.5,28.75) .. controls (95.77,28.75) and (104.59,34.94) .. (108.44,43.8) ;  
\draw  [draw opacity=0][line width=2.25]  (201.29,43.76) .. controls (204.06,32.97) and (213.85,25) .. (225.5,25) .. controls (230.5,25) and (235.16,26.47) .. (239.07,29) -- (225.5,50) -- cycle ; \draw  [color={rgb, 255:red, 223; green, 83; blue, 107 }  ,draw opacity=1 ][line width=2.25]  (201.29,43.76) .. controls (204.06,32.97) and (213.85,25) .. (225.5,25) .. controls (230.5,25) and (235.16,26.47) .. (239.07,29) ;  
\draw  [draw opacity=0][line width=2.25]  (239.74,70.55) .. controls (235.7,73.35) and (230.79,75) .. (225.5,75) .. controls (216.15,75) and (208.01,69.87) .. (203.72,62.27) -- (225.5,50) -- cycle ; \draw  [color={rgb, 255:red, 223; green, 83; blue, 107 }  ,draw opacity=1 ][line width=2.25]  (239.74,70.55) .. controls (235.7,73.35) and (230.79,75) .. (225.5,75) .. controls (216.15,75) and (208.01,69.87) .. (203.72,62.27) ;  
\draw [color={rgb, 255:red, 223; green, 83; blue, 107 }  ,draw opacity=1 ][line width=2.25]    (270.24,7.49) -- (241.61,28.19) -- (238.32,30.25) ;
\draw [color={rgb, 255:red, 223; green, 83; blue, 107 }  ,draw opacity=1 ][line width=2.25]  [dash pattern={on 6.75pt off 4.5pt}]  (270.24,7.49) .. controls (310.24,17.94) and (309.9,82.12) .. (270.24,94) ;

\draw (149,79) node [anchor=north west][inner sep=0.75pt]    {$\textcolor[rgb]{0.87,0.33,0.42}{n}$};

\end{tikzpicture}

%% file: Diagrams/Del/32.tex
\tikzset{every picture/.style={line width=0.75pt}} 

\begin{tikzpicture}[x=0.45pt,y=0.45pt,yscale=-1,xscale=1]

\draw [color={rgb, 255:red, 223; green, 83; blue, 107 }  ,draw opacity=1 ][line width=2.25]  [dash pattern={on 6.75pt off 4.5pt}]  (37.5,13) .. controls (-2.5,23.44) and (-2.16,87.63) .. (37.5,99.51) ;
\draw [color={rgb, 255:red, 223; green, 83; blue, 107 }  ,draw opacity=1 ][line width=2.25]  [dash pattern={on 6.75pt off 4.5pt}]  (271.49,12.08) .. controls (311.49,22.53) and (311.15,86.71) .. (271.49,98.59) ;
\draw  [draw opacity=0][fill={rgb, 255:red, 208; green, 208; blue, 208 }  ,fill opacity=1 ] (81.17,69.26) -- (37.5,99.51) -- (28.49,86.5) -- (72.16,56.25) -- cycle ;
\draw    (64.78,61.65) -- (28.49,86.5) ;
\draw [color={rgb, 255:red, 223; green, 83; blue, 107 }  ,draw opacity=1 ][line width=2.25]    (71.65,75.49) -- (37.5,99.51) ;
\draw  [draw opacity=0][fill={rgb, 255:red, 208; green, 208; blue, 208 }  ,fill opacity=1 ] (81.17,43.25) -- (37.5,13) -- (28.49,26) -- (72.16,56.25) -- cycle ;
\draw    (28.49,26) -- (64.38,50.55) ;
\draw [color={rgb, 255:red, 223; green, 83; blue, 107 }  ,draw opacity=1 ][line width=2.25]    (37.5,13) -- (70,35.15) ;
\draw  [draw opacity=0][fill={rgb, 255:red, 208; green, 208; blue, 208 }  ,fill opacity=1 ] (227.82,68.34) -- (271.49,98.59) -- (280.5,85.59) -- (236.83,55.34) -- cycle ;
\draw  [draw opacity=0][fill={rgb, 255:red, 208; green, 208; blue, 208 }  ,fill opacity=1 ] (227.82,42.33) -- (271.49,12.08) -- (280.5,25.09) -- (236.83,55.34) -- cycle ;
\draw    (244.2,60.74) -- (280.5,85.59) ;
\draw    (280.5,25.09) -- (244.6,49.64) ;
\draw  [fill={rgb, 255:red, 128; green, 128; blue, 128 }  ,fill opacity=1 ] (60.25,54.59) .. controls (60.25,40.78) and (71.44,29.59) .. (85.25,29.59) .. controls (99.06,29.59) and (110.25,40.78) .. (110.25,54.59) .. controls (110.25,68.4) and (99.06,79.59) .. (85.25,79.59) .. controls (71.44,79.59) and (60.25,68.4) .. (60.25,54.59) -- cycle ;
\draw  [fill={rgb, 255:red, 128; green, 128; blue, 128 }  ,fill opacity=1 ] (200.25,53.59) .. controls (200.25,39.78) and (211.44,28.59) .. (225.25,28.59) .. controls (239.06,28.59) and (250.25,39.78) .. (250.25,53.59) .. controls (250.25,67.4) and (239.06,78.59) .. (225.25,78.59) .. controls (211.44,78.59) and (200.25,67.4) .. (200.25,53.59) -- cycle ;
\draw [color={rgb, 255:red, 223; green, 83; blue, 107 }  ,draw opacity=1 ][line width=2.25]    (237,75.09) -- (271.49,98.59) ;
\draw  [draw opacity=0][line width=2.25]  (68.4,36.13) .. controls (72.84,32.07) and (78.76,29.59) .. (85.25,29.59) .. controls (99.06,29.59) and (110.25,40.78) .. (110.25,54.59) .. controls (110.25,68.4) and (99.06,79.59) .. (85.25,79.59) .. controls (79.49,79.59) and (74.19,77.64) .. (69.96,74.37) -- (85.25,54.59) -- cycle ; \draw  [color={rgb, 255:red, 223; green, 83; blue, 107 }  ,draw opacity=1 ][line width=2.25]  (68.4,36.13) .. controls (72.84,32.07) and (78.76,29.59) .. (85.25,29.59) .. controls (99.06,29.59) and (110.25,40.78) .. (110.25,54.59) .. controls (110.25,68.4) and (99.06,79.59) .. (85.25,79.59) .. controls (79.49,79.59) and (74.19,77.64) .. (69.96,74.37) ;  
\draw  [draw opacity=0][line width=2.25]  (238.56,74.76) .. controls (234.71,77.19) and (230.14,78.59) .. (225.25,78.59) .. controls (211.44,78.59) and (200.25,67.4) .. (200.25,53.59) .. controls (200.25,39.78) and (211.44,28.59) .. (225.25,28.59) .. controls (230.68,28.59) and (235.7,30.32) .. (239.8,33.26) -- (225.25,53.59) -- cycle ; \draw  [color={rgb, 255:red, 223; green, 83; blue, 107 }  ,draw opacity=1 ][line width=2.25]  (238.56,74.76) .. controls (234.71,77.19) and (230.14,78.59) .. (225.25,78.59) .. controls (211.44,78.59) and (200.25,67.4) .. (200.25,53.59) .. controls (200.25,39.78) and (211.44,28.59) .. (225.25,28.59) .. controls (230.68,28.59) and (235.7,30.32) .. (239.8,33.26) ;  
\draw [color={rgb, 255:red, 223; green, 83; blue, 107 }  ,draw opacity=1 ][line width=2.25]    (271.49,12.08) -- (238.75,34.09) ;

\draw (124,50) node [anchor=north west][inner sep=0.75pt]    {$\textcolor[rgb]{0.87,0.33,0.42}{n+1}$};

\end{tikzpicture}

%% file: Diagrams/Del/41.tex
\tikzset{every picture/.style={line width=0.75pt}} 

\begin{tikzpicture}[x=0.45pt,y=0.45pt,yscale=-1,xscale=1]

\draw [color={rgb, 255:red, 223; green, 83; blue, 107 }  ,draw opacity=1 ][line width=2.25]  [dash pattern={on 6.75pt off 4.5pt}]  (38.58,16.99) .. controls (-13.5,7.25) and (-0.5,112.25) .. (42.5,141.25) .. controls (85.5,170.25) and (244.5,131.25) .. (271.24,97) ;
\draw  [draw opacity=0][fill={rgb, 255:red, 208; green, 208; blue, 208 }  ,fill opacity=1 ] (227.57,66.75) -- (271.24,97) -- (280.25,84) -- (236.58,53.75) -- cycle ;
\draw  [draw opacity=0][fill={rgb, 255:red, 208; green, 208; blue, 208 }  ,fill opacity=1 ] (227.57,40.74) -- (271.24,10.49) -- (280.25,23.5) -- (236.58,53.75) -- cycle ;
\draw    (243.95,59.15) -- (280.25,84) ;
\draw    (280.25,23.5) -- (244.35,48.05) ;
\draw  [draw opacity=0][fill={rgb, 255:red, 208; green, 208; blue, 208 }  ,fill opacity=1 ] (82.25,73.25) -- (38.58,103.5) -- (29.57,90.5) -- (73.24,60.25) -- cycle ;
\draw  [draw opacity=0][fill={rgb, 255:red, 208; green, 208; blue, 208 }  ,fill opacity=1 ] (82.25,47.24) -- (38.58,16.99) -- (29.57,30) -- (73.24,60.25) -- cycle ;
\draw    (65.87,65.65) -- (29.57,90.5) ;
\draw    (29.57,30) -- (65.47,54.55) ;
\draw  [fill={rgb, 255:red, 208; green, 208; blue, 208 }  ,fill opacity=1 ] (85.75,44.66) -- (231,44.66) -- (231,65.66) -- (85.75,65.66) -- cycle ;
\draw  [fill={rgb, 255:red, 128; green, 128; blue, 128 }  ,fill opacity=1 ] (61.5,56.75) .. controls (61.5,42.94) and (72.69,31.75) .. (86.5,31.75) .. controls (100.31,31.75) and (111.5,42.94) .. (111.5,56.75) .. controls (111.5,70.56) and (100.31,81.75) .. (86.5,81.75) .. controls (72.69,81.75) and (61.5,70.56) .. (61.5,56.75) -- cycle ;
\draw  [fill={rgb, 255:red, 128; green, 128; blue, 128 }  ,fill opacity=1 ] (201.5,53) .. controls (201.5,39.19) and (212.69,28) .. (226.5,28) .. controls (240.31,28) and (251.5,39.19) .. (251.5,53) .. controls (251.5,66.81) and (240.31,78) .. (226.5,78) .. controls (212.69,78) and (201.5,66.81) .. (201.5,53) -- cycle ;
\draw [color={rgb, 255:red, 223; green, 83; blue, 107 }  ,draw opacity=1 ][line width=2.25]    (108.4,45.66) -- (203.4,45.66) ;
\draw [color={rgb, 255:red, 223; green, 83; blue, 107 }  ,draw opacity=1 ][line width=2.25]    (109.2,66.46) -- (205.5,66.41) ;
\draw [color={rgb, 255:red, 223; green, 83; blue, 107 }  ,draw opacity=1 ][line width=2.25]    (38.58,16.99) -- (69.5,39.41) ;
\draw [color={rgb, 255:red, 223; green, 83; blue, 107 }  ,draw opacity=1 ][line width=2.25]    (238.53,73.43) -- (271.24,97) ;
\draw [color={rgb, 255:red, 223; green, 83; blue, 107 }  ,draw opacity=1 ][line width=2.25]    (72.12,78.45) -- (38.58,103.5) ;
\draw  [draw opacity=0][line width=2.25]  (110.07,65.1) .. controls (106.63,74.8) and (97.38,81.75) .. (86.5,81.75) .. controls (80.94,81.75) and (75.8,79.94) .. (71.65,76.87) -- (86.5,56.75) -- cycle ; \draw  [color={rgb, 255:red, 223; green, 83; blue, 107 }  ,draw opacity=1 ][line width=2.25]  (110.07,65.1) .. controls (106.63,74.8) and (97.38,81.75) .. (86.5,81.75) .. controls (80.94,81.75) and (75.8,79.94) .. (71.65,76.87) ;  
\draw  [draw opacity=0][line width=2.25]  (68,39.93) .. controls (72.57,34.91) and (79.17,31.75) .. (86.5,31.75) .. controls (96.77,31.75) and (105.59,37.94) .. (109.44,46.8) -- (86.5,56.75) -- cycle ; \draw  [color={rgb, 255:red, 223; green, 83; blue, 107 }  ,draw opacity=1 ][line width=2.25]  (68,39.93) .. controls (72.57,34.91) and (79.17,31.75) .. (86.5,31.75) .. controls (96.77,31.75) and (105.59,37.94) .. (109.44,46.8) ;  
\draw  [draw opacity=0][line width=2.25]  (202.29,46.76) .. controls (205.06,35.97) and (214.85,28) .. (226.5,28) .. controls (231.5,28) and (236.16,29.47) .. (240.07,32) -- (226.5,53) -- cycle ; \draw  [color={rgb, 255:red, 223; green, 83; blue, 107 }  ,draw opacity=1 ][line width=2.25]  (202.29,46.76) .. controls (205.06,35.97) and (214.85,28) .. (226.5,28) .. controls (231.5,28) and (236.16,29.47) .. (240.07,32) ;  
\draw  [draw opacity=0][line width=2.25]  (240.74,73.55) .. controls (236.7,76.35) and (231.79,78) .. (226.5,78) .. controls (217.15,78) and (209.01,72.87) .. (204.72,65.27) -- (226.5,53) -- cycle ; \draw  [color={rgb, 255:red, 223; green, 83; blue, 107 }  ,draw opacity=1 ][line width=2.25]  (240.74,73.55) .. controls (236.7,76.35) and (231.79,78) .. (226.5,78) .. controls (217.15,78) and (209.01,72.87) .. (204.72,65.27) ;  
\draw [color={rgb, 255:red, 223; green, 83; blue, 107 }  ,draw opacity=1 ][line width=2.25]    (271.24,10.49) -- (242.61,31.19) -- (239.32,33.25) ;
\draw [color={rgb, 255:red, 223; green, 83; blue, 107 }  ,draw opacity=1 ][line width=2.25]  [dash pattern={on 6.75pt off 4.5pt}]  (271.24,10.49) .. controls (323.5,-6.75) and (312.5,112.25) .. (269.5,138.25) .. controls (226.5,164.25) and (96.5,136.25) .. (38.58,103.5) ;

\draw (149,79) node [anchor=north west][inner sep=0.75pt]    {$\textcolor[rgb]{0.87,0.33,0.42}{n}$};

\end{tikzpicture}

%% file: Diagrams/Del/42.tex
\tikzset{every picture/.style={line width=0.75pt}} 

\begin{tikzpicture}[x=0.45pt,y=0.45pt,yscale=-1,xscale=1]

\draw [color={rgb, 255:red, 223; green, 83; blue, 107 }  ,draw opacity=1 ][line width=2.25]  [dash pattern={on 6.75pt off 4.5pt}]  (38.58,16.99) .. controls (-13.5,7.25) and (-0.5,112.25) .. (42.5,141.25) .. controls (85.5,170.25) and (244.5,131.25) .. (271.24,97) ;
\draw  [draw opacity=0][fill={rgb, 255:red, 208; green, 208; blue, 208 }  ,fill opacity=1 ] (227.57,66.75) -- (271.24,97) -- (280.25,84) -- (236.58,53.75) -- cycle ;
\draw  [draw opacity=0][fill={rgb, 255:red, 208; green, 208; blue, 208 }  ,fill opacity=1 ] (227.57,40.74) -- (271.24,10.49) -- (280.25,23.5) -- (236.58,53.75) -- cycle ;
\draw    (243.95,59.15) -- (280.25,84) ;
\draw    (280.25,23.5) -- (244.35,48.05) ;
\draw  [draw opacity=0][fill={rgb, 255:red, 208; green, 208; blue, 208 }  ,fill opacity=1 ] (82.25,73.25) -- (38.58,103.5) -- (29.57,90.5) -- (73.24,60.25) -- cycle ;
\draw  [draw opacity=0][fill={rgb, 255:red, 208; green, 208; blue, 208 }  ,fill opacity=1 ] (82.25,47.24) -- (38.58,16.99) -- (29.57,30) -- (73.24,60.25) -- cycle ;
\draw    (65.87,65.65) -- (29.57,90.5) ;
\draw    (29.57,30) -- (65.47,54.55) ;
\draw  [fill={rgb, 255:red, 128; green, 128; blue, 128 }  ,fill opacity=1 ] (61.5,56.75) .. controls (61.5,42.94) and (72.69,31.75) .. (86.5,31.75) .. controls (100.31,31.75) and (111.5,42.94) .. (111.5,56.75) .. controls (111.5,70.56) and (100.31,81.75) .. (86.5,81.75) .. controls (72.69,81.75) and (61.5,70.56) .. (61.5,56.75) -- cycle ;
\draw  [fill={rgb, 255:red, 128; green, 128; blue, 128 }  ,fill opacity=1 ] (201.5,53) .. controls (201.5,39.19) and (212.69,28) .. (226.5,28) .. controls (240.31,28) and (251.5,39.19) .. (251.5,53) .. controls (251.5,66.81) and (240.31,78) .. (226.5,78) .. controls (212.69,78) and (201.5,66.81) .. (201.5,53) -- cycle ;
\draw [color={rgb, 255:red, 223; green, 83; blue, 107 }  ,draw opacity=1 ][line width=2.25]    (38.58,16.99) -- (69.5,39.41) ;
\draw [color={rgb, 255:red, 223; green, 83; blue, 107 }  ,draw opacity=1 ][line width=2.25]    (238.53,73.43) -- (271.24,97) ;
\draw [color={rgb, 255:red, 223; green, 83; blue, 107 }  ,draw opacity=1 ][line width=2.25]    (72.12,78.45) -- (38.58,103.5) ;
\draw  [draw opacity=0][line width=2.25]  (68,39.93) .. controls (72.57,34.91) and (79.17,31.75) .. (86.5,31.75) .. controls (100.31,31.75) and (111.5,42.94) .. (111.5,56.75) .. controls (111.5,70.56) and (100.31,81.75) .. (86.5,81.75) .. controls (80.85,81.75) and (75.65,79.88) .. (71.46,76.72) -- (86.5,56.75) -- cycle ; \draw  [color={rgb, 255:red, 223; green, 83; blue, 107 }  ,draw opacity=1 ][line width=2.25]  (68,39.93) .. controls (72.57,34.91) and (79.17,31.75) .. (86.5,31.75) .. controls (100.31,31.75) and (111.5,42.94) .. (111.5,56.75) .. controls (111.5,70.56) and (100.31,81.75) .. (86.5,81.75) .. controls (80.85,81.75) and (75.65,79.88) .. (71.46,76.72) ;  
\draw  [draw opacity=0][line width=2.25]  (240.74,73.55) .. controls (236.7,76.35) and (231.79,78) .. (226.5,78) .. controls (212.69,78) and (201.5,66.81) .. (201.5,53) .. controls (201.5,39.19) and (212.69,28) .. (226.5,28) .. controls (231.78,28) and (236.67,29.63) .. (240.7,32.42) -- (226.5,53) -- cycle ; \draw  [color={rgb, 255:red, 223; green, 83; blue, 107 }  ,draw opacity=1 ][line width=2.25]  (240.74,73.55) .. controls (236.7,76.35) and (231.79,78) .. (226.5,78) .. controls (212.69,78) and (201.5,66.81) .. (201.5,53) .. controls (201.5,39.19) and (212.69,28) .. (226.5,28) .. controls (231.78,28) and (236.67,29.63) .. (240.7,32.42) ;  
\draw [color={rgb, 255:red, 223; green, 83; blue, 107 }  ,draw opacity=1 ][line width=2.25]    (271.24,10.49) -- (242.61,31.19) -- (239.32,33.25) ;
\draw [color={rgb, 255:red, 223; green, 83; blue, 107 }  ,draw opacity=1 ][line width=2.25]  [dash pattern={on 6.75pt off 4.5pt}]  (271.24,10.49) .. controls (323.5,-6.75) and (312.5,112.25) .. (269.5,138.25) .. controls (226.5,164.25) and (96.5,136.25) .. (38.58,103.5) ;

\draw (124,55) node [anchor=north west][inner sep=0.75pt]    {$\textcolor[rgb]{0.87,0.33,0.42}{n+1}$};

\end{tikzpicture}

%% file: Diagrams/Con/11.tex
\tikzset{every picture/.style={line width=0.75pt}} 

\begin{tikzpicture}[x=0.45pt,y=0.45pt,yscale=-1,xscale=1]

\draw  [draw opacity=0][fill={rgb, 255:red, 208; green, 208; blue, 208 }  ,fill opacity=1 ] (76.75,45.59) -- (23.64,47.04) -- (24.07,62.85) -- (77.18,61.4) -- cycle ;
\draw  [draw opacity=0][fill={rgb, 255:red, 208; green, 208; blue, 208 }  ,fill opacity=1 ] (87.82,66.54) -- (50.87,104.71) -- (39.5,93.71) -- (76.45,55.54) -- cycle ;
\draw  [draw opacity=0][fill={rgb, 255:red, 208; green, 208; blue, 208 }  ,fill opacity=1 ] (89.12,38.83) -- (48.34,4.77) -- (38.2,16.91) -- (78.98,50.97) -- cycle ;
\draw    (81.39,73.37) -- (50.87,104.71) ;
\draw    (70.27,62.27) -- (39.5,93.71) ;
\draw    (48.34,4.77) -- (80.41,32.07) ;
\draw    (38.2,16.91) -- (71.75,44.59) ;
\draw  [draw opacity=0][fill={rgb, 255:red, 208; green, 208; blue, 208 }  ,fill opacity=1 ] (233.72,68.62) -- (277.39,98.87) -- (286.4,85.86) -- (242.73,55.61) -- cycle ;
\draw  [draw opacity=0][fill={rgb, 255:red, 208; green, 208; blue, 208 }  ,fill opacity=1 ] (233.72,42.61) -- (277.39,12.36) -- (286.4,25.36) -- (242.73,55.61) -- cycle ;
\draw  [fill={rgb, 255:red, 208; green, 208; blue, 208 }  ,fill opacity=1 ] (81.25,44.91) -- (226.5,44.91) -- (226.5,65.91) -- (81.25,65.91) -- cycle ;
\draw  [fill={rgb, 255:red, 80; green, 227; blue, 194 }  ,fill opacity=1 ] (201,55) .. controls (201,41.19) and (212.19,30) .. (226,30) .. controls (239.81,30) and (251,41.19) .. (251,55) .. controls (251,68.81) and (239.81,80) .. (226,80) .. controls (212.19,80) and (201,68.81) .. (201,55) -- cycle ;
\draw    (107.9,44.91) -- (202.9,44.91) ;
\draw    (108.7,65.71) -- (203.7,65.71) ;
\draw    (241.35,74.06) -- (277.39,98.87) ;
\draw    (250.1,61.01) -- (286.4,85.86) ;
\draw    (277.39,12.36) -- (243,36.66) ;
\draw    (286.4,25.36) -- (250.5,49.91) ;
\draw    (23.64,47.04) -- (76.75,45.59) ;
\draw    (24.07,62.85) -- (77.18,61.4) ;
\draw  [fill={rgb, 255:red, 223; green, 83; blue, 107 }  ,fill opacity=1 ] (61,57) .. controls (61,43.19) and (72.19,32) .. (86,32) .. controls (99.81,32) and (111,43.19) .. (111,57) .. controls (111,70.81) and (99.81,82) .. (86,82) .. controls (72.19,82) and (61,70.81) .. (61,57) -- cycle ;

\draw (88,85.4) node [anchor=north west][inner sep=0.75pt]    {$\textcolor[rgb]{0.87,0.33,0.42}{n}$};
\draw (214,85.4) node [anchor=north west][inner sep=0.75pt]    {$\textcolor[rgb]{0.31,0.89,0.76}{m}$};

\end{tikzpicture}

%% file: Diagrams/Con/12.tex
\tikzset{every picture/.style={line width=0.75pt}} 

\begin{tikzpicture}[x=0.45pt,y=0.45pt,yscale=-1,xscale=1]

\draw  [draw opacity=0][fill={rgb, 255:red, 208; green, 208; blue, 208 }  ,fill opacity=1 ] (136.75,49.88) -- (83.64,51.33) -- (84.07,67.14) -- (137.18,65.69) -- cycle ;
\draw  [draw opacity=0][fill={rgb, 255:red, 208; green, 208; blue, 208 }  ,fill opacity=1 ] (147.82,70.83) -- (110.87,109) -- (99.5,98) -- (136.45,59.83) -- cycle ;
\draw  [draw opacity=0][fill={rgb, 255:red, 208; green, 208; blue, 208 }  ,fill opacity=1 ] (149.12,43.12) -- (108.34,9.06) -- (98.2,21.2) -- (138.98,55.26) -- cycle ;
\draw    (141.39,77.66) -- (110.87,109) ;
\draw    (130.27,66.56) -- (99.5,98) ;
\draw    (108.34,9.06) -- (140.41,36.36) ;
\draw    (98.2,21.2) -- (131.75,48.88) ;
\draw    (83.64,51.33) -- (136.75,49.88) ;
\draw    (84.07,67.14) -- (137.18,65.69) ;
\draw  [draw opacity=0][fill={rgb, 255:red, 208; green, 208; blue, 208 }  ,fill opacity=1 ] (159.82,69.25) -- (203.49,99.5) -- (212.5,86.5) -- (168.83,56.25) -- cycle ;
\draw  [draw opacity=0][fill={rgb, 255:red, 208; green, 208; blue, 208 }  ,fill opacity=1 ] (159.82,43.24) -- (203.49,12.99) -- (212.5,26) -- (168.83,56.25) -- cycle ;
\draw    (167.45,74.7) -- (203.49,99.5) ;
\draw    (176.2,61.65) -- (212.5,86.5) ;
\draw    (203.49,12.99) -- (169.1,37.3) ;
\draw    (212.5,26) -- (176.6,50.55) ;
\draw  [fill={rgb, 255:red, 205; green, 11; blue, 188 }  ,fill opacity=1 ] (124.7,55.06) .. controls (124.7,40.51) and (136.5,28.71) .. (151.05,28.71) .. controls (165.6,28.71) and (177.4,40.51) .. (177.4,55.06) .. controls (177.4,69.62) and (165.6,81.41) .. (151.05,81.41) .. controls (136.5,81.41) and (124.7,69.62) .. (124.7,55.06) -- cycle ;

\draw (128,95) node [anchor=north west][inner sep=0.75pt]    {$\textcolor[rgb]{0.8,0.04,0.74}{n+m}$};

\end{tikzpicture}

%% file: Diagrams/Con/21.tex
\tikzset{every picture/.style={line width=0.75pt}} 

\begin{tikzpicture}[x=0.45pt,y=0.45pt,yscale=-1,xscale=1]

\draw  [draw opacity=0][fill={rgb, 255:red, 208; green, 208; blue, 208 }  ,fill opacity=1 ] (76.75,45.59) -- (23.64,47.04) -- (24.07,62.85) -- (77.18,61.4) -- cycle ;
\draw  [draw opacity=0][fill={rgb, 255:red, 208; green, 208; blue, 208 }  ,fill opacity=1 ] (87.82,66.54) -- (50.87,104.71) -- (39.5,93.71) -- (76.45,55.54) -- cycle ;
\draw  [draw opacity=0][fill={rgb, 255:red, 208; green, 208; blue, 208 }  ,fill opacity=1 ] (89.12,38.83) -- (48.34,4.77) -- (38.2,16.91) -- (78.98,50.97) -- cycle ;
\draw    (81.39,73.37) -- (50.87,104.71) ;
\draw    (70.27,62.27) -- (39.5,93.71) ;
\draw    (48.34,4.77) -- (80.41,32.07) ;
\draw    (38.2,16.91) -- (71.75,44.59) ;
\draw  [draw opacity=0][fill={rgb, 255:red, 208; green, 208; blue, 208 }  ,fill opacity=1 ] (233.72,68.62) -- (277.39,98.87) -- (286.4,85.86) -- (242.73,55.61) -- cycle ;
\draw  [draw opacity=0][fill={rgb, 255:red, 208; green, 208; blue, 208 }  ,fill opacity=1 ] (233.72,42.61) -- (277.39,12.36) -- (286.4,25.36) -- (242.73,55.61) -- cycle ;
\draw  [fill={rgb, 255:red, 208; green, 208; blue, 208 }  ,fill opacity=1 ] (81.25,44.91) -- (226.5,44.91) -- (226.5,65.91) -- (81.25,65.91) -- cycle ;
\draw  [fill={rgb, 255:red, 223; green, 83; blue, 107 }  ,fill opacity=1 ] (201,55) .. controls (201,41.19) and (212.19,30) .. (226,30) .. controls (239.81,30) and (251,41.19) .. (251,55) .. controls (251,68.81) and (239.81,80) .. (226,80) .. controls (212.19,80) and (201,68.81) .. (201,55) -- cycle ;
\draw    (107.9,44.91) -- (202.9,44.91) ;
\draw    (108.7,65.71) -- (203.7,65.71) ;
\draw    (241.35,74.06) -- (277.39,98.87) ;
\draw    (250.1,61.01) -- (286.4,85.86) ;
\draw    (277.39,12.36) -- (243,36.66) ;
\draw    (286.4,25.36) -- (250.5,49.91) ;
\draw    (23.64,47.04) -- (76.75,45.59) ;
\draw    (24.07,62.85) -- (77.18,61.4) ;
\draw  [fill={rgb, 255:red, 223; green, 83; blue, 107 }  ,fill opacity=1 ] (61,56) .. controls (61,42.19) and (72.19,31) .. (86,31) .. controls (99.81,31) and (111,42.19) .. (111,56) .. controls (111,69.81) and (99.81,81) .. (86,81) .. controls (72.19,81) and (61,69.81) .. (61,56) -- cycle ;

\draw (147,79.4) node [anchor=north west][inner sep=0.75pt]    {$\textcolor[rgb]{0.87,0.33,0.42}{n}$};

\end{tikzpicture}

%% file: Diagrams/Con/22.tex
\tikzset{every picture/.style={line width=0.75pt}} 

\begin{tikzpicture}[x=0.45pt,y=0.45pt,yscale=-1,xscale=1]

\draw  [draw opacity=0][fill={rgb, 255:red, 208; green, 208; blue, 208 }  ,fill opacity=1 ] (136.75,49.88) -- (83.64,51.33) -- (84.07,67.14) -- (137.18,65.69) -- cycle ;
\draw  [draw opacity=0][fill={rgb, 255:red, 208; green, 208; blue, 208 }  ,fill opacity=1 ] (147.82,70.83) -- (110.87,109) -- (99.5,98) -- (136.45,59.83) -- cycle ;
\draw  [draw opacity=0][fill={rgb, 255:red, 208; green, 208; blue, 208 }  ,fill opacity=1 ] (149.12,43.12) -- (108.34,9.06) -- (98.2,21.2) -- (138.98,55.26) -- cycle ;
\draw    (141.39,77.66) -- (110.87,109) ;
\draw    (130.27,66.56) -- (99.5,98) ;
\draw    (108.34,9.06) -- (140.41,36.36) ;
\draw    (98.2,21.2) -- (131.75,48.88) ;
\draw    (83.64,51.33) -- (136.75,49.88) ;
\draw    (84.07,67.14) -- (137.18,65.69) ;
\draw  [draw opacity=0][fill={rgb, 255:red, 208; green, 208; blue, 208 }  ,fill opacity=1 ] (159.82,69.25) -- (203.49,99.5) -- (212.5,86.5) -- (168.83,56.25) -- cycle ;
\draw  [draw opacity=0][fill={rgb, 255:red, 208; green, 208; blue, 208 }  ,fill opacity=1 ] (159.82,43.24) -- (203.49,12.99) -- (212.5,26) -- (168.83,56.25) -- cycle ;
\draw    (167.45,74.7) -- (203.49,99.5) ;
\draw    (176.2,61.65) -- (212.5,86.5) ;
\draw    (203.49,12.99) -- (169.1,37.3) ;
\draw    (212.5,26) -- (176.6,50.55) ;
\draw  [fill={rgb, 255:red, 223; green, 83; blue, 107 }  ,fill opacity=1 ] (124.7,55.06) .. controls (124.7,40.51) and (136.5,28.71) .. (151.05,28.71) .. controls (165.6,28.71) and (177.4,40.51) .. (177.4,55.06) .. controls (177.4,69.62) and (165.6,81.41) .. (151.05,81.41) .. controls (136.5,81.41) and (124.7,69.62) .. (124.7,55.06) -- cycle ;

\draw (128,95) node [anchor=north west][inner sep=0.75pt]    {$\textcolor[rgb]{0.87,0.33,0.42}{n+1}$};

\end{tikzpicture}

%% file: Diagrams/Con/31.tex
\tikzset{every picture/.style={line width=0.75pt}} 

\begin{tikzpicture}[x=0.45pt,y=0.45pt,yscale=-1,xscale=1]

\draw  [draw opacity=0][fill={rgb, 255:red, 208; green, 208; blue, 208 }  ,fill opacity=1 ] (136.75,49.88) -- (83.64,51.33) -- (84.07,67.14) -- (137.18,65.69) -- cycle ;
\draw  [draw opacity=0][fill={rgb, 255:red, 208; green, 208; blue, 208 }  ,fill opacity=1 ] (147.82,70.83) -- (110.87,109) -- (99.5,98) -- (136.45,59.83) -- cycle ;
\draw  [draw opacity=0][fill={rgb, 255:red, 208; green, 208; blue, 208 }  ,fill opacity=1 ] (149.12,43.12) -- (108.34,9.06) -- (98.2,21.2) -- (138.98,55.26) -- cycle ;
\draw    (141.39,77.66) -- (110.87,109) ;
\draw    (130.27,66.56) -- (99.5,98) ;
\draw    (108.34,9.06) -- (140.41,36.36) ;
\draw    (98.2,21.2) -- (131.75,48.88) ;
\draw    (83.64,51.33) -- (136.75,49.88) ;
\draw    (84.07,67.14) -- (137.18,65.69) ;
\draw  [draw opacity=0][fill={rgb, 255:red, 208; green, 208; blue, 208 }  ,fill opacity=1 ] (159.82,69.25) -- (203.49,99.5) -- (212.5,86.5) -- (168.83,56.25) -- cycle ;
\draw  [draw opacity=0][fill={rgb, 255:red, 208; green, 208; blue, 208 }  ,fill opacity=1 ] (159.82,43.24) -- (203.49,12.99) -- (212.5,26) -- (168.83,56.25) -- cycle ;
\draw    (167.45,74.7) -- (203.49,99.5) ;
\draw    (176.2,61.65) -- (212.5,86.5) ;
\draw    (203.49,12.99) -- (169.1,37.3) ;
\draw    (212.5,26) -- (176.6,50.55) ;
\draw  [fill={rgb, 255:red, 208; green, 208; blue, 208 }  ,fill opacity=1 ] (147.07,76.01) .. controls (152.49,92.27) and (164.03,103.51) .. (177.4,103.51) .. controls (195.97,103.51) and (211.03,81.82) .. (211.03,55.06) .. controls (211.03,28.31) and (195.97,6.62) .. (177.4,6.62) .. controls (165.09,6.62) and (154.33,16.14) .. (148.47,30.35) -- (159.08,39.41) .. controls (162.42,27.8) and (169.37,19.81) .. (177.4,19.81) .. controls (188.69,19.81) and (197.84,35.59) .. (197.84,55.06) .. controls (197.84,74.53) and (188.69,90.32) .. (177.4,90.32) .. controls (168.8,90.32) and (161.43,81.15) .. (158.42,68.17) -- cycle ;
\draw  [fill={rgb, 255:red, 223; green, 83; blue, 107 }  ,fill opacity=1 ] (124.7,55.06) .. controls (124.7,40.51) and (136.5,28.71) .. (151.05,28.71) .. controls (165.6,28.71) and (177.4,40.51) .. (177.4,55.06) .. controls (177.4,69.62) and (165.6,81.41) .. (151.05,81.41) .. controls (136.5,81.41) and (124.7,69.62) .. (124.7,55.06) -- cycle ;

\draw (137,92.4) node [anchor=north west][inner sep=0.75pt]    {$\textcolor[rgb]{0.87,0.33,0.42}{n}$};

\end{tikzpicture}

%% file: Diagrams/Con/32.tex
\tikzset{every picture/.style={line width=0.75pt}} 

\begin{tikzpicture}[x=0.45pt,y=0.45pt,yscale=-1,xscale=1]

\draw  [draw opacity=0][fill={rgb, 255:red, 208; green, 208; blue, 208 }  ,fill opacity=1 ] (76.75,45.59) -- (23.64,47.04) -- (24.07,62.85) -- (77.18,61.4) -- cycle ;
\draw  [draw opacity=0][fill={rgb, 255:red, 208; green, 208; blue, 208 }  ,fill opacity=1 ] (87.82,66.54) -- (50.87,104.71) -- (39.5,93.71) -- (76.45,55.54) -- cycle ;
\draw  [draw opacity=0][fill={rgb, 255:red, 208; green, 208; blue, 208 }  ,fill opacity=1 ] (89.12,38.83) -- (48.34,4.77) -- (38.2,16.91) -- (78.98,50.97) -- cycle ;
\draw    (81.39,73.37) -- (50.87,104.71) ;
\draw    (70.27,62.27) -- (39.5,93.71) ;
\draw    (48.34,4.77) -- (80.41,32.07) ;
\draw    (38.2,16.91) -- (71.75,44.59) ;
\draw  [draw opacity=0][fill={rgb, 255:red, 208; green, 208; blue, 208 }  ,fill opacity=1 ] (233.72,68.62) -- (277.39,98.87) -- (286.4,85.86) -- (242.73,55.61) -- cycle ;
\draw  [draw opacity=0][fill={rgb, 255:red, 208; green, 208; blue, 208 }  ,fill opacity=1 ] (233.72,42.61) -- (277.39,12.36) -- (286.4,25.36) -- (242.73,55.61) -- cycle ;
\draw  [fill={rgb, 255:red, 223; green, 83; blue, 107 }  ,fill opacity=1 ] (201,55) .. controls (201,41.19) and (212.19,30) .. (226,30) .. controls (239.81,30) and (251,41.19) .. (251,55) .. controls (251,68.81) and (239.81,80) .. (226,80) .. controls (212.19,80) and (201,68.81) .. (201,55) -- cycle ;
\draw    (241.35,74.06) -- (277.39,98.87) ;
\draw    (250.1,61.01) -- (286.4,85.86) ;
\draw    (277.39,12.36) -- (243,36.66) ;
\draw    (286.4,25.36) -- (250.5,49.91) ;
\draw    (23.64,47.04) -- (76.75,45.59) ;
\draw    (24.07,62.85) -- (77.18,61.4) ;
\draw  [fill={rgb, 255:red, 223; green, 83; blue, 107 }  ,fill opacity=1 ] (61,56) .. controls (61,42.19) and (72.19,31) .. (86,31) .. controls (99.81,31) and (111,42.19) .. (111,56) .. controls (111,69.81) and (99.81,81) .. (86,81) .. controls (72.19,81) and (61,69.81) .. (61,56) -- cycle ;

\draw (124,55) node [anchor=north west][inner sep=0.75pt]    {$\textcolor[rgb]{0.87,0.33,0.42}{n+1}$};

\end{tikzpicture}

%% file: Diagrams/Con/41.tex
\tikzset{every picture/.style={line width=0.75pt}} 

\begin{tikzpicture}[x=0.45pt,y=0.45pt,yscale=-1,xscale=1]

\draw  [draw opacity=0][fill={rgb, 255:red, 208; green, 208; blue, 208 }  ,fill opacity=1 ] (135.75,48.88) -- (82.64,50.33) -- (83.07,66.14) -- (136.18,64.69) -- cycle ;
\draw  [draw opacity=0][fill={rgb, 255:red, 208; green, 208; blue, 208 }  ,fill opacity=1 ] (146.82,69.83) -- (109.87,108) -- (98.5,97) -- (135.45,58.83) -- cycle ;
\draw  [draw opacity=0][fill={rgb, 255:red, 208; green, 208; blue, 208 }  ,fill opacity=1 ] (148.12,42.12) -- (107.34,8.06) -- (97.2,20.2) -- (137.98,54.26) -- cycle ;
\draw    (140.39,76.66) -- (109.87,108) ;
\draw    (129.27,65.56) -- (98.5,97) ;
\draw    (107.34,8.06) -- (139.41,35.36) ;
\draw    (97.2,20.2) -- (130.75,47.88) ;
\draw    (82.64,50.33) -- (135.75,48.88) ;
\draw    (83.07,66.14) -- (136.18,64.69) ;
\draw  [draw opacity=0][fill={rgb, 255:red, 208; green, 208; blue, 208 }  ,fill opacity=1 ] (159.82,69.25) -- (203.49,99.5) -- (212.5,86.5) -- (168.83,56.25) -- cycle ;
\draw  [draw opacity=0][fill={rgb, 255:red, 208; green, 208; blue, 208 }  ,fill opacity=1 ] (159.82,43.24) -- (203.49,12.99) -- (212.5,26) -- (168.83,56.25) -- cycle ;
\draw    (167.45,74.7) -- (203.49,99.5) ;
\draw    (176.2,61.65) -- (212.5,86.5) ;
\draw    (203.49,12.99) -- (169.1,37.3) ;
\draw    (212.5,26) -- (176.6,50.55) ;
\draw  [fill={rgb, 255:red, 208; green, 208; blue, 208 }  ,fill opacity=1 ] (176.15,9.62) .. controls (191.55,9.82) and (213,31.92) .. (204.5,56.92) .. controls (196.5,47.42) and (189,24.17) .. (176.15,22.81) .. controls (164.25,23.42) and (158.67,36.18) .. (157.25,41.92) .. controls (147.42,33.68) and (152.76,38.2) .. (147.5,33.42) .. controls (149.28,28.75) and (160.75,9.42) .. (176.15,9.62) -- cycle ;
\draw  [fill={rgb, 255:red, 208; green, 208; blue, 208 }  ,fill opacity=1 ] (176.15,104.22) .. controls (191.55,104.02) and (213,81.92) .. (204.5,56.92) .. controls (196.5,66.42) and (189,89.67) .. (176.15,91.03) .. controls (164.25,90.42) and (158.67,77.66) .. (157.25,71.92) .. controls (147.42,80.16) and (152.76,75.64) .. (147.5,80.42) .. controls (149.28,85.1) and (160.75,104.42) .. (176.15,104.22) -- cycle ;
\draw  [fill={rgb, 255:red, 223; green, 83; blue, 107 }  ,fill opacity=1 ] (124.7,55.06) .. controls (124.7,40.51) and (136.5,28.71) .. (151.05,28.71) .. controls (165.6,28.71) and (177.4,40.51) .. (177.4,55.06) .. controls (177.4,69.62) and (165.6,81.41) .. (151.05,81.41) .. controls (136.5,81.41) and (124.7,69.62) .. (124.7,55.06) -- cycle ;

\draw (137,95.4) node [anchor=north west][inner sep=0.75pt]    {$\textcolor[rgb]{0.87,0.33,0.42}{n}$};

\end{tikzpicture}

%% file: Diagrams/Con/42.tex
\tikzset{every picture/.style={line width=0.75pt}} 

\begin{tikzpicture}[x=0.45pt,y=0.45pt,yscale=-1,xscale=1]

\draw  [draw opacity=0][fill={rgb, 255:red, 208; green, 208; blue, 208 }  ,fill opacity=1 ] (126.75,47.88) -- (73.64,49.33) -- (74.07,65.14) -- (127.18,63.69) -- cycle ;
\draw  [draw opacity=0][fill={rgb, 255:red, 208; green, 208; blue, 208 }  ,fill opacity=1 ] (137.82,68.83) -- (100.87,107) -- (89.5,96) -- (126.45,57.83) -- cycle ;
\draw  [draw opacity=0][fill={rgb, 255:red, 208; green, 208; blue, 208 }  ,fill opacity=1 ] (139.12,41.12) -- (98.34,7.06) -- (88.2,19.2) -- (128.98,53.26) -- cycle ;
\draw    (131.39,75.66) -- (100.87,107) ;
\draw    (120.27,64.56) -- (89.5,96) ;
\draw    (98.34,7.06) -- (130.41,34.36) ;
\draw    (88.2,19.2) -- (121.75,46.88) ;
\draw    (73.64,49.33) -- (126.75,47.88) ;
\draw    (74.07,65.14) -- (127.18,63.69) ;
\draw  [draw opacity=0][fill={rgb, 255:red, 208; green, 208; blue, 208 }  ,fill opacity=1 ] (177.92,68.25) -- (221.59,98.5) -- (230.6,85.5) -- (186.93,55.25) -- cycle ;
\draw  [draw opacity=0][fill={rgb, 255:red, 208; green, 208; blue, 208 }  ,fill opacity=1 ] (177.92,42.24) -- (221.59,11.99) -- (230.6,25) -- (186.93,55.25) -- cycle ;
\draw    (185.55,73.7) -- (221.59,98.5) ;
\draw    (194.3,60.65) -- (230.6,85.5) ;
\draw    (221.59,11.99) -- (187.2,36.3) ;
\draw    (230.6,25) -- (194.7,49.55) ;
\draw  [fill={rgb, 255:red, 223; green, 83; blue, 107 }  ,fill opacity=1 ] (145.82,71.68) .. controls (137.73,78.91) and (125.96,77.22) .. (119.55,67.9) .. controls (113.13,58.58) and (114.49,45.17) .. (122.58,37.93) .. controls (130.68,30.7) and (142.44,32.39) .. (148.86,41.71) .. controls (156.6,52.96) and (160.48,58.59) .. (160.47,58.59) .. controls (160.48,58.59) and (155.59,62.95) .. (145.82,71.68) -- cycle ;
\draw  [fill={rgb, 255:red, 223; green, 83; blue, 107 }  ,fill opacity=1 ] (167.19,67.93) .. controls (173.54,77.15) and (185.31,78.69) .. (193.49,71.39) .. controls (201.68,64.08) and (203.16,50.68) .. (196.81,41.46) .. controls (190.46,32.24) and (178.69,30.69) .. (170.51,38) .. controls (160.64,46.82) and (155.7,51.23) .. (155.69,51.23) .. controls (155.7,51.23) and (159.53,56.8) .. (167.19,67.93) -- cycle ;

\draw (128,82.4) node [anchor=north west][inner sep=0.75pt]    {$\textcolor[rgb]{0.87,0.33,0.42}{n+1}$};

\end{tikzpicture}

%% file: Diagrams/Packaging.tex
\tikzset{every picture/.style={line width=0.75pt}} 

\begin{tikzpicture}[x=0.7pt,y=0.7pt,yscale=-1,xscale=1]

\draw  [fill={rgb, 255:red, 208; green, 208; blue, 208 }  ,fill opacity=1 ] (504.83,228.57) .. controls (508.8,240.28) and (520.29,248.75) .. (533.85,248.75) .. controls (550.7,248.75) and (564.35,235.69) .. (564.35,219.58) .. controls (564.35,203.47) and (550.7,190.41) .. (533.85,190.41) .. controls (519.13,190.41) and (506.85,200.39) .. (503.98,213.65) -- (511.19,215.08) .. controls (513.39,205.19) and (522.7,197.76) .. (533.85,197.76) .. controls (546.64,197.76) and (557.01,207.53) .. (557.01,219.58) .. controls (557.01,231.63) and (546.64,241.4) .. (533.85,241.4) .. controls (523.59,241.4) and (514.89,235.11) .. (511.85,226.4) -- cycle ;
\draw  [color={rgb, 255:red, 0; green, 0; blue, 0 }  ,draw opacity=1 ][fill={rgb, 255:red, 0; green, 0; blue, 0 }  ,fill opacity=1 ] (161.07,136.03) .. controls (161.07,133.36) and (163.23,131.2) .. (165.9,131.2) .. controls (168.57,131.2) and (170.73,133.36) .. (170.73,136.03) .. controls (170.73,138.7) and (168.57,140.87) .. (165.9,140.87) .. controls (163.23,140.87) and (161.07,138.7) .. (161.07,136.03) -- cycle ;
\draw  [color={rgb, 255:red, 0; green, 0; blue, 0 }  ,draw opacity=1 ][fill={rgb, 255:red, 0; green, 0; blue, 0 }  ,fill opacity=1 ] (251.07,136.03) .. controls (251.07,133.36) and (253.23,131.2) .. (255.9,131.2) .. controls (258.57,131.2) and (260.73,133.36) .. (260.73,136.03) .. controls (260.73,138.7) and (258.57,140.87) .. (255.9,140.87) .. controls (253.23,140.87) and (251.07,138.7) .. (251.07,136.03) -- cycle ;
\draw  [color={rgb, 255:red, 0; green, 0; blue, 0 }  ,draw opacity=1 ][fill={rgb, 255:red, 0; green, 0; blue, 0 }  ,fill opacity=1 ] (469.07,134.7) .. controls (469.07,132.03) and (471.23,129.87) .. (473.9,129.87) .. controls (476.57,129.87) and (478.73,132.03) .. (478.73,134.7) .. controls (478.73,137.37) and (476.57,139.53) .. (473.9,139.53) .. controls (471.23,139.53) and (469.07,137.37) .. (469.07,134.7) -- cycle ;
\draw  [color={rgb, 255:red, 0; green, 0; blue, 0 }  ,draw opacity=1 ][fill={rgb, 255:red, 0; green, 0; blue, 0 }  ,fill opacity=1 ] (559.07,134.7) .. controls (559.07,132.03) and (561.23,129.87) .. (563.9,129.87) .. controls (566.57,129.87) and (568.73,132.03) .. (568.73,134.7) .. controls (568.73,137.37) and (566.57,139.53) .. (563.9,139.53) .. controls (561.23,139.53) and (559.07,137.37) .. (559.07,134.7) -- cycle ;
\draw    (165.9,131.2) .. controls (181.17,111) and (241.17,110.67) .. (255.9,131.2) ;
\draw    (165.9,140.87) .. controls (181.17,161.07) and (241.17,161.4) .. (255.9,140.87) ;
\draw   (255.9,136.03) .. controls (255.9,122.23) and (267.09,111.03) .. (280.9,111.03) .. controls (294.71,111.03) and (305.9,122.23) .. (305.9,136.03) .. controls (305.9,149.84) and (294.71,161.03) .. (280.9,161.03) .. controls (267.09,161.03) and (255.9,149.84) .. (255.9,136.03) -- cycle ;
\draw    (478.73,134.7) -- (559.07,134.7) ;
\draw   (423.9,134.7) .. controls (423.9,120.89) and (435.09,109.7) .. (448.9,109.7) .. controls (462.71,109.7) and (473.9,120.89) .. (473.9,134.7) .. controls (473.9,148.51) and (462.71,159.7) .. (448.9,159.7) .. controls (435.09,159.7) and (423.9,148.51) .. (423.9,134.7) -- cycle ;
\draw   (563.9,134.7) .. controls (563.9,120.89) and (575.09,109.7) .. (588.9,109.7) .. controls (602.71,109.7) and (613.9,120.89) .. (613.9,134.7) .. controls (613.9,148.51) and (602.71,159.7) .. (588.9,159.7) .. controls (575.09,159.7) and (563.9,148.51) .. (563.9,134.7) -- cycle ;
\draw  [color={rgb, 255:red, 0; green, 0; blue, 0 }  ,draw opacity=1 ][fill={rgb, 255:red, 0; green, 0; blue, 0 }  ,fill opacity=1 ] (161.07,215.03) .. controls (161.07,212.36) and (163.23,210.2) .. (165.9,210.2) .. controls (168.57,210.2) and (170.73,212.36) .. (170.73,215.03) .. controls (170.73,217.7) and (168.57,219.87) .. (165.9,219.87) .. controls (163.23,219.87) and (161.07,217.7) .. (161.07,215.03) -- cycle ;
\draw  [color={rgb, 255:red, 0; green, 0; blue, 0 }  ,draw opacity=1 ][fill={rgb, 255:red, 0; green, 0; blue, 0 }  ,fill opacity=1 ] (251.07,215.03) .. controls (251.07,212.36) and (253.23,210.2) .. (255.9,210.2) .. controls (258.57,210.2) and (260.73,212.36) .. (260.73,215.03) .. controls (260.73,217.7) and (258.57,219.87) .. (255.9,219.87) .. controls (253.23,219.87) and (251.07,217.7) .. (251.07,215.03) -- cycle ;
\draw   (255.9,215.03) .. controls (255.9,201.23) and (267.09,190.03) .. (280.9,190.03) .. controls (294.71,190.03) and (305.9,201.23) .. (305.9,215.03) .. controls (305.9,228.84) and (294.71,240.03) .. (280.9,240.03) .. controls (267.09,240.03) and (255.9,228.84) .. (255.9,215.03) -- cycle ;
\draw  [color={rgb, 255:red, 0; green, 0; blue, 0 }  ,draw opacity=1 ][fill={rgb, 255:red, 128; green, 128; blue, 128 }  ,fill opacity=1 ][line width=0.75]  (434.43,220.46) .. controls (434.43,215.39) and (438.54,211.28) .. (443.6,211.28) .. controls (448.67,211.28) and (452.78,215.39) .. (452.78,220.46) .. controls (452.78,225.52) and (448.67,229.63) .. (443.6,229.63) .. controls (438.54,229.63) and (434.43,225.52) .. (434.43,220.46) -- cycle ;
\draw  [color={rgb, 255:red, 0; green, 0; blue, 0 }  ,draw opacity=1 ][fill={rgb, 255:red, 128; green, 128; blue, 128 }  ,fill opacity=1 ][line width=0.75]  (498.41,220.9) .. controls (498.41,215.83) and (502.51,211.72) .. (507.58,211.72) .. controls (512.65,211.72) and (516.76,215.83) .. (516.76,220.9) .. controls (516.76,225.97) and (512.65,230.07) .. (507.58,230.07) .. controls (502.51,230.07) and (498.41,225.97) .. (498.41,220.9) -- cycle ;
\draw  [color={rgb, 255:red, 0; green, 0; blue, 0 }  ,draw opacity=1 ][fill={rgb, 255:red, 128; green, 128; blue, 128 }  ,fill opacity=1 ][line width=0.75]  (593.51,220.12) .. controls (593.51,215.05) and (597.62,210.94) .. (602.69,210.94) .. controls (607.75,210.94) and (611.86,215.05) .. (611.86,220.12) .. controls (611.86,225.18) and (607.75,229.29) .. (602.69,229.29) .. controls (597.62,229.29) and (593.51,225.18) .. (593.51,220.12) -- cycle ;
\draw  [fill={rgb, 255:red, 208; green, 208; blue, 208 }  ,fill opacity=1 ] (218.33,65) -- (303.33,65) -- (303.33,56.09) -- (218.33,56.09) -- cycle ;
\draw  [fill={rgb, 255:red, 208; green, 208; blue, 208 }  ,fill opacity=1 ] (211.31,67.18) .. controls (215.28,78.89) and (226.77,87.35) .. (240.34,87.35) .. controls (257.18,87.35) and (270.84,74.29) .. (270.84,58.18) .. controls (270.84,42.08) and (257.18,29.02) .. (240.34,29.02) .. controls (225.61,29.02) and (213.33,38.99) .. (210.47,52.26) -- (217.67,53.69) .. controls (219.87,43.79) and (229.18,36.36) .. (240.34,36.36) .. controls (253.13,36.36) and (263.49,46.13) .. (263.49,58.18) .. controls (263.49,70.24) and (253.13,80.01) .. (240.34,80.01) .. controls (230.07,80.01) and (221.37,73.72) .. (218.33,65) -- cycle ;
\draw  [fill={rgb, 255:red, 208; green, 208; blue, 208 }  ,fill opacity=1 ] (152.92,64.91) -- (210.33,64.91) -- (210.33,56) -- (152.92,56) -- cycle ;
\draw [color={rgb, 255:red, 126; green, 211; blue, 33 }  ,draw opacity=1 ][line width=1.5]    (156.06,65.02) -- (204.6,64.9) ;
\draw [color={rgb, 255:red, 126; green, 211; blue, 33 }  ,draw opacity=1 ][line width=1.5]    (155.4,55.7) -- (205.8,56.1) ;
\draw [color={rgb, 255:red, 205; green, 11; blue, 188 }  ,draw opacity=1 ][line width=1.5]    (263.75,55.75) -- (215.25,55.85) ;
\draw [color={rgb, 255:red, 205; green, 11; blue, 188 }  ,draw opacity=1 ][line width=1.5]    (262.25,65) -- (214.6,65) ;
\draw  [draw opacity=0][line width=1.5]  (217.35,56.16) .. controls (218.41,44.82) and (228.31,35.95) .. (240.36,35.97) .. controls (253.1,35.99) and (263.41,45.96) .. (263.39,58.23) .. controls (263.36,70.5) and (253.02,80.43) .. (240.28,80.41) .. controls (229.61,80.39) and (220.64,73.4) .. (218.03,63.92) -- (240.32,58.19) -- cycle ; \draw  [color={rgb, 255:red, 205; green, 11; blue, 188 }  ,draw opacity=1 ][line width=1.5]  (217.35,56.16) .. controls (218.41,44.82) and (228.31,35.95) .. (240.36,35.97) .. controls (253.1,35.99) and (263.41,45.96) .. (263.39,58.23) .. controls (263.36,70.5) and (253.02,80.43) .. (240.28,80.41) .. controls (229.61,80.39) and (220.64,73.4) .. (218.03,63.92) ;  
\draw  [color={rgb, 255:red, 0; green, 0; blue, 0 }  ,draw opacity=1 ][fill={rgb, 255:red, 223; green, 83; blue, 107 }  ,fill opacity=1 ][line width=0.75]  (139.43,60.96) .. controls (139.43,55.89) and (143.54,51.78) .. (148.6,51.78) .. controls (153.67,51.78) and (157.78,55.89) .. (157.78,60.96) .. controls (157.78,66.02) and (153.67,70.13) .. (148.6,70.13) .. controls (143.54,70.13) and (139.43,66.02) .. (139.43,60.96) -- cycle ;
\draw [color={rgb, 255:red, 205; green, 11; blue, 188 }  ,draw opacity=1 ][line width=1.5]    (299.75,55.75) -- (270.75,55.75) ;
\draw [color={rgb, 255:red, 205; green, 11; blue, 188 }  ,draw opacity=1 ][line width=1.5]    (300.05,65.08) -- (270,65) ;
\draw  [color={rgb, 255:red, 0; green, 0; blue, 0 }  ,draw opacity=1 ][fill={rgb, 255:red, 223; green, 83; blue, 107 }  ,fill opacity=1 ][line width=0.75]  (298.51,60.62) .. controls (298.51,55.55) and (302.62,51.44) .. (307.69,51.44) .. controls (312.75,51.44) and (316.86,55.55) .. (316.86,60.62) .. controls (316.86,65.68) and (312.75,69.79) .. (307.69,69.79) .. controls (302.62,69.79) and (298.51,65.68) .. (298.51,60.62) -- cycle ;
\draw  [draw opacity=0][line width=1.5]  (210.85,52.68) .. controls (213.55,38.6) and (226.4,28.16) .. (241.51,28.63) .. controls (258.2,29.15) and (271.32,42.82) .. (270.81,59.14) .. controls (270.29,75.47) and (256.34,88.28) .. (239.65,87.75) .. controls (226.92,87.35) and (216.26,79.31) .. (212.15,68.27) -- (240.58,58.19) -- cycle ; \draw  [color={rgb, 255:red, 126; green, 211; blue, 33 }  ,draw opacity=1 ][line width=1.5]  (210.85,52.68) .. controls (213.55,38.6) and (226.4,28.16) .. (241.51,28.63) .. controls (258.2,29.15) and (271.32,42.82) .. (270.81,59.14) .. controls (270.29,75.47) and (256.34,88.28) .. (239.65,87.75) .. controls (226.92,87.35) and (216.26,79.31) .. (212.15,68.27) ;  
\draw  [color={rgb, 255:red, 0; green, 0; blue, 0 }  ,draw opacity=1 ][fill={rgb, 255:red, 80; green, 227; blue, 194 }  ,fill opacity=1 ][line width=0.75]  (203.41,61.4) .. controls (203.41,56.33) and (207.51,52.22) .. (212.58,52.22) .. controls (217.65,52.22) and (221.76,56.33) .. (221.76,61.4) .. controls (221.76,66.47) and (217.65,70.57) .. (212.58,70.57) .. controls (207.51,70.57) and (203.41,66.47) .. (203.41,61.4) -- cycle ;
\draw  [draw opacity=0][line width=1.5]  (157.52,64.19) .. controls (156.19,67.79) and (152.7,70.35) .. (148.6,70.35) .. controls (143.36,70.35) and (139.12,66.15) .. (139.12,60.96) .. controls (139.12,55.77) and (143.36,51.56) .. (148.6,51.56) .. controls (152.15,51.56) and (155.25,53.49) .. (156.88,56.35) -- (148.6,60.96) -- cycle ; \draw  [color={rgb, 255:red, 126; green, 211; blue, 33 }  ,draw opacity=1 ][line width=1.5]  (157.52,64.19) .. controls (156.19,67.79) and (152.7,70.35) .. (148.6,70.35) .. controls (143.36,70.35) and (139.12,66.15) .. (139.12,60.96) .. controls (139.12,55.77) and (143.36,51.56) .. (148.6,51.56) .. controls (152.15,51.56) and (155.25,53.49) .. (156.88,56.35) ;  
\draw  [draw opacity=0][line width=1.5]  (298.77,63.85) .. controls (300.11,67.45) and (303.59,70.01) .. (307.69,70.01) .. controls (312.93,70.01) and (317.18,65.81) .. (317.18,60.62) .. controls (317.18,55.43) and (312.93,51.22) .. (307.69,51.22) .. controls (304.14,51.22) and (301.04,53.15) .. (299.42,56.01) -- (307.69,60.62) -- cycle ; \draw  [color={rgb, 255:red, 205; green, 11; blue, 188 }  ,draw opacity=1 ][line width=1.5]  (298.77,63.85) .. controls (300.11,67.45) and (303.59,70.01) .. (307.69,70.01) .. controls (312.93,70.01) and (317.18,65.81) .. (317.18,60.62) .. controls (317.18,55.43) and (312.93,51.22) .. (307.69,51.22) .. controls (304.14,51.22) and (301.04,53.15) .. (299.42,56.01) ;  
\draw  [draw opacity=0][line width=1.5]  (204.19,57.01) .. controls (205.71,54.16) and (208.67,52.18) .. (212.1,52.01) -- (212.58,61.4) -- cycle ; \draw  [color={rgb, 255:red, 126; green, 211; blue, 33 }  ,draw opacity=1 ][line width=1.5]  (204.19,57.01) .. controls (205.71,54.16) and (208.67,52.18) .. (212.1,52.01) ;  
\draw  [draw opacity=0][line width=1.5]  (214.61,70.58) .. controls (213.95,70.72) and (213.28,70.8) .. (212.58,70.8) .. controls (208.21,70.8) and (204.53,67.87) .. (203.43,63.89) -- (212.58,61.4) -- cycle ; \draw  [color={rgb, 255:red, 126; green, 211; blue, 33 }  ,draw opacity=1 ][line width=1.5]  (214.61,70.58) .. controls (213.95,70.72) and (213.28,70.8) .. (212.58,70.8) .. controls (208.21,70.8) and (204.53,67.87) .. (203.43,63.89) ;  
\draw  [draw opacity=0][line width=1.5]  (220.52,56.25) .. controls (219.65,54.94) and (218.47,53.86) .. (217.07,53.12) -- (212.58,61.4) -- cycle ; \draw  [color={rgb, 255:red, 205; green, 11; blue, 188 }  ,draw opacity=1 ][line width=1.5]  (220.52,56.25) .. controls (219.65,54.94) and (218.47,53.86) .. (217.07,53.12) ;  
\draw  [draw opacity=0][line width=1.5]  (219.27,68.06) .. controls (220.39,66.97) and (221.23,65.6) .. (221.68,64.06) -- (212.58,61.4) -- cycle ; \draw  [color={rgb, 255:red, 205; green, 11; blue, 188 }  ,draw opacity=1 ][line width=1.5]  (219.27,68.06) .. controls (220.39,66.97) and (221.23,65.6) .. (221.68,64.06) ;  
\draw  [fill={rgb, 255:red, 208; green, 208; blue, 208 }  ,fill opacity=1 ] (473.58,66.25) -- (558.58,66.25) -- (558.58,57.34) -- (473.58,57.34) -- cycle ;
\draw  [fill={rgb, 255:red, 208; green, 208; blue, 208 }  ,fill opacity=1 ] (474.77,69.43) .. controls (470.81,81.14) and (459.31,89.6) .. (445.75,89.6) .. controls (428.91,89.6) and (415.25,76.54) .. (415.25,60.43) .. controls (415.25,44.33) and (428.91,31.27) .. (445.75,31.27) .. controls (460.47,31.27) and (472.76,41.24) .. (475.62,54.51) -- (468.41,55.94) .. controls (466.22,46.04) and (456.9,38.61) .. (445.75,38.61) .. controls (432.96,38.61) and (422.59,48.38) .. (422.59,60.43) .. controls (422.59,72.49) and (432.96,82.26) .. (445.75,82.26) .. controls (456.01,82.26) and (464.71,75.97) .. (467.75,67.25) -- cycle ;
\draw  [fill={rgb, 255:red, 208; green, 208; blue, 208 }  ,fill opacity=1 ] (560.06,69.18) .. controls (564.03,80.89) and (575.52,89.35) .. (589.09,89.35) .. controls (605.93,89.35) and (619.59,76.29) .. (619.59,60.18) .. controls (619.59,44.08) and (605.93,31.02) .. (589.09,31.02) .. controls (574.36,31.02) and (562.08,40.99) .. (559.22,54.26) -- (566.42,55.69) .. controls (568.62,45.79) and (577.93,38.36) .. (589.09,38.36) .. controls (601.88,38.36) and (612.24,48.13) .. (612.24,60.18) .. controls (612.24,72.24) and (601.88,82.01) .. (589.09,82.01) .. controls (578.82,82.01) and (570.12,75.72) .. (567.08,67) -- cycle ;
\draw  [draw opacity=0][line width=1.5]  (559.36,54.67) .. controls (562.05,40.59) and (574.91,30.15) .. (590.02,30.62) .. controls (606.71,31.15) and (619.83,44.81) .. (619.31,61.13) .. controls (618.8,77.46) and (604.85,90.27) .. (588.16,89.75) .. controls (575.42,89.35) and (564.77,81.3) .. (560.66,70.27) -- (589.09,60.18) -- cycle ; \draw  [color={rgb, 255:red, 80; green, 227; blue, 194 }  ,draw opacity=1 ][line width=1.5]  (559.36,54.67) .. controls (562.05,40.59) and (574.91,30.15) .. (590.02,30.62) .. controls (606.71,31.15) and (619.83,44.81) .. (619.31,61.13) .. controls (618.8,77.46) and (604.85,90.27) .. (588.16,89.75) .. controls (575.42,89.35) and (564.77,81.3) .. (560.66,70.27) ;  
\draw  [draw opacity=0][line width=1.5]  (566.11,58.15) .. controls (567.18,46.81) and (577.08,37.94) .. (589.13,37.96) .. controls (601.86,37.99) and (612.18,47.95) .. (612.15,60.22) .. controls (612.13,72.5) and (601.79,82.43) .. (589.05,82.4) .. controls (578.38,82.39) and (569.41,75.39) .. (566.79,65.91) -- (589.09,60.18) -- cycle ; \draw  [color={rgb, 255:red, 223; green, 83; blue, 107 }  ,draw opacity=1 ][line width=1.5]  (566.11,58.15) .. controls (567.18,46.81) and (577.08,37.94) .. (589.13,37.96) .. controls (601.86,37.99) and (612.18,47.95) .. (612.15,60.22) .. controls (612.13,72.5) and (601.79,82.43) .. (589.05,82.4) .. controls (578.38,82.39) and (569.41,75.39) .. (566.79,65.91) ;  
\draw  [color={rgb, 255:red, 0; green, 0; blue, 0 }  ,draw opacity=1 ][fill={rgb, 255:red, 205; green, 11; blue, 188 }  ,fill opacity=1 ][line width=0.75]  (552.68,61.46) .. controls (552.68,56.39) and (556.79,52.28) .. (561.85,52.28) .. controls (566.92,52.28) and (571.03,56.39) .. (571.03,61.46) .. controls (571.03,66.52) and (566.92,70.63) .. (561.85,70.63) .. controls (556.79,70.63) and (552.68,66.52) .. (552.68,61.46) -- cycle ;
\draw  [draw opacity=0][line width=1.5]  (468.72,58.4) .. controls (467.66,47.06) and (457.75,38.19) .. (445.71,38.21) .. controls (432.97,38.24) and (422.66,48.2) .. (422.68,60.47) .. controls (422.7,72.75) and (433.05,82.68) .. (445.79,82.65) .. controls (456.46,82.64) and (465.42,75.64) .. (468.04,66.16) -- (445.75,60.43) -- cycle ; \draw  [color={rgb, 255:red, 223; green, 83; blue, 107 }  ,draw opacity=1 ][line width=1.5]  (468.72,58.4) .. controls (467.66,47.06) and (457.75,38.19) .. (445.71,38.21) .. controls (432.97,38.24) and (422.66,48.2) .. (422.68,60.47) .. controls (422.7,72.75) and (433.05,82.68) .. (445.79,82.65) .. controls (456.46,82.64) and (465.42,75.64) .. (468.04,66.16) ;  
\draw  [draw opacity=0][line width=1.5]  (475.47,54.92) .. controls (472.78,40.84) and (459.93,30.4) .. (444.82,30.87) .. controls (428.13,31.4) and (415.01,45.06) .. (415.52,61.38) .. controls (416.04,77.71) and (429.98,90.52) .. (446.68,90) .. controls (459.41,89.6) and (470.06,81.55) .. (474.17,70.52) -- (445.75,60.43) -- cycle ; \draw  [color={rgb, 255:red, 80; green, 227; blue, 194 }  ,draw opacity=1 ][line width=1.5]  (475.47,54.92) .. controls (472.78,40.84) and (459.93,30.4) .. (444.82,30.87) .. controls (428.13,31.4) and (415.01,45.06) .. (415.52,61.38) .. controls (416.04,77.71) and (429.98,90.52) .. (446.68,90) .. controls (459.41,89.6) and (470.06,81.55) .. (474.17,70.52) ;  
\draw  [color={rgb, 255:red, 0; green, 0; blue, 0 }  ,draw opacity=1 ][fill={rgb, 255:red, 126; green, 211; blue, 33 }  ,fill opacity=1 ][line width=0.75]  (463.68,61.46) .. controls (463.68,56.39) and (467.79,52.28) .. (472.85,52.28) .. controls (477.92,52.28) and (482.03,56.39) .. (482.03,61.46) .. controls (482.03,66.52) and (477.92,70.63) .. (472.85,70.63) .. controls (467.79,70.63) and (463.68,66.52) .. (463.68,61.46) -- cycle ;
\draw [color={rgb, 255:red, 80; green, 227; blue, 194 }  ,draw opacity=1 ][line width=1.5]    (481.15,57.2) -- (553.5,57.25) ;
\draw [color={rgb, 255:red, 80; green, 227; blue, 194 }  ,draw opacity=1 ][line width=1.5]    (481.15,66.7) -- (553.5,66.75) ;
\draw  [draw opacity=0][line width=1.5]  (553.09,58.13) .. controls (554.43,54.82) and (557.63,52.44) .. (561.41,52.26) -- (561.89,61.64) -- cycle ; \draw  [color={rgb, 255:red, 80; green, 227; blue, 194 }  ,draw opacity=1 ][line width=1.5]  (553.09,58.13) .. controls (554.43,54.82) and (557.63,52.44) .. (561.41,52.26) ;  
\draw  [draw opacity=0][line width=1.5]  (561.97,71.04) .. controls (561.95,71.04) and (561.92,71.04) .. (561.89,71.04) .. controls (558.13,71.04) and (554.88,68.87) .. (553.35,65.73) -- (561.89,61.64) -- cycle ; \draw  [color={rgb, 255:red, 80; green, 227; blue, 194 }  ,draw opacity=1 ][line width=1.5]  (561.97,71.04) .. controls (561.95,71.04) and (561.92,71.04) .. (561.89,71.04) .. controls (558.13,71.04) and (554.88,68.87) .. (553.35,65.73) ;  
\draw  [draw opacity=0][line width=1.5]  (481.59,57.77) .. controls (480.28,54.74) and (477.41,52.53) .. (473.97,52.13) -- (472.85,61.46) -- cycle ; \draw  [color={rgb, 255:red, 80; green, 227; blue, 194 }  ,draw opacity=1 ][line width=1.5]  (481.59,57.77) .. controls (480.28,54.74) and (477.41,52.53) .. (473.97,52.13) ;  
\draw  [draw opacity=0][line width=1.5]  (473.29,70.84) .. controls (476.69,70.69) and (479.62,68.77) .. (481.17,65.98) -- (472.85,61.46) -- cycle ; \draw  [color={rgb, 255:red, 80; green, 227; blue, 194 }  ,draw opacity=1 ][line width=1.5]  (473.29,70.84) .. controls (476.69,70.69) and (479.62,68.77) .. (481.17,65.98) ;  
\draw  [draw opacity=0][line width=1.5]  (567.5,69.22) .. controls (569.86,67.51) and (571.38,64.75) .. (571.38,61.64) .. controls (571.38,58.16) and (569.47,55.13) .. (566.64,53.5) -- (561.89,61.64) -- cycle ; \draw  [color={rgb, 255:red, 223; green, 83; blue, 107 }  ,draw opacity=1 ][line width=1.5]  (567.5,69.22) .. controls (569.86,67.51) and (571.38,64.75) .. (571.38,61.64) .. controls (571.38,58.16) and (569.47,55.13) .. (566.64,53.5) ;  
\draw  [draw opacity=0][line width=1.5]  (467.38,69.13) .. controls (464.95,67.43) and (463.37,64.63) .. (463.37,61.46) .. controls (463.37,57.94) and (465.31,54.88) .. (468.2,53.27) -- (472.85,61.46) -- cycle ; \draw  [color={rgb, 255:red, 223; green, 83; blue, 107 }  ,draw opacity=1 ][line width=1.5]  (467.38,69.13) .. controls (464.95,67.43) and (463.37,64.63) .. (463.37,61.46) .. controls (463.37,57.94) and (465.31,54.88) .. (468.2,53.27) ;  

\draw (72.5,54.4) node [anchor=north west][inner sep=0.75pt]    {$\underline{\mathbb{G}}$};
\draw (50,123.9) node [anchor=north west][inner sep=0.75pt]    {$G(\mathbb{G} ;\mathcal{V})$};
\draw (345,123.9) node [anchor=north west][inner sep=0.75pt]    {$G\left(\mathbb{G}^{*} ;\mathcal{B}\right)$};
\draw (249,151.4) node [anchor=north west][inner sep=0.75pt]    {$2$};
\draw (157,150.4) node [anchor=north west][inner sep=0.75pt]    {$0$};
\draw (475.9,142.93) node [anchor=north west][inner sep=0.75pt]    {$1$};
\draw (548.9,142.93) node [anchor=north west][inner sep=0.75pt]    {$1$};
\draw (74,207.9) node [anchor=north west][inner sep=0.75pt]    {$K$};
\draw (249,230.4) node [anchor=north west][inner sep=0.75pt]    {$2$};
\draw (160,229.4) node [anchor=north west][inner sep=0.75pt]    {$0$};
\draw (361,205.4) node [anchor=north west][inner sep=0.75pt]    {$\mathbb{G}[ K]$};
\draw (137.5,21.9) node [anchor=north west][inner sep=0.75pt]    {$(\textcolor[rgb]{0.87,0.33,0.42}{0} ,\textcolor[rgb]{0.31,0.89,0.76}{2} ,\textcolor[rgb]{0.49,0.83,0.13}{1} ,\textcolor[rgb]{0.8,0.04,0.74}{1})$};
\draw (482.75,23.4) node [anchor=north west][inner sep=0.75pt]    {$(\textcolor[rgb]{0.49,0.83,0.13}{1} ,\textcolor[rgb]{0.8,0.04,0.74}{1} ,\textcolor[rgb]{0.87,0.33,0.42}{0} ,\textcolor[rgb]{0.31,0.89,0.76}{2})$};
\draw (367.5,48.4) node [anchor=north west][inner sep=0.75pt]    {$\underline{\mathbb{G}}^{*}$};

\end{tikzpicture}

%% file: Diagrams/QTtree.tex
\tikzset{every picture/.style={line width=0.75pt}} 

\begin{tikzpicture}[x=0.65pt,y=0.65pt,yscale=-1,xscale=1]

\draw  [fill={rgb, 255:red, 208; green, 208; blue, 208 }  ,fill opacity=1 ] (499.92,307.2) -- (565.1,307.2) -- (565.1,294.1) -- (499.92,294.1) -- cycle ;
\draw  [color={rgb, 255:red, 0; green, 0; blue, 0 }  ,draw opacity=1 ][fill={rgb, 255:red, 208; green, 208; blue, 208 }  ,fill opacity=1 ][line width=0.75]  (428.06,201) .. controls (416.49,197.43) and (409.14,186.26) .. (411.34,174.87) .. controls (413.71,162.59) and (426.28,154.7) .. (439.42,157.23) .. controls (452.55,159.77) and (461.27,171.78) .. (458.9,184.06) .. controls (457.07,193.52) and (449.18,200.38) .. (439.65,201.87) -- (438.05,193.94) .. controls (444.47,193.09) and (449.78,188.67) .. (450.97,182.53) .. controls (452.49,174.63) and (446.64,166.85) .. (437.88,165.16) .. controls (429.13,163.47) and (420.8,168.5) .. (419.27,176.4) .. controls (417.85,183.75) and (422.82,190.99) .. (430.58,193.33) -- cycle ;
\draw    (257,109) -- (231.24,141.43) ;
\draw [shift={(230,143)}, rotate = 308.45] [color={rgb, 255:red, 0; green, 0; blue, 0 }  ][line width=0.75]    (10.93,-3.29) .. controls (6.95,-1.4) and (3.31,-0.3) .. (0,0) .. controls (3.31,0.3) and (6.95,1.4) .. (10.93,3.29)   ;
\draw    (390,112) -- (415.76,144.43) ;
\draw [shift={(417,146)}, rotate = 231.55] [color={rgb, 255:red, 0; green, 0; blue, 0 }  ][line width=0.75]    (10.93,-3.29) .. controls (6.95,-1.4) and (3.31,-0.3) .. (0,0) .. controls (3.31,0.3) and (6.95,1.4) .. (10.93,3.29)   ;
\draw    (400,226.5) -- (374.24,258.93) ;
\draw [shift={(373,260.5)}, rotate = 308.45] [color={rgb, 255:red, 0; green, 0; blue, 0 }  ][line width=0.75]    (10.93,-3.29) .. controls (6.95,-1.4) and (3.31,-0.3) .. (0,0) .. controls (3.31,0.3) and (6.95,1.4) .. (10.93,3.29)   ;
\draw    (480,225.5) -- (505.76,257.93) ;
\draw [shift={(507,259.5)}, rotate = 231.55] [color={rgb, 255:red, 0; green, 0; blue, 0 }  ][line width=0.75]    (10.93,-3.29) .. controls (6.95,-1.4) and (3.31,-0.3) .. (0,0) .. controls (3.31,0.3) and (6.95,1.4) .. (10.93,3.29)   ;
\draw [color={rgb, 255:red, 245; green, 199; blue, 16 }  ,draw opacity=1 ][line width=1.5]    (510.84,307.49) -- (558.63,307.49) ;
\draw  [color={rgb, 255:red, 0; green, 0; blue, 0 }  ,draw opacity=1 ][fill={rgb, 255:red, 208; green, 208; blue, 208 }  ,fill opacity=1 ][line width=0.75]  (278.56,63) .. controls (266.99,59.42) and (259.64,48.26) .. (261.84,36.86) .. controls (264.21,24.58) and (276.78,16.69) .. (289.91,19.23) .. controls (303.05,21.77) and (311.77,33.78) .. (309.39,46.05) .. controls (307.57,55.52) and (299.68,62.38) .. (290.15,63.86) -- (288.54,55.93) .. controls (294.97,55.08) and (300.28,50.67) .. (301.47,44.52) .. controls (302.99,36.62) and (297.13,28.85) .. (288.38,27.16) .. controls (279.63,25.46) and (271.29,30.5) .. (269.77,38.39) .. controls (268.35,45.74) and (273.32,52.99) .. (281.07,55.32) -- cycle ;
\draw  [color={rgb, 255:red, 0; green, 0; blue, 0 }  ,draw opacity=1 ][fill={rgb, 255:red, 208; green, 208; blue, 208 }  ,fill opacity=1 ][line width=0.75]  (366.53,61.94) .. controls (366.55,62.29) and (366.56,62.64) .. (366.56,63) .. controls (366.56,79.63) and (346.86,93.12) .. (322.56,93.12) .. controls (298.26,93.12) and (278.56,79.63) .. (278.56,63) .. controls (278.56,62.85) and (278.56,62.7) .. (278.56,62.55) -- (287.81,62.64) .. controls (287.81,62.76) and (287.8,62.88) .. (287.8,63) .. controls (287.8,74.53) and (303.36,83.88) .. (322.56,83.88) .. controls (341.75,83.88) and (357.31,74.53) .. (357.31,63) .. controls (357.31,62.72) and (357.31,62.44) .. (357.29,62.17) -- cycle ;
\draw  [color={rgb, 255:red, 0; green, 0; blue, 0 }  ,draw opacity=1 ][fill={rgb, 255:red, 208; green, 208; blue, 208 }  ,fill opacity=1 ][line width=0.75]  (278.78,56.18) .. controls (278.76,55.83) and (278.75,55.48) .. (278.75,55.13) .. controls (278.75,38.49) and (298.45,25) .. (322.75,25) .. controls (347.05,25) and (366.75,38.49) .. (366.75,55.13) .. controls (366.75,55.27) and (366.75,55.42) .. (366.75,55.57) -- (357.5,55.48) .. controls (357.5,55.36) and (357.5,55.24) .. (357.5,55.13) .. controls (357.5,43.59) and (341.94,34.25) .. (322.75,34.25) .. controls (303.56,34.25) and (288,43.59) .. (288,55.13) .. controls (288,55.4) and (288,55.68) .. (288.02,55.96) -- cycle ;
\draw  [color={rgb, 255:red, 0; green, 0; blue, 0 }  ,draw opacity=1 ][fill={rgb, 255:red, 223; green, 83; blue, 107 }  ,fill opacity=1 ][line width=0.75]  (272.45,60.92) .. controls (272.45,53.89) and (277.98,48.19) .. (284.79,48.19) .. controls (291.6,48.19) and (297.12,53.89) .. (297.12,60.92) .. controls (297.12,67.95) and (291.6,73.65) .. (284.79,73.65) .. controls (277.98,73.65) and (272.45,67.95) .. (272.45,60.92) -- cycle ;
\draw  [color={rgb, 255:red, 0; green, 0; blue, 0 }  ,draw opacity=1 ][fill={rgb, 255:red, 74; green, 144; blue, 226 }  ,fill opacity=1 ][line width=0.75]  (347.45,60.92) .. controls (347.45,53.89) and (352.98,48.19) .. (359.79,48.19) .. controls (366.6,48.19) and (372.12,53.89) .. (372.12,60.92) .. controls (372.12,67.95) and (366.6,73.65) .. (359.79,73.65) .. controls (352.98,73.65) and (347.45,67.95) .. (347.45,60.92) -- cycle ;
\draw  [draw opacity=0][line width=1.5]  (304.24,27.25) .. controls (299.81,22.25) and (293.09,19.06) .. (285.58,19.08) .. controls (272.27,19.1) and (261.49,29.14) .. (261.51,41.5) .. controls (261.53,49.58) and (266.15,56.65) .. (273.07,60.57) -- (285.62,41.46) -- cycle ; \draw  [color={rgb, 255:red, 245; green, 199; blue, 16 }  ,draw opacity=1 ][line width=1.5]  (304.24,27.25) .. controls (299.81,22.25) and (293.09,19.06) .. (285.58,19.08) .. controls (272.27,19.1) and (261.49,29.14) .. (261.51,41.5) .. controls (261.53,49.58) and (266.15,56.65) .. (273.07,60.57) ;  
\draw  [draw opacity=0][line width=1.5]  (353.96,72.08) .. controls (350.19,69.92) and (347.63,65.73) .. (347.63,60.92) .. controls (347.63,55.48) and (350.89,50.84) .. (355.48,49.02) -- (359.79,60.92) -- cycle ; \draw  [color={rgb, 255:red, 245; green, 199; blue, 16 }  ,draw opacity=1 ][line width=1.5]  (353.96,72.08) .. controls (350.19,69.92) and (347.63,65.73) .. (347.63,60.92) .. controls (347.63,55.48) and (350.89,50.84) .. (355.48,49.02) ;  
\draw  [draw opacity=0][line width=1.5]  (363.54,73.01) .. controls (368.42,71.36) and (371.95,66.57) .. (371.95,60.92) .. controls (371.95,55.97) and (369.25,51.68) .. (365.3,49.58) -- (359.79,60.92) -- cycle ; \draw  [color={rgb, 255:red, 245; green, 199; blue, 16 }  ,draw opacity=1 ][line width=1.5]  (363.54,73.01) .. controls (368.42,71.36) and (371.95,66.57) .. (371.95,60.92) .. controls (371.95,55.97) and (369.25,51.68) .. (365.3,49.58) ;  
\draw  [draw opacity=0][line width=1.5]  (291.01,71.84) .. controls (294.54,69.64) and (296.92,65.61) .. (296.95,61) -- (284.79,60.92) -- cycle ; \draw  [color={rgb, 255:red, 245; green, 199; blue, 16 }  ,draw opacity=1 ][line width=1.5]  (291.01,71.84) .. controls (294.54,69.64) and (296.92,65.61) .. (296.95,61) ;  
\draw  [draw opacity=0][line width=1.5]  (294.93,53.91) .. controls (293.55,51.72) and (291.54,50.01) .. (289.16,49.05) -- (284.79,60.92) -- cycle ; \draw  [color={rgb, 255:red, 245; green, 199; blue, 16 }  ,draw opacity=1 ][line width=1.5]  (294.93,53.91) .. controls (293.55,51.72) and (291.54,50.01) .. (289.16,49.05) ;  
\draw  [draw opacity=0][line width=1.5]  (281.48,73.15) .. controls (276.37,71.65) and (272.63,66.74) .. (272.63,60.92) .. controls (272.63,60.61) and (272.64,60.3) .. (272.66,60) -- (284.79,60.92) -- cycle ; \draw  [color={rgb, 255:red, 245; green, 199; blue, 16 }  ,draw opacity=1 ][line width=1.5]  (281.48,73.15) .. controls (276.37,71.65) and (272.63,66.74) .. (272.63,60.92) .. controls (272.63,60.61) and (272.64,60.3) .. (272.66,60) ;  
\draw  [draw opacity=0][line width=1.5]  (275.24,53.04) .. controls (276.6,51.25) and (278.4,49.85) .. (280.48,49.02) -- (284.79,60.92) -- cycle ; \draw  [color={rgb, 255:red, 245; green, 199; blue, 16 }  ,draw opacity=1 ][line width=1.5]  (275.24,53.04) .. controls (276.6,51.25) and (278.4,49.85) .. (280.48,49.02) ;  
\draw  [draw opacity=0][line width=1.5]  (296.57,61.63) .. controls (304.46,57.94) and (309.87,50.35) .. (309.85,41.59) .. controls (309.85,39.67) and (309.59,37.8) .. (309.09,36.02) -- (285.75,41.63) -- cycle ; \draw  [color={rgb, 255:red, 245; green, 199; blue, 16 }  ,draw opacity=1 ][line width=1.5]  (296.57,61.63) .. controls (304.46,57.94) and (309.87,50.35) .. (309.85,41.59) .. controls (309.85,39.67) and (309.59,37.8) .. (309.09,36.02) ;  
\draw  [draw opacity=0][line width=1.5]  (296.64,30.77) .. controls (293.73,28.4) and (289.85,26.96) .. (285.59,26.97) .. controls (276.58,26.98) and (269.29,33.48) .. (269.3,41.49) .. controls (269.31,46.3) and (271.97,50.57) .. (276.04,53.2) -- (285.62,41.46) -- cycle ; \draw  [color={rgb, 255:red, 245; green, 199; blue, 16 }  ,draw opacity=1 ][line width=1.5]  (296.64,30.77) .. controls (293.73,28.4) and (289.85,26.96) .. (285.59,26.97) .. controls (276.58,26.98) and (269.29,33.48) .. (269.3,41.49) .. controls (269.31,46.3) and (271.97,50.57) .. (276.04,53.2) ;  
\draw  [draw opacity=0][line width=1.5]  (294,54.13) .. controls (298.84,51.61) and (302.08,46.94) .. (302.07,41.6) .. controls (302.06,40.54) and (301.93,39.5) .. (301.69,38.5) -- (285.75,41.63) -- cycle ; \draw  [color={rgb, 255:red, 245; green, 199; blue, 16 }  ,draw opacity=1 ][line width=1.5]  (294,54.13) .. controls (298.84,51.61) and (302.08,46.94) .. (302.07,41.6) .. controls (302.06,40.54) and (301.93,39.5) .. (301.69,38.5) ;  
\draw  [draw opacity=0][line width=1.5]  (366.83,50.49) .. controls (361.67,35.69) and (343.91,24.8) .. (322.81,24.84) .. controls (302.21,24.88) and (284.83,35.31) .. (279.26,49.59) -- (322.87,59.39) -- cycle ; \draw  [color={rgb, 255:red, 245; green, 199; blue, 16 }  ,draw opacity=1 ][line width=1.5]  (366.83,50.49) .. controls (361.67,35.69) and (343.91,24.8) .. (322.81,24.84) .. controls (302.21,24.88) and (284.83,35.31) .. (279.26,49.59) ;  
\draw  [draw opacity=0][line width=1.5]  (364.61,73) .. controls (357.61,85.11) and (341.52,93.57) .. (322.81,93.54) .. controls (303.97,93.5) and (287.83,84.88) .. (280.95,72.61) -- (322.87,59.39) -- cycle ; \draw  [color={rgb, 255:red, 245; green, 199; blue, 16 }  ,draw opacity=1 ][line width=1.5]  (364.61,73) .. controls (357.61,85.11) and (341.52,93.57) .. (322.81,93.54) .. controls (303.97,93.5) and (287.83,84.88) .. (280.95,72.61) ;  
\draw  [draw opacity=0][line width=1.5]  (355.87,49.27) .. controls (349.94,40.07) and (337.26,33.72) .. (322.59,33.75) .. controls (307.89,33.77) and (295.21,40.2) .. (289.33,49.46) -- (322.64,61.05) -- cycle ; \draw  [color={rgb, 255:red, 245; green, 199; blue, 16 }  ,draw opacity=1 ][line width=1.5]  (355.87,49.27) .. controls (349.94,40.07) and (337.26,33.72) .. (322.59,33.75) .. controls (307.89,33.77) and (295.21,40.2) .. (289.33,49.46) ;  
\draw  [draw opacity=0][line width=1.5]  (354.05,71.45) .. controls (347.57,79.22) and (335.9,84.4) .. (322.59,84.38) .. controls (309.03,84.35) and (297.19,78.94) .. (290.82,70.89) -- (322.64,57.31) -- cycle ; \draw  [color={rgb, 255:red, 245; green, 199; blue, 16 }  ,draw opacity=1 ][line width=1.5]  (354.05,71.45) .. controls (347.57,79.22) and (335.9,84.4) .. (322.59,84.38) .. controls (309.03,84.35) and (297.19,78.94) .. (290.82,70.89) ;  
\draw  [color={rgb, 255:red, 0; green, 0; blue, 0 }  ,draw opacity=1 ][fill={rgb, 255:red, 208; green, 208; blue, 208 }  ,fill opacity=1 ][line width=0.75]  (165.56,195) .. controls (153.99,191.42) and (146.64,180.26) .. (148.84,168.86) .. controls (151.21,156.58) and (163.78,148.69) .. (176.91,151.23) .. controls (190.05,153.77) and (198.77,165.78) .. (196.39,178.05) .. controls (194.57,187.52) and (186.68,194.38) .. (177.15,195.86) -- (175.54,187.93) .. controls (181.97,187.08) and (187.28,182.67) .. (188.47,176.52) .. controls (189.99,168.62) and (184.13,160.85) .. (175.38,159.16) .. controls (166.63,157.46) and (158.29,162.5) .. (156.77,170.39) .. controls (155.35,177.74) and (160.32,184.99) .. (168.07,187.32) -- cycle ;
\draw  [color={rgb, 255:red, 0; green, 0; blue, 0 }  ,draw opacity=1 ][fill={rgb, 255:red, 208; green, 208; blue, 208 }  ,fill opacity=1 ][line width=0.75]  (165.78,188.18) .. controls (165.76,187.83) and (165.75,187.48) .. (165.75,187.13) .. controls (165.75,170.49) and (185.45,157) .. (209.75,157) .. controls (234.05,157) and (253.75,170.49) .. (253.75,187.13) .. controls (253.75,187.27) and (253.75,187.42) .. (253.75,187.57) -- (244.5,187.48) .. controls (244.5,187.36) and (244.5,187.24) .. (244.5,187.13) .. controls (244.5,175.59) and (228.94,166.25) .. (209.75,166.25) .. controls (190.56,166.25) and (175,175.59) .. (175,187.13) .. controls (175,187.4) and (175,187.68) .. (175.02,187.96) -- cycle ;
\draw  [color={rgb, 255:red, 0; green, 0; blue, 0 }  ,draw opacity=1 ][fill={rgb, 255:red, 223; green, 83; blue, 107 }  ,fill opacity=1 ][line width=0.75]  (159.45,192.92) .. controls (159.45,185.89) and (164.98,180.19) .. (171.79,180.19) .. controls (178.6,180.19) and (184.12,185.89) .. (184.12,192.92) .. controls (184.12,199.95) and (178.6,205.65) .. (171.79,205.65) .. controls (164.98,205.65) and (159.45,199.95) .. (159.45,192.92) -- cycle ;
\draw  [color={rgb, 255:red, 0; green, 0; blue, 0 }  ,draw opacity=1 ][fill={rgb, 255:red, 74; green, 144; blue, 226 }  ,fill opacity=1 ][line width=0.75]  (234.45,192.92) .. controls (234.45,185.89) and (239.98,180.19) .. (246.79,180.19) .. controls (253.6,180.19) and (259.12,185.89) .. (259.12,192.92) .. controls (259.12,199.95) and (253.6,205.65) .. (246.79,205.65) .. controls (239.98,205.65) and (234.45,199.95) .. (234.45,192.92) -- cycle ;
\draw  [draw opacity=0][line width=1.5]  (191.24,159.25) .. controls (186.81,154.25) and (180.09,151.06) .. (172.58,151.08) .. controls (159.27,151.1) and (148.49,161.14) .. (148.51,173.5) .. controls (148.53,181.58) and (153.15,188.65) .. (160.07,192.57) -- (172.62,173.46) -- cycle ; \draw  [color={rgb, 255:red, 245; green, 199; blue, 16 }  ,draw opacity=1 ][line width=1.5]  (191.24,159.25) .. controls (186.81,154.25) and (180.09,151.06) .. (172.58,151.08) .. controls (159.27,151.1) and (148.49,161.14) .. (148.51,173.5) .. controls (148.53,181.58) and (153.15,188.65) .. (160.07,192.57) ;  
\draw  [draw opacity=0][line width=1.5]  (253.28,182.16) .. controls (256.69,184.42) and (258.95,188.39) .. (258.95,192.92) .. controls (258.95,199.94) and (253.5,205.63) .. (246.79,205.63) .. controls (240.07,205.63) and (234.63,199.94) .. (234.63,192.92) .. controls (234.63,187.48) and (237.89,182.84) .. (242.48,181.02) -- (246.79,192.92) -- cycle ; \draw  [color={rgb, 255:red, 245; green, 199; blue, 16 }  ,draw opacity=1 ][line width=1.5]  (253.28,182.16) .. controls (256.69,184.42) and (258.95,188.39) .. (258.95,192.92) .. controls (258.95,199.94) and (253.5,205.63) .. (246.79,205.63) .. controls (240.07,205.63) and (234.63,199.94) .. (234.63,192.92) .. controls (234.63,187.48) and (237.89,182.84) .. (242.48,181.02) ;  
\draw  [draw opacity=0][line width=1.5]  (181.93,185.91) .. controls (180.55,183.72) and (178.54,182.01) .. (176.16,181.05) -- (171.79,192.92) -- cycle ; \draw  [color={rgb, 255:red, 245; green, 199; blue, 16 }  ,draw opacity=1 ][line width=1.5]  (181.93,185.91) .. controls (180.55,183.72) and (178.54,182.01) .. (176.16,181.05) ;  
\draw  [draw opacity=0][line width=1.5]  (183.95,193.16) .. controls (183.82,200.07) and (178.42,205.63) .. (171.79,205.63) .. controls (165.07,205.63) and (159.63,199.94) .. (159.63,192.92) .. controls (159.63,192.61) and (159.64,192.3) .. (159.66,192) -- (171.79,192.92) -- cycle ; \draw  [color={rgb, 255:red, 245; green, 199; blue, 16 }  ,draw opacity=1 ][line width=1.5]  (183.95,193.16) .. controls (183.82,200.07) and (178.42,205.63) .. (171.79,205.63) .. controls (165.07,205.63) and (159.63,199.94) .. (159.63,192.92) .. controls (159.63,192.61) and (159.64,192.3) .. (159.66,192) ;  
\draw  [draw opacity=0][line width=1.5]  (162.24,185.04) .. controls (163.6,183.25) and (165.4,181.85) .. (167.48,181.02) -- (171.79,192.92) -- cycle ; \draw  [color={rgb, 255:red, 245; green, 199; blue, 16 }  ,draw opacity=1 ][line width=1.5]  (162.24,185.04) .. controls (163.6,183.25) and (165.4,181.85) .. (167.48,181.02) ;  
\draw  [draw opacity=0][line width=1.5]  (183.57,193.63) .. controls (191.46,189.94) and (196.87,182.35) .. (196.85,173.59) .. controls (196.85,171.67) and (196.59,169.8) .. (196.09,168.02) -- (172.75,173.63) -- cycle ; \draw  [color={rgb, 255:red, 245; green, 199; blue, 16 }  ,draw opacity=1 ][line width=1.5]  (183.57,193.63) .. controls (191.46,189.94) and (196.87,182.35) .. (196.85,173.59) .. controls (196.85,171.67) and (196.59,169.8) .. (196.09,168.02) ;  
\draw  [draw opacity=0][line width=1.5]  (183.64,162.77) .. controls (180.73,160.4) and (176.85,158.96) .. (172.59,158.97) .. controls (163.58,158.98) and (156.29,165.48) .. (156.3,173.49) .. controls (156.31,178.3) and (158.97,182.57) .. (163.04,185.2) -- (172.62,173.46) -- cycle ; \draw  [color={rgb, 255:red, 245; green, 199; blue, 16 }  ,draw opacity=1 ][line width=1.5]  (183.64,162.77) .. controls (180.73,160.4) and (176.85,158.96) .. (172.59,158.97) .. controls (163.58,158.98) and (156.29,165.48) .. (156.3,173.49) .. controls (156.31,178.3) and (158.97,182.57) .. (163.04,185.2) ;  
\draw  [draw opacity=0][line width=1.5]  (181,186.13) .. controls (185.84,183.61) and (189.08,178.94) .. (189.07,173.6) .. controls (189.06,172.54) and (188.93,171.5) .. (188.69,170.5) -- (172.75,173.63) -- cycle ; \draw  [color={rgb, 255:red, 245; green, 199; blue, 16 }  ,draw opacity=1 ][line width=1.5]  (181,186.13) .. controls (185.84,183.61) and (189.08,178.94) .. (189.07,173.6) .. controls (189.06,172.54) and (188.93,171.5) .. (188.69,170.5) ;  
\draw  [draw opacity=0][line width=1.5]  (253.83,182.49) .. controls (248.67,167.69) and (230.91,156.8) .. (209.81,156.84) .. controls (189.21,156.88) and (171.83,167.31) .. (166.26,181.59) -- (209.87,191.39) -- cycle ; \draw  [color={rgb, 255:red, 245; green, 199; blue, 16 }  ,draw opacity=1 ][line width=1.5]  (253.83,182.49) .. controls (248.67,167.69) and (230.91,156.8) .. (209.81,156.84) .. controls (189.21,156.88) and (171.83,167.31) .. (166.26,181.59) ;  
\draw  [draw opacity=0][line width=1.5]  (242.87,181.27) .. controls (236.94,172.07) and (224.26,165.72) .. (209.59,165.75) .. controls (194.89,165.77) and (182.21,172.2) .. (176.33,181.46) -- (209.64,193.05) -- cycle ; \draw  [color={rgb, 255:red, 245; green, 199; blue, 16 }  ,draw opacity=1 ][line width=1.5]  (242.87,181.27) .. controls (236.94,172.07) and (224.26,165.72) .. (209.59,165.75) .. controls (194.89,165.77) and (182.21,172.2) .. (176.33,181.46) ;  
\draw  [fill={rgb, 255:red, 208; green, 208; blue, 208 }  ,fill opacity=1 ] (432.66,204.03) .. controls (435.61,213.87) and (444.62,221.03) .. (455.27,221.03) .. controls (468.31,221.03) and (478.89,210.28) .. (478.89,197.03) .. controls (478.89,183.77) and (468.31,173.03) .. (455.27,173.03) .. controls (444.48,173.03) and (435.39,180.37) .. (432.55,190.39) -- (440.34,192.67) .. controls (442.19,186.01) and (448.17,181.13) .. (455.27,181.13) .. controls (463.84,181.13) and (470.79,188.25) .. (470.79,197.03) .. controls (470.79,205.81) and (463.84,212.92) .. (455.27,212.92) .. controls (448.26,212.92) and (442.33,208.17) .. (440.41,201.63) -- cycle ;
\draw  [color={rgb, 255:red, 0; green, 0; blue, 0 }  ,draw opacity=1 ][fill={rgb, 255:red, 205; green, 11; blue, 188 }  ,fill opacity=1 ][line width=0.75]  (420.36,197.08) .. controls (420.36,190.06) and (425.94,184.36) .. (432.82,184.36) .. controls (439.71,184.36) and (445.29,190.06) .. (445.29,197.08) .. controls (445.29,204.11) and (439.71,209.81) .. (432.82,209.81) .. controls (425.94,209.81) and (420.36,204.11) .. (420.36,197.08) -- cycle ;
\draw  [draw opacity=0][line width=1.5]  (434.76,184.62) .. controls (439.64,176.37) and (449.32,171.74) .. (459.34,173.76) .. controls (472.19,176.36) and (480.73,188.92) .. (478.41,201.82) .. controls (476.09,214.71) and (463.79,223.06) .. (450.94,220.46) .. controls (444.05,219.06) and (438.4,214.8) .. (435,209.16) -- (455.14,197.11) -- cycle ; \draw  [color={rgb, 255:red, 245; green, 199; blue, 16 }  ,draw opacity=1 ][line width=1.5]  (434.76,184.62) .. controls (439.64,176.37) and (449.32,171.74) .. (459.34,173.76) .. controls (472.19,176.36) and (480.73,188.92) .. (478.41,201.82) .. controls (476.09,214.71) and (463.79,223.06) .. (450.94,220.46) .. controls (444.05,219.06) and (438.4,214.8) .. (435,209.16) ;  
\draw  [draw opacity=0][line width=1.5]  (441.86,188.13) .. controls (445.41,182.72) and (452.21,179.98) .. (458.98,181.78) .. controls (467.53,184.07) and (472.79,192.74) .. (470.74,201.16) .. controls (468.68,209.58) and (460.09,214.55) .. (451.55,212.27) .. controls (447.62,211.22) and (444.39,208.82) .. (442.24,205.72) -- (455.27,197.03) -- cycle ; \draw  [color={rgb, 255:red, 245; green, 199; blue, 16 }  ,draw opacity=1 ][line width=1.5]  (441.86,188.13) .. controls (445.41,182.72) and (452.21,179.98) .. (458.98,181.78) .. controls (467.53,184.07) and (472.79,192.74) .. (470.74,201.16) .. controls (468.68,209.58) and (460.09,214.55) .. (451.55,212.27) .. controls (447.62,211.22) and (444.39,208.82) .. (442.24,205.72) ;  
\draw  [draw opacity=0][line width=1.5]  (458.4,173.64) .. controls (455.63,164.09) and (446.24,157.06) .. (435.08,157.08) .. controls (421.77,157.11) and (410.99,167.15) .. (411.02,179.51) .. controls (411.03,186.81) and (414.81,193.29) .. (420.65,197.37) -- (435.12,179.46) -- cycle ; \draw  [color={rgb, 255:red, 245; green, 199; blue, 16 }  ,draw opacity=1 ][line width=1.5]  (458.4,173.64) .. controls (455.63,164.09) and (446.24,157.06) .. (435.08,157.08) .. controls (421.77,157.11) and (410.99,167.15) .. (411.02,179.51) .. controls (411.03,186.81) and (414.81,193.29) .. (420.65,197.37) ;  
\draw  [draw opacity=0][line width=1.5]  (444.53,200.07) .. controls (452.37,196.98) and (458.08,190.14) .. (459.07,181.98) -- (435.12,179.46) -- cycle ; \draw  [color={rgb, 255:red, 245; green, 199; blue, 16 }  ,draw opacity=1 ][line width=1.5]  (444.53,200.07) .. controls (452.37,196.98) and (458.08,190.14) .. (459.07,181.98) ;  
\draw  [draw opacity=0][line width=1.5]  (443.97,191.63) .. controls (447.79,189.44) and (450.51,185.86) .. (451.24,181.71) -- (435.12,179.46) -- cycle ; \draw  [color={rgb, 255:red, 245; green, 199; blue, 16 }  ,draw opacity=1 ][line width=1.5]  (443.97,191.63) .. controls (447.79,189.44) and (450.51,185.86) .. (451.24,181.71) ;  
\draw  [draw opacity=0][line width=1.5]  (450.06,173.63) .. controls (447.53,168.52) and (441.78,164.96) .. (435.09,164.97) .. controls (426.08,164.99) and (418.79,171.49) .. (418.8,179.49) .. controls (418.81,182.94) and (420.17,186.1) .. (422.44,188.58) -- (435.12,179.46) -- cycle ; \draw  [color={rgb, 255:red, 245; green, 199; blue, 16 }  ,draw opacity=1 ][line width=1.5]  (450.06,173.63) .. controls (447.53,168.52) and (441.78,164.96) .. (435.09,164.97) .. controls (426.08,164.99) and (418.79,171.49) .. (418.8,179.49) .. controls (418.81,182.94) and (420.17,186.1) .. (422.44,188.58) ;  
\draw  [draw opacity=0][line width=1.5]  (445.23,198.37) .. controls (444.94,201.33) and (443.65,204) .. (441.71,206.01) -- (432.82,197.08) -- cycle ; \draw  [color={rgb, 255:red, 245; green, 199; blue, 16 }  ,draw opacity=1 ][line width=1.5]  (445.23,198.37) .. controls (444.94,201.33) and (443.65,204) .. (441.71,206.01) ;  
\draw  [draw opacity=0][line width=1.5]  (423.21,188.98) .. controls (425.5,186.16) and (428.95,184.36) .. (432.82,184.36) .. controls (433.75,184.36) and (434.66,184.46) .. (435.53,184.66) -- (432.82,197.08) -- cycle ; \draw  [color={rgb, 255:red, 245; green, 199; blue, 16 }  ,draw opacity=1 ][line width=1.5]  (423.21,188.98) .. controls (425.5,186.16) and (428.95,184.36) .. (432.82,184.36) .. controls (433.75,184.36) and (434.66,184.46) .. (435.53,184.66) ;  
\draw  [draw opacity=0][line width=1.5]  (435.81,209.44) .. controls (434.85,209.68) and (433.85,209.81) .. (432.82,209.81) .. controls (425.94,209.81) and (420.36,204.11) .. (420.36,197.08) .. controls (420.36,196.82) and (420.36,196.56) .. (420.38,196.3) -- (432.82,197.08) -- cycle ; \draw  [color={rgb, 255:red, 245; green, 199; blue, 16 }  ,draw opacity=1 ][line width=1.5]  (435.81,209.44) .. controls (434.85,209.68) and (433.85,209.81) .. (432.82,209.81) .. controls (425.94,209.81) and (420.36,204.11) .. (420.36,197.08) .. controls (420.36,196.82) and (420.36,196.56) .. (420.38,196.3) ;  
\draw  [draw opacity=0][line width=1.5]  (441.12,187.41) .. controls (442.59,188.73) and (443.76,190.41) .. (444.49,192.3) -- (432.89,196.96) -- cycle ; \draw  [color={rgb, 255:red, 245; green, 199; blue, 16 }  ,draw opacity=1 ][line width=1.5]  (441.12,187.41) .. controls (442.59,188.73) and (443.76,190.41) .. (444.49,192.3) ;  
\draw  [color={rgb, 255:red, 0; green, 0; blue, 0 }  ,draw opacity=1 ][fill={rgb, 255:red, 208; green, 208; blue, 208 }  ,fill opacity=1 ][line width=0.75]  (357.06,319.5) .. controls (345.49,315.93) and (338.14,304.76) .. (340.34,293.37) .. controls (342.71,281.09) and (355.28,273.2) .. (368.42,275.73) .. controls (381.55,278.27) and (390.27,290.28) .. (387.9,302.56) .. controls (386.07,312.02) and (378.18,318.88) .. (368.65,320.37) -- (367.05,312.44) .. controls (373.47,311.59) and (378.78,307.17) .. (379.97,301.03) .. controls (381.49,293.13) and (375.64,285.35) .. (366.88,283.66) .. controls (358.13,281.97) and (349.8,287) .. (348.27,294.9) .. controls (346.85,302.25) and (351.82,309.49) .. (359.58,311.83) -- cycle ;
\draw  [color={rgb, 255:red, 0; green, 0; blue, 0 }  ,draw opacity=1 ][fill={rgb, 255:red, 205; green, 11; blue, 188 }  ,fill opacity=1 ][line width=0.75]  (349.36,315.58) .. controls (349.36,308.56) and (354.94,302.86) .. (361.82,302.86) .. controls (368.71,302.86) and (374.29,308.56) .. (374.29,315.58) .. controls (374.29,322.61) and (368.71,328.31) .. (361.82,328.31) .. controls (354.94,328.31) and (349.36,322.61) .. (349.36,315.58) -- cycle ;
\draw  [draw opacity=0][line width=1.5]  (373.78,318.47) .. controls (382.3,315) and (388.24,307.1) .. (388.22,297.92) .. controls (388.2,285.56) and (377.39,275.56) .. (364.08,275.58) .. controls (350.77,275.61) and (339.99,285.65) .. (340.02,298.01) .. controls (340.03,305.31) and (343.81,311.79) .. (349.65,315.87) -- (364.12,297.96) -- cycle ; \draw  [color={rgb, 255:red, 245; green, 199; blue, 16 }  ,draw opacity=1 ][line width=1.5]  (373.78,318.47) .. controls (382.3,315) and (388.24,307.1) .. (388.22,297.92) .. controls (388.2,285.56) and (377.39,275.56) .. (364.08,275.58) .. controls (350.77,275.61) and (339.99,285.65) .. (340.02,298.01) .. controls (340.03,305.31) and (343.81,311.79) .. (349.65,315.87) ;  
\draw  [draw opacity=0][line width=1.5]  (373.09,310.07) .. controls (377.52,307.47) and (380.44,303) .. (380.43,297.94) .. controls (380.42,289.93) and (373.1,283.46) .. (364.09,283.47) .. controls (355.08,283.49) and (347.79,289.99) .. (347.8,297.99) .. controls (347.81,301.44) and (349.17,304.6) .. (351.44,307.08) -- (364.12,297.96) -- cycle ; \draw  [color={rgb, 255:red, 245; green, 199; blue, 16 }  ,draw opacity=1 ][line width=1.5]  (373.09,310.07) .. controls (377.52,307.47) and (380.44,303) .. (380.43,297.94) .. controls (380.42,289.93) and (373.1,283.46) .. (364.09,283.47) .. controls (355.08,283.49) and (347.79,289.99) .. (347.8,297.99) .. controls (347.81,301.44) and (349.17,304.6) .. (351.44,307.08) ;  
\draw  [draw opacity=0][line width=1.5]  (352.21,307.48) .. controls (354.5,304.66) and (357.95,302.86) .. (361.82,302.86) .. controls (366.95,302.86) and (371.36,306.02) .. (373.27,310.54) -- (361.82,315.58) -- cycle ; \draw  [color={rgb, 255:red, 245; green, 199; blue, 16 }  ,draw opacity=1 ][line width=1.5]  (352.21,307.48) .. controls (354.5,304.66) and (357.95,302.86) .. (361.82,302.86) .. controls (366.95,302.86) and (371.36,306.02) .. (373.27,310.54) ;  
\draw  [draw opacity=0][line width=1.5]  (374.08,317.94) .. controls (372.99,323.84) and (367.92,328.31) .. (361.82,328.31) .. controls (354.94,328.31) and (349.36,322.61) .. (349.36,315.58) .. controls (349.36,315.32) and (349.36,315.06) .. (349.38,314.8) -- (361.82,315.58) -- cycle ; \draw  [color={rgb, 255:red, 245; green, 199; blue, 16 }  ,draw opacity=1 ][line width=1.5]  (374.08,317.94) .. controls (372.99,323.84) and (367.92,328.31) .. (361.82,328.31) .. controls (354.94,328.31) and (349.36,322.61) .. (349.36,315.58) .. controls (349.36,315.32) and (349.36,315.06) .. (349.38,314.8) ;  
\draw  [color={rgb, 255:red, 0; green, 0; blue, 0 }  ,draw opacity=1 ][fill={rgb, 255:red, 205; green, 11; blue, 188 }  ,fill opacity=1 ][line width=0.75]  (486.95,300.92) .. controls (486.95,293.89) and (492.48,288.19) .. (499.29,288.19) .. controls (506.1,288.19) and (511.62,293.89) .. (511.62,300.92) .. controls (511.62,307.95) and (506.1,313.65) .. (499.29,313.65) .. controls (492.48,313.65) and (486.95,307.95) .. (486.95,300.92) -- cycle ;
\draw  [color={rgb, 255:red, 0; green, 0; blue, 0 }  ,draw opacity=1 ][fill={rgb, 255:red, 205; green, 11; blue, 188 }  ,fill opacity=1 ][line width=0.75]  (557.45,300.92) .. controls (557.45,293.89) and (562.98,288.19) .. (569.79,288.19) .. controls (576.6,288.19) and (582.12,293.89) .. (582.12,300.92) .. controls (582.12,307.95) and (576.6,313.65) .. (569.79,313.65) .. controls (562.98,313.65) and (557.45,307.95) .. (557.45,300.92) -- cycle ;
\draw [color={rgb, 255:red, 245; green, 199; blue, 16 }  ,draw opacity=1 ][line width=1.5]    (509.66,294.54) -- (559.25,294.5) ;
\draw  [draw opacity=0][line width=1.5]  (510.16,306.62) .. controls (508.16,310.78) and (504.04,313.63) .. (499.29,313.63) .. controls (492.57,313.63) and (487.13,307.94) .. (487.13,300.92) .. controls (487.13,293.89) and (492.57,288.2) .. (499.29,288.2) .. controls (503.87,288.2) and (507.86,290.85) .. (509.93,294.76) -- (499.29,300.92) -- cycle ; \draw  [color={rgb, 255:red, 245; green, 199; blue, 16 }  ,draw opacity=1 ][line width=1.5]  (510.16,306.62) .. controls (508.16,310.78) and (504.04,313.63) .. (499.29,313.63) .. controls (492.57,313.63) and (487.13,307.94) .. (487.13,300.92) .. controls (487.13,293.89) and (492.57,288.2) .. (499.29,288.2) .. controls (503.87,288.2) and (507.86,290.85) .. (509.93,294.76) ;  
\draw  [draw opacity=0][line width=1.5]  (559.37,294.35) .. controls (561.5,290.67) and (565.37,288.2) .. (569.79,288.2) .. controls (576.5,288.2) and (581.95,293.89) .. (581.95,300.92) .. controls (581.95,307.94) and (576.5,313.63) .. (569.79,313.63) .. controls (565.22,313.63) and (561.24,310.99) .. (559.16,307.1) -- (569.79,300.92) -- cycle ; \draw  [color={rgb, 255:red, 245; green, 199; blue, 16 }  ,draw opacity=1 ][line width=1.5]  (559.37,294.35) .. controls (561.5,290.67) and (565.37,288.2) .. (569.79,288.2) .. controls (576.5,288.2) and (581.95,293.89) .. (581.95,300.92) .. controls (581.95,307.94) and (576.5,313.63) .. (569.79,313.63) .. controls (565.22,313.63) and (561.24,310.99) .. (559.16,307.1) ;  

\draw (217,112.4) node [anchor=north west][inner sep=0.75pt]    {$x$};
\draw (360,229.9) node [anchor=north west][inner sep=0.75pt]    {$x$};
\draw (503,227.9) node [anchor=north west][inner sep=0.75pt]    {$y$};
\draw (254.75,142.5) node [anchor=north west][inner sep=0.75pt]   [align=left] {(\textcolor[rgb]{0.87,0.33,0.42}{0},\textcolor[rgb]{0.29,0.56,0.89}{0},\textcolor[rgb]{0.96,0.78,0.06}{1})};
\draw (560,262.5) node [anchor=north west][inner sep=0.75pt]   [align=left] {(\textcolor[rgb]{0.8,0.04,0.74}{1},\textcolor[rgb]{0.96,0.78,0.06}{0})};
\draw (391.5,261.5) node [anchor=north west][inner sep=0.75pt]   [align=left] {(\textcolor[rgb]{0.8,0.04,0.74}{0},\textcolor[rgb]{0.96,0.78,0.06}{1})};
\draw (358.5,4.75) node [anchor=north west][inner sep=0.75pt]   [align=left] {(\textcolor[rgb]{0.87,0.33,0.42}{0},\textcolor[rgb]{0.29,0.56,0.89}{0},\textcolor[rgb]{0.96,0.78,0.06}{0})};
\draw (247,24.9) node [anchor=north west][inner sep=0.75pt]    {$e$};
\draw (323,4.4) node [anchor=north west][inner sep=0.75pt]    {$f$};
\draw (350,90.9) node [anchor=north west][inner sep=0.75pt]    {$g$};
\draw (134,156.9) node [anchor=north west][inner sep=0.75pt]    {$e$};
\draw (210,136.4) node [anchor=north west][inner sep=0.75pt]    {$f$};
\draw (393.5,158.9) node [anchor=north west][inner sep=0.75pt]    {$e$};
\draw (485.75,181.15) node [anchor=north west][inner sep=0.75pt]    {$f$};
\draw (471.5,142) node [anchor=north west][inner sep=0.75pt]   [align=left] {(\textcolor[rgb]{0.8,0.04,0.74}{0},\textcolor[rgb]{0.96,0.78,0.06}{0})};
\draw (325.5,278.9) node [anchor=north west][inner sep=0.75pt]    {$e$};
\draw (529.5,271.9) node [anchor=north west][inner sep=0.75pt]    {$e$};

\end{tikzpicture}

%% file: quasitree.bbl
\begin{thebibliography}{10}

\bibitem{BR1}
B.~Bollob{\'a}s and O.~Riordan.
\newblock A polynomial invariant of graphs on orientable surfaces.
\newblock {\em Proc. Lond. Math. Soc. (3)}, 83(3):513--531, 2001.

\bibitem{BR2}
B.~Bollob{\'a}s and O.~Riordan.
\newblock A polynomial of graphs on surfaces.
\newblock {\em Math. Ann.}, 323(1):81--96, 2002.

\bibitem{Butler}
C.~Butler.
\newblock A quasi-tree expansion of the {Krushkal} polynomial.
\newblock {\em Adv. Appl. Math.}, 94:3--22, 2018.

\bibitem{Champan}
A.~Champanerkar, I.~Kofman, and N.~Stoltzfus.
\newblock Quasi-tree expansion for the {Bollob{\'a}s}-{Riordan}-{Tutte} polynomial.
\newblock {\em Bull. Lond. Math. Soc.}, 43(5):972--984, 2011.

\bibitem{ChmutovPD}
S.~Chmutov.
\newblock Generalized duality for graphs on surfaces and the signed {Bollob{\'a}s}-{Riordan} polynomial.
\newblock {\em J. Comb. Theory, Ser. B}, 99(3):617--638, 2009.

\bibitem{graphsonsurfaces}
J.~A. Ellis-Monaghan and I.~Moffatt.
\newblock {\em Graphs on surfaces. {Dualities}, polynomials, and knots}.
\newblock SpringerBriefs Math. New York, NY: Springer, 2013.

\bibitem{Handbook}
J.~A. Ellis-Monaghan and I.~Moffatt, editors.
\newblock {\em Handbook of the {Tutte} polynomial and related topics}.
\newblock Boca Raton, FL: CRC Press, 2022.

\bibitem{maps1}
A.~Goodall, T.~Krajewski, G.~Regts, and L.~Vena.
\newblock A {Tutte} polynomial for maps.
\newblock {\em Comb. Probab. Comput.}, 27(6):913--945, 2018.

\bibitem{maps2}
A.~Goodall, B.~Litjens, G.~Regts, and L.~Vena.
\newblock A {Tutte} polynomial for maps. {II}: {The} non-orientable case.
\newblock {\em Eur. J. Comb.}, 86:32, 2020.
\newblock Id/No 103095.

\bibitem{zbMATH01624430}
J.~L. Gross and T.~W. Tucker.
\newblock {\em Topological graph theory.}
\newblock Mineola, NY: Dover Publications, reprint of the 1987 orig. edition, 2001.

\bibitem{HM}
S.~Huggett and I.~Moffatt.
\newblock Types of embedded graphs and their {T}utte polynomials.
\newblock {\em Math. Proc. Cambridge Philos. Soc.}, 169(2):255--297, 2020.

\bibitem{KMT}
T.~Krajewski, I.~Moffatt, and A.~Tanasa.
\newblock Hopf algebras and {T}utte polynomials.
\newblock {\em Adv. in Appl. Math.}, 95:271--330, 2018.

\bibitem{Krushkal}
V.~Krushkal.
\newblock Graphs, links, and duality on surfaces.
\newblock {\em Comb. Probab. Comput.}, 20(2):267--287, 2011.

\bibitem{LV}
M.~Las~Vergnas.
\newblock On the {Tutte} polynomial of a morphism of matroids.
\newblock {\em Ann. {Discrete} {Math}.}, 8:7--20, 1980.

\bibitem{maya1}
I.~Moffatt and M.~Thompson.
\newblock Deletion-contraction and the surface {Tutte} polynomial.
\newblock {\em Eur. J. Comb.}, 118:20, 2024.
\newblock Id/No 103933.

\bibitem{thesis}
M.~Thompson.
\newblock {\em Topological Analogues of the Tutte polynomial and their Decompositions}.
\newblock PhD thesis, Royal Holloway, University of London, 2024.

\bibitem{TuttePoly}
W.~Tutte.
\newblock A contribution to the theory of chromatic polynomials.
\newblock {\em Can. J. Math.}, 6:80--91, 1954.

\bibitem{V-T}
F.~Vignes-Tourneret.
\newblock Non-orientable quasi-trees for the {Bollob{\'a}s}-{Riordan} polynomial.
\newblock {\em Eur. J. Comb.}, 32(4):510--532, 2011.

\end{thebibliography}
